\documentclass[11pt]{amsart}

\usepackage{
    amsmath,
    amsfonts,
    amssymb,
    amsthm,
    amscd,
    comment,
    enumitem,
    etoolbox,
    textcomp,   % To avoid warning in gensymb
    gensymb,    % For \degree
    mathtools,
    mathdots,
}
\usepackage[dvipsnames]{xcolor}
\usepackage[cmtip,all]{xy}
\usepackage[normalem]{ulem} % for strikethough \sout
\usepackage[T1]{fontenc}
\usepackage{dsfont}                     % For \mathbbm{1} (unit in monoidal category)
\usepackage[colorlinks=true, linkcolor=blue, citecolor=blue, urlcolor=blue, breaklinks=true]{hyperref}

\DeclareFontFamily{OT1}{pzc}{}
\DeclareFontShape{OT1}{pzc}{m}{it}{<-> s * [1.10] pzcmi7t}{}
\DeclareMathAlphabet{\mathpzc}{OT1}{pzc}{m}{it}

\usepackage{zref-clever}

\zcsetup{
  cap,                      % To capitalize names of theorems, etc.
  nameinlink=false,         % Options are true, false, single, and tsingle
  lastsep = {, and }        % Oxford comma
}

\zcRefTypeSetup{assumption}{
    Name-sg = Assumption ,
    name-sg = assumption ,
    Name-pl = Assumptions ,
    name-pl = assumptions ,
}
\zcRefTypeSetup{cor}{
    Name-sg = Corollary ,
    name-sg = corollary ,
    Name-pl = Corollaries ,
    name-pl = corollaries ,
}
\zcRefTypeSetup{defin}{
    Name-sg = Definition ,
    name-sg = definition ,
    Name-pl = Definitions ,
    name-pl = definitions ,
}
\zcRefTypeSetup{enumi}{
    Name-sg = {} ,
    name-sg = {} ,
    Name-pl = {} ,
    name-pl = {} ,
}
\zcRefTypeSetup{equation}{
    Name-sg = {} ,
    name-sg = {} ,
    Name-pl = {} ,
    name-pl = {} ,
}
\zcRefTypeSetup{eg}{
    Name-sg = Example ,
    name-sg = example ,
    Name-pl = Examples ,
    name-pl = examples ,
}
\zcRefTypeSetup{lem}{
    Name-sg = Lemma ,
    name-sg = lemma ,
    Name-pl = Lemmas ,
    name-pl = lemmas ,
}
\zcRefTypeSetup{prop}{
    Name-sg = Proposition ,
    name-sg = proposition ,
    Name-pl = Propositions ,
    name-pl = propositions ,
}
\zcRefTypeSetup{rem}{
    Name-sg = Remark ,
    name-sg = remark ,
    Name-pl = Remarks ,
    name-pl = remarks ,
}
\zcRefTypeSetup{theo}{
    Name-sg = Theorem ,
    name-sg = theorem ,
    Name-pl = Theorems ,
    name-pl = theorems ,
}

\zcsetup{
  countertype = {
    enumi=item, enumii=item, enumiii=item, enumiv=item
  },
  counterresetby = {
    enumii=enumi, enumiii=enumii, enumiv=enumiii
  }
}
\zcRefTypeSetup{item}{
    Name-sg = {} ,
    name-sg = {} ,
    Name-pl = {} ,
    name-pl = {} ,
}

\AddToHook{env/assumption/begin}{\zcsetup{countertype={theo=assumption}}}
\AddToHook{env/cor/begin}{\zcsetup{countertype={theo=cor}}}
\AddToHook{env/defin/begin}{\zcsetup{countertype={theo=defin}}}
\AddToHook{env/eg/begin}{\zcsetup{countertype={theo=eg}}}
\AddToHook{env/lem/begin}{\zcsetup{countertype={theo=lem}}}
\AddToHook{env/prop/begin}{\zcsetup{countertype={theo=prop}}}
\AddToHook{env/rem/begin}{\zcsetup{countertype={theo=rem}}}

\newcommand{\cref}[1]{\zcref{#1}}
\newcommand{\Cref}[1]{\zcref[S]{#1}}

\newcommand\C{\mathbb{C}}
\newcommand\N{\mathbb{N}}

\newcommand\Z{\mathbb{Z}}
\newcommand\kk{\Bbbk}
\newcommand\one{\mathds{1}}

\newcommand\bA{\mathbf{A}}

\newcommand\bi{\mathbf{i}}

\newcommand\bu{\bullet}

\newcommand\cC{\mathcal{C}}
\newcommand\cD{\mathcal{D}}

\newcommand\cM{\mathcal{M}}
\newcommand\cN{\mathcal{N}}

\newcommand\op{\mathrm{op}}

\newcommand\rev{\textup{rev}}
\newcommand\triv{\textup{triv}}

\newcommand\md{\textup{-mod}}
\newcommand\tmod{\textup{-tmod}}

\newcommand\Vset{I_V}                                   % Index set for basis of V

\newcommand{\fg}{\mathfrak{g}}
\newcommand{\fh}{\mathfrak{h}}

\newcommand{\ft}{\mathfrak{t}}
\newcommand{\fgl}{\mathfrak{gl}}

\newcommand{\osp}{\mathfrak{osp}}

\newcommand{\OSp}{\mathrm{OSp}}

\newcommand{\va}{\varsigma}
\newcommand{\tva}{\tilde{\va}}
\newcommand{\Ualg}{\mathrm{U}}
\newcommand{\Ui}{\Ualg^\imath}
\newcommand{\UiA}[1][A]{\Ui_{#1}}
\newcommand{\Uis}{\Ui_\sigma}
\newcommand{\UisA}[1][A]{\Ualg_{\sigma,#1}^\imath}

\newcommand{\Us}{\Ualg_\sigma}

\newcommand\Brauer{\mathpzc{Brauer}}
\newcommand\cEnd{\mathpzc{End}}     % Category of endofunctors
\newcommand\iqB{\mathpzc{B}}        % iquantum Brauer category
\newcommand\OS{\mathpzc{OS}}        % Framed HOMFLYPT skein category (oriented skein category)
\newcommand\DS{\mathpzc{DS}}        % Disoriented skein category

\newcommand{\upobj}{{\mathord{\uparrow}}}
\newcommand{\downobj}{{\mathord{\downarrow}}}
\newcommand{\Bobj}{\mathsf{B}}

\newcommand\bF{\mathbf{F}}
\newcommand\bG{\mathbf{G}}
\newcommand\bH{\mathbf{H}}
\newcommand\bR{\mathbf{R}}
\newcommand\bRA[1][A]{\bR^{#1}}

\newcommand\word{\langle \upobj,\downobj \rangle}
\newcommand\fm{\mathfrak{m}}

\DeclareMathOperator{\Aut}{Aut}
\DeclareMathOperator{\codom}{codom}             % codomain of a morphism
\DeclareMathOperator{\coev}{coev}
\DeclareMathOperator{\dom}{dom}                 % domain of a morphism
\DeclareMathOperator{\End}{End}
\DeclareMathOperator{\ev}{ev}

\DeclareMathOperator{\flip}{flip}
\DeclareMathOperator{\Hom}{Hom}
\DeclareMathOperator{\id}{id}
\DeclareMathOperator{\Mor}{Mor}                 % morphisms of a category

\DeclareMathOperator{\Res}{Res}                 % Restriction functor

\usepackage{tikz}
\usetikzlibrary{arrows.meta,patterns}
\usetikzlibrary{decorations.markings,decorations.pathreplacing}
\usepackage{tikz-cd}

\tikzset{multi/.style={very thick}}
\tikzset{anchorbase/.style={>=To,baseline={([yshift=-0.5ex]current bounding box.center)}}}
\tikzset{ % Syntax: \begin{tikzpicture}[centerzero={0,0.2}]
    centerzero/.style={>=To,baseline={([yshift=-0.5ex](#1))}},
    centerzero/.default={0,0}
}
\tikzset{cscale/.style={xscale=0.65,yscale=0.5}}        % scaling for diagrams
\tikzset{wipe/.style={white,line width=4pt}}

\newcommand{\strandlabel}[1]{$\scriptstyle{#1}$}
\newcommand{\botlabel}[1]{node[anchor=north] {\strandlabel{#1}}}
\newcommand{\toplabel}[1]{node[anchor=south] {\strandlabel{#1}}}
\newcommand{\braidup}{to[out=up,in=down]}
\newcommand{\braiddown}{to[out=down,in=up]}
\newcommand{\coupon}[2]{ % \coupon{position}{label}
    \draw (#1) node[inner sep=2pt,draw,fill=white,rounded corners] {$\scriptstyle{#2}$}
}

\newcommand{\opendot}[1]{
    \node[white] at (#1) {$\scriptstyle{\bullet}$};
    \node at (#1) {$\scriptstyle{\circ}$}
}
\newcommand{\bubright}[1]{% \bubright{position}
    \draw[->] (#1)+(0.2,0) arc(360:0:0.2)
}

\newcommand{\bubun}[1]{% \bubun{position}
    \draw (#1)+(0.2,0) arc(0:360:0.2)
}

\newcommand{\rightbub}{
    \begin{tikzpicture}[centerzero]
        \bubright{0,0};
    \end{tikzpicture}
}

\newcommand{\posupcross}{
    \begin{tikzpicture}[centerzero]
        \draw[->] (0.2,-0.2) -- (-0.2,0.2);
        \draw[wipe] (-0.2,-0.2) -- (0.2,0.2);
        \draw[->] (-0.2,-0.2) -- (0.2,0.2);
    \end{tikzpicture}
}
\newcommand{\poscross}{
    \begin{tikzpicture}[centerzero]
        \draw (0.2,-0.2) -- (-0.2,0.2);
        \draw[wipe] (-0.2,-0.2) -- (0.2,0.2);
        \draw (-0.2,-0.2) -- (0.2,0.2);
    \end{tikzpicture}
}
\newcommand{\negupcross}{
    \begin{tikzpicture}[centerzero]
        \draw[->] (-0.2,-0.2) -- (0.2,0.2);
        \draw[wipe] (0.2,-0.2) -- (-0.2,0.2);
        \draw[->] (0.2,-0.2) -- (-0.2,0.2);
    \end{tikzpicture}
}
\newcommand{\negcross}{
    \begin{tikzpicture}[centerzero]
        \draw (-0.2,-0.2) -- (0.2,0.2);
        \draw[wipe] (0.2,-0.2) -- (-0.2,0.2);
        \draw (0.2,-0.2) -- (-0.2,0.2);
    \end{tikzpicture}
}
\newcommand{\posrightcross}{
    \begin{tikzpicture}[centerzero]
        \draw[->] (-0.2,-0.2) -- (0.2,0.2);
        \draw[wipe] (0.2,-0.2) -- (-0.2,0.2);
        \draw[<-] (0.2,-0.2) -- (-0.2,0.2);
    \end{tikzpicture}
}
\newcommand{\negrightcross}{
    \begin{tikzpicture}[centerzero]
        \draw[<-] (0.2,-0.2) -- (-0.2,0.2);
        \draw[wipe] (-0.2,-0.2) -- (0.2,0.2);
        \draw[->] (-0.2,-0.2) -- (0.2,0.2);
    \end{tikzpicture}
}
\newcommand{\posdowncross}{
    \begin{tikzpicture}[centerzero]
        \draw[<-] (0.2,-0.2) -- (-0.2,0.2);
        \draw[wipe] (-0.2,-0.2) -- (0.2,0.2);
        \draw[<-] (-0.2,-0.2) -- (0.2,0.2);
    \end{tikzpicture}
}
\newcommand{\negdowncross}{
    \begin{tikzpicture}[centerzero]
        \draw[<-] (-0.2,-0.2) -- (0.2,0.2);
        \draw[wipe] (0.2,-0.2) -- (-0.2,0.2);
        \draw[<-] (0.2,-0.2) -- (-0.2,0.2);
    \end{tikzpicture}
}
\newcommand{\negleftcross}{
    \begin{tikzpicture}[centerzero]
        \draw[->] (0.2,-0.2) -- (-0.2,0.2);
        \draw[wipe] (-0.2,-0.2) -- (0.2,0.2);
        \draw[<-] (-0.2,-0.2) -- (0.2,0.2);
    \end{tikzpicture}
}
\newcommand{\posleftcross}{
    \begin{tikzpicture}[centerzero]
        \draw[<-] (-0.2,-0.2) -- (0.2,0.2);
        \draw[wipe] (0.2,-0.2) -- (-0.2,0.2);
        \draw[->] (0.2,-0.2) -- (-0.2,0.2);
    \end{tikzpicture}
}
\newcommand{\symcross}{
    \begin{tikzpicture}[centerzero]
        \draw (-0.2,-0.2) -- (0.2,0.2);
        \draw (0.2,-0.2) -- (-0.2,0.2);
    \end{tikzpicture}
}

\newcommand{\thickstrand}[1]{
    \begin{tikzpicture}[scale=0.5,centerzero]
        \draw[multi] (0,-0.4) \botlabel{#1} -- (0,0.4);
    \end{tikzpicture}
}

\newcommand{\thickstrandup}[1]{
    \begin{tikzpicture}[scale=0.5,centerzero]
        \draw[->,thick] (0,-0.4) \botlabel{#1} -- (0,0.4);
    \end{tikzpicture}
}

\newcommand{\upstrand}{
    \begin{tikzpicture}[centerzero]
        \draw[->] (0,-0.2) -- (0,0.2);
    \end{tikzpicture}
}
\newcommand{\strand}{
    \begin{tikzpicture}[centerzero]
        \draw[-] (0,-0.2) -- (0,0.2);
    \end{tikzpicture}
}
\newcommand{\downstrand}{
    \begin{tikzpicture}[centerzero]
        \draw[<-] (0,-0.2) -- (0,0.2);
    \end{tikzpicture}
}
\newcommand{\togupdown}{
    \begin{tikzpicture}[centerzero]
        \draw[->] (0,0) -- (0,0.2);
        \draw[->] (0,0) -- (0,-0.2);
        \opendot{0,0};
    \end{tikzpicture}
}
\newcommand{\togdownup}{
    \begin{tikzpicture}[centerzero]
        \draw[->] (0,0.2) -- (0,0);
        \draw[->] (0,-.2) -- (0,0);
        \opendot{0,0};
    \end{tikzpicture}
}

\newcommand{\rightcup}{
    \begin{tikzpicture}[anchorbase]
        \draw[->] (-0.15,0.15) -- (-0.15,0) arc(180:360:0.15) -- (0.15,0.15);
    \end{tikzpicture}
}
\newcommand{\leftcup}{
    \begin{tikzpicture}[anchorbase]
        \draw[<-] (-0.15,0.15) -- (-0.15,0) arc(180:360:0.15) -- (0.15,0.15);
    \end{tikzpicture}
}
\newcommand{\rightcap}{
    \begin{tikzpicture}[anchorbase]
        \draw[->] (-0.15,-0.15) -- (-0.15,0) arc(180:0:0.15) -- (0.15,-0.15);
    \end{tikzpicture}
}
\newcommand{\leftcap}{
    \begin{tikzpicture}[anchorbase]
        \draw[<-] (-0.15,-0.15) -- (-0.15,0) arc(180:0:0.15) -- (0.15,-0.15);
    \end{tikzpicture}
}
\newcommand{\Bcup}{
    \begin{tikzpicture}[anchorbase]
        \draw[-] (-0.15,0.15) -- (-0.15,0) arc(180:360:0.15) -- (0.15,0.15);
    \end{tikzpicture}
}
\newcommand{\Bcap}{
    \begin{tikzpicture}[anchorbase]
        \draw[-] (-0.15,-0.15) -- (-0.15,0) arc(180:0:0.15) -- (0.15,-0.15);
    \end{tikzpicture}
}

\newcommand{\Bbub}{
    \begin{tikzpicture}[centerzero]
        \bubun{0,0};
    \end{tikzpicture}
}

\newtheorem{theo}{Theorem}[section]

\newtheorem{prop}[theo]{Proposition}
\newtheorem{lem}[theo]{Lemma}
\newtheorem{cor}[theo]{Corollary}

\theoremstyle{definition}
\newtheorem{defin}[theo]{Definition}
\newtheorem{rem}[theo]{Remark}

\numberwithin{equation}{section}
\allowdisplaybreaks

\setenumerate[1]{label=(\alph*)}          % Only need to reference enumerate items with \ref{}

\newtoggle{comments}
\newtoggle{details}
\newtoggle{detailsnote}

\iftoggle{comments}{%
    \usepackage[notref,notcite]{showkeys}   % To show labels
    \newcommand{\acomments}[1]{
        \ \\
        {\color{red}
            \textbf{AS:} #1
        }
        \ \\
    }

    \newcommand{\ycomments}[1]{
        \ \\
        {\color{blue}
            \textbf{YS:} #1
        }
        \ \\
    }
    \newcommand{\question}[1]{
        \ \\
        {\color{blue}
            \textbf{Question:} #1
        }
        \ \\
    }
    }{%
        \newcommand{\acomments}[1]{}
        \newcommand{\ycomments}[1]{}
        \newcommand{\question}[1]{}
    }

\iftoggle{details}{%
    \newcommand{\details}[1]{
        \ \\
        {\color{OliveGreen}
            \textbf{Details:} #1
        }
        \\
    }
}{%
    \newcommand{\details}[1]{\ignorespaces}
}

\begin{document}
%===============

\title{The disoriented skein and \MakeLowercase{i}quantum Brauer categories}

\author{Hadi Salmasian}
\address[H.S.]{
    Department of Mathematics and Statistics \\
    University of Ottawa \\
    Ottawa, ON, K1N 6N5, Canada
}
\urladdr{\href{https://mysite.science.uottawa.ca/hsalmasi/}{mysite.science.uottawa.ca/hsalmasi/}, \textrm{\textit{ORCiD}:} \href{https://orcid.org/0000-0002-1073-7183}{orcid.org/0000-0002-1073-7183}}
\email{hadi.salmasian@uottawa.ca}

\author{Alistair Savage}
\address[A.S.]{
    Department of Mathematics and Statistics \\
    University of Ottawa \\
    Ottawa, ON, K1N 6N5, Canada
}
\urladdr{\href{https://alistairsavage.ca}{alistairsavage.ca}, \textrm{\textit{ORCiD}:} \href{https://orcid.org/0000-0002-2859-0239}{orcid.org/0000-0002-2859-0239}}
\email{alistair.savage@uottawa.ca}

\author{Yaolong Shen}
\address[Y.S.]{
    Department of Mathematics and Statistics \\
    University of Ottawa \\
    Ottawa, ON, K1N 6N5, Canada
}
\urladdr{\href{https://sites.google.com/virginia.edu/yaolongshen}{sites.google.com/virginia.edu/yaolongshen}, \textrm{\textit{ORCiD}:} \href{https://orcid.org/0000-0002-8840-3394}{orcid.org/0000-0002-8840-3394}}
\email{yshen5@uottawa.ca}

\begin{abstract}
    We develop a diagrammatic approach to the representation theory of the quantum symmetric pairs corresponding to orthosymplectic Lie superalgebras inside general linear Lie superalgebras.  Our approach is based on the disoriented skein category, which we define as a module category over the framed HOMFLYPT skein category.  The disoriented skein category admits full incarnation functors to the categories of modules over the iquantum enveloping algebras corresponding to the quantum symmetric pairs, and it can be viewed as an interpolating category for these categories of modules.  We define an equivalence of module categories between the disoriented skein category and the iquantum Brauer category (also known as the $q$-Brauer category), after endowing the latter with the structure of a module category over the framed HOMFLYPT skein category. The disoriented skein category has some advantages over the iquantum Brauer category, possessing duality structure and allowing the incarnation functors to be \emph{strict} morphisms of module categories. Finally, we construct explicit bases for the morphism spaces of the disoriented skein and iquantum Brauer categories. 
\end{abstract}



\ifboolexpr{togl{comments} or togl{details}}{%
  {\color{magenta}DETAILS OR COMMENTS ON}
}{%
}

\maketitle
\thispagestyle{empty}

\tableofcontents

%=====================
\section{Introduction}
%=====================

In this paper we study a diagrammatic approach to the theory of quantum symmetric pairs, using the string diagram calculus of monoidal categories and module categories over them.  Such diagrammatic methods have proven to be extremely valuable in representation theory, especially via their role in the categorification of quantum groups and connections to link invariants.  However, the use of these methods in the theory of quantum symmetric pairs has only just begun; see \cite{BSWW18,BWW23,BWW25}.

Associated to a Lie algebra $\fg$ and an involutive Lie algebra automorphism $\theta \colon \fg \to \fg$, one has the symmetric pair $(\fg,\fg^\theta)$.  The universal enveloping algebra $U(\fg^\theta)$ is a Hopf subalgebra of $U(\fg)$.  Thus, $U(\fg^\theta)\md$ is a monoidal category and the restriction functor $U(\fg)\md \to U(\fg^\theta)\md$ is monoidal.  Moving to the setting of quantum enveloping algebras, the situation becomes more subtle since the quantum enveloping algebra $U_q(\fg^\theta)$ is not naturally a Hopf subalgebra of $U_q(\fg)$.  A general theory of quantum symmetric pairs was developed for all finite types by Letzter in \cite{Let99} and then extended to the Kac--Moody setting by Kolb in \cite{Kol14}.  In the quantum world, $U(\fg^\theta)$ is replaced by a right coideal subalgebra of $U_q(\fg)$ denoted by $\Ui$.  In recent works extending this theory to the super setting and developing canonical basis theory, $\Ui$ is often referred to as an \emph{iquantum enveloping algebra}.  The category $\Ui\md$ is not a monoidal category, but it is a right module category over the monoidal category $U_q(\fg)$.  The iquantum enveloping algebras and their representation theory have attracted increasing interest in recent years.  It is becoming apparent that much of the theory of quantum enveloping algebras has natural analogues in the iquantum setting.  We refer the reader to the exposition \cite{Wan23} for an overview of this development.  In particular, the theory of quantum symmetric pairs has been extended to the case of Lie \emph{super}algebras in \cite{KY20,She25,SW24}.

In the current paper, we focus on the case where $\fg = \fgl(m|2n)$ is the general linear Lie superalgebra and $\fg^\theta = \osp(m|2n)$ is the orthosymplectic Lie superalgebra.  Many features of the representation theory of $\Ualg = U_q(\fg)$ are captured in the \emph{framed HOMFLYPT skein category}, or \emph{oriented skein category} for short, denoted $\OS(q,t)$.  This diagrammatic monoidal category, first introduced in \cite[\S5.2]{Tur89}, is a quotient of the category of framed oriented tangles and underpins the HOMFLYPT link invariant.  The connection to representation theory arises from the fact that there is a full monoidal functor
\[
    \bR_\OS \colon \OS(q,q^{m-2n}) \to \Ualg\md.
\]
The analogous diagrammatic category for $U_q(\osp(m|2n))$ is the \emph{Kauffman category}, first introduced in \cite[\S7.7]{Tur89}.  However, since the usual quantum enveloping algebra $U_q(\osp(m|2n))$ is not the correct starting point for quantizing the symmetric pair $(\fg,\fg^\theta)$, the Kauffman category is not well suited for the diagrammatic study of the representation theory of $\Ui\md$.

Because $\Ui\md$ is a right module category over $\Ualg\md$, it is natural to expect that a diagrammatic description of $\Ui\md$ should be given by a right module category over $\OS(q,t)$.  We restrict our attention to the category $\Ualg\tmod$ of tensor modules.  By definition, this is the full subcategory of $\Ualg$-modules whose objects are finite direct summands of tensor products of the natural module $V^+$ and its dual $V^-$.  A key observation is that the natural module and its dual become isomorphic after restriction to $\Ui$.  Thus, there is an isomorphism (\cref{varphi})
\[
    \varphi \colon V^- \xrightarrow{\cong} V^+ \qquad \text{of $\Ui$-modules}.
\]
It turns out that, in a way that we make precise in the current paper, this isomorphism determines the module category structure of $\Ui\tmod$ over $\Ualg\tmod$.

The oriented skein category $\OS(q,t)$ is generated by two objects, $\upobj$ and $\downobj$, which are the diagrammatic analogues of the natural module and its dual.  We define the \emph{disoriented skein category}, $\DS(q,t)$, to be the right module category over $\OS(q,t)$ generated by mutually inverse isomorphisms
\[
    \togupdown \colon \downobj \to \upobj,\qquad
    \togdownup \colon \upobj \to \downobj,
\]
subject to relations that can be found in \cref{DStoggles,DScurls}.  Then, in \cref{subsec:DS}, we define a morphism of module categories
\[
    \bR_\DS \colon \DS(q,q^{m-2n}) \to \Ui\tmod,\qquad
    \togupdown \mapsto \varphi.
\]
We prove (\cref{fullthm}) that the functor $\bR_\DS$ is full.

There is another diagrammatic category that has appeared in the literature in connection to the representation theory of $\Ui$.  This category was based on the \emph{$q$-Brauer algebras} used in \cite{Mol03,Wen12} to study the endomorphism algebras $\End_{\Ui}((V^+)^{\otimes r})$.  These algebras were incorporated into a \emph{$q$-Brauer category} $\iqB(q,t)$ in \cite{ST19}, and there is a full functor (\cref{fullthm})
\[
    \bR_\iqB \colon \iqB(q,q^{m-2n}) \to \Ui\tmod.
\]
In this paper we will use the term \emph{iquantum Brauer} instead of \emph{$q$-Brauer}; see \cref{terminology}.  The category $\iqB(q,t)$ was defined in \cite[Def.~7.9]{ST19} as a module category over a monoidal category version of the tower of Iwahori--Hecke algebras of type $A$.  Given its connection to representation theory, it is natural to expect that $\iqB(q,t)$ is also a right module category over $\OS(q,t)$.  In fact, the search for this structure was the original motivation for the current paper.  We describe this module category structure explicitly in \cref{sec:iqBmodstruct}.

We show in \cref{equivthm} that the categories $\DS(q,t)$ and $\iqB(q,t)$ are equivalent as right module categories over $\OS(q,t)$.  This equivalence is compatible with the functors to the categories of $\Ui$-modules.  Setting $\DS = \DS(q,q^{m-2n})$, $\OS = \OS(q,q^{m-2n})$, and $\iqB = \iqB(q,q^{m-2n})$, our main results can be summarized in the following diagram:
\begin{equation} \label{apidura}
    \begin{tikzcd}
        \DS \times \OS \arrow[rrr,"\otimes"] \arrow[dd, shift right,"\bF \times \id"'] \arrow[dr,"\bR_\DS \times \bR_\OS"] & & & \DS \arrow[shift right,dd,"\bF"'] \arrow[dl,"\bR_\DS"']
        \\
        & \Ui\tmod \times \Ualg\tmod \arrow[r,"\otimes"] & \Ui\tmod &
        \\
        \iqB \times \OS \arrow[rrr,"\otimes"] \arrow[ur,"\bR_\iqB \times \bR_\OS"']  \arrow[shift right,uu,"\bG \times \id"'] & & & \iqB \arrow[ul,"\bR_\iqB"] \arrow[shift right,uu,"\bG"']
    \end{tikzcd}
\end{equation}
The horizontal arrows labelled $\otimes$ are the right module category structures, while the functors $\bF$ and $\bG$ give the equivalence of module categories between $\DS(q,t)$ and $\OS(q,t)$.  The diagram \cref{apidura} illustrates an important advantage of the disoriented skein category over the iquantum Brauer category: the top square in \cref{apidura} commutes, while the bottom square only commutes up to natural isomorphism.  This corresponds to the fact that $\bR_\DS$ is a \emph{strict} morphism of $\OS$-modules, whereas $\bR_\iqB$ is a morphism of $\OS$-modules that is \emph{not} strict.  The $\OS(q,t)$-module structure is also simpler for $\DS(q,t)$ than it is for $\iqB(q,t)$.  Yet another benefit of $\DS(q,t)$ is that the diagrams can contain cups and caps in arbitrary positions, whereas the cups and caps in $\iqB(q,t)$ can only appear on the left side of diagrams.  The resulting duality in $\DS(q,t)$ makes it much easier to prove a basis theorem describing bases for the morphism spaces (\cref{basis}).  Using the equivalences $\bF$ and $\bG$, one can then deduce a basis theorem for $\iqB(q,t)$ (\cref{iqBbasis}).

The categories $\DS(q,t)$ and $\iqB(q,t)$ should both be thought of as interpolating categories for the categories of $\Ui$-modules, similar to the interpolating categories introduced by Deligne \cite{Del07}.   The benefits of the disoriented skein category $\DS(q,t)$ over the iquantum Brauer category $\iqB(q,t)$ arise from the fact that the former category contains generating objects corresponding to the restriction to $\Ui$ of the natural $\Ualg$-module $V^+$ \emph{and} its dual $V^-$, whereas the latter category only contains a generating object corresponding to $V^+$.  Even though $V^+$ and $V^-$ are isomorphic as $\Ui$-modules, including them both in the diagrammatics allows greater flexibility and better compatibility with the structure of a module category over $\Ualg\tmod$, where $V^+$ and $V^-$ are \emph{not} isomorphic.

We conclude this introduction with some possible directions for future research.  The most obvious open question is to describe the kernel of the functors $\bR_\DS$ and $\bR_\iqB$.  Since these functors are full, a precise description of the kernel would give a complete presentation of the categories of tensor modules.  We hope to explore this in upcoming work.

We expect that the methods developed in the current paper can be used to develop diagrammatics for other quantum symmetric pairs.  For any quantum symmetric pair $(\Ualg,\Ui)$ for which there exists a good diagrammatic calculus for the representation theory of $\Ualg$, one should be able to develop a diagrammatics calculus for the representation theory of $\Ui$ in a matter analogous to the definition of $\DS(q,t)$ from $\OS(q,t)$.

Recently, the Brauer and Kauffman categories have been extended in \cite{MS24,MS25} to incorporate the spin representation.  However, the iquantum Brauer category only captures the behaviour of tensor modules for $\Ui$, which are all obtained by restriction from $\Ualg$-modules.  It would be interesting to enlarge $\DS(q,t)$ to a module category over $\OS(q,t)$ that incorporates a larger class of modules.  In particular, there should be an iquantum analogue of the quantum spin Brauer category of \cite{MS25}. There should also be affine and cyclotomic analogues of both the disoriented skein and iquantum Brauer categories, analogous to the affine and cyclotomic versions of the oriented Brauer, oriented skein, Brauer, and Kauffman categories studied in \cite{BCNR17,Bru17,RS19,GRS22,SaWe24}.

Another diagrammatic approach to representation theory involves the theory of \emph{webs}.  In \cite{ST19}, the authors introduce categories that should be viewed as web versions of the iquantum Brauer category.  One should also be able to define a web version of the disoriented skein category by taking a partial idempotent completion at idempotents corresponding to quantum symmetrizers and antisymmetrizers.  This disoriented web category should be a module category over the web categories of \cite{CKM14}.  We expect that the relationship between this disoriented web category and the web categories of \cite{ST19} would be analogous to the relationship between the disoriented skein category and the iquantum Brauer category.

%----------------------------
\subsection*{Acknowledgments}
%----------------------------

The research of A.S.\ and H.S.\ was supported by Discovery Grants RGPIN-2023-03842 and RGPIN-2024-04030 from the Natural Sciences and Engineering Research Council of Canada.  Y.S.\ was supported by these grants and the Fields Institute for Research in Mathematical Sciences.  The authors would like to thank Linliang Song for helpful conversations.

%------------------------------------------------------
\subsection*{Changes relative to the published version}
%------------------------------------------------------

This arXiv version of the paper includes the following changes relative to the published version, to correct mistakes noticed after publication:
\begin{itemize}
    \item In \cref{horizonOS,barOS,horizonDS,barDS,barB,lem:bar}, the word ``antilinear'' was changed to ``linear'', since the maps defined there are $\kk$-linear.  The definition of the term \emph{$\kk$-antilinear} was removed from the start of \cref{sec:diagrammcats} since it is no longer used.
        
    \item Just above \cref{deer}, linearity in the second argument was added as a relation.
    
    \item In \cref{KEF}, $F_h$ was changed to $K_h$.
\end{itemize}

%========================================================
\section{Diagrammatic categories\label{sec:diagrammcats}}
%========================================================

In this section we introduce the diagrammatic categories of interest to us.  Throughout this section $\kk$ is an arbitrary commutative ring and $q,t$ are elements of $\kk^\times$ such that $t-t^{-1}$ is divisible by $q-q^{-1}$.
%We assume that we have a ring automorphism $\xi$ of $\kk$ such that $\xi(q) = q^{-1}$ and $\xi(t)=t^{-1}$.  If $\cC$ and $\cD$ are $\kk$-linear categories, we say that a functor $\bH \colon \cC \to \cD$ is \emph{$\kk$-antilinear} if $\bH(a f + b g) = \xi(a) \bH(f) + \xi(b) \bH(g)$ for all morphisms $f$ and $g$ in the same morphism space of $\cC$.

%--------------------------------------------------------
\subsection{The oriented skein category\label{subsec:OS}}
%--------------------------------------------------------

The \emph{framed HOMFLYPT skein category}, or \emph{oriented skein category} for short, $\OS(q,t)$, is the $\kk$-linear strict monoidal category generated by objects $\upobj$, $\downobj$ and morphisms
\[
    \posupcross,\ \negupcross \colon \upobj \otimes \upobj \to \upobj \otimes \upobj,\quad
    \rightcup \colon \one \to \downobj \otimes \upobj,\quad
    \rightcap \colon \upobj \otimes \downobj \to \one,\quad
    \leftcup \colon \one \to \upobj \otimes \downobj,\quad
    \leftcap \colon \downobj \otimes \upobj \to \one,
\]
subject to the relations
\begin{gather} \label{obraid}
    \begin{tikzpicture}[centerzero]
        \draw[->] (0.2,-0.4) to[out=135,in=down] (-0.15,0) to[out=up,in=225] (0.2,0.4);
        \draw[wipe] (-0.2,-0.4) to[out=45,in=down] (0.15,0) to[out=up,in=-45] (-0.2,0.4);
        \draw[->] (-0.2,-0.4) to[out=45,in=down] (0.15,0) to[out=up,in=-45] (-0.2,0.4);
    \end{tikzpicture}
    \ =\
    \begin{tikzpicture}[centerzero]
        \draw[->] (-0.2,-0.4) -- (-0.2,0.4);
        \draw[->] (0.2,-0.4) -- (0.2,0.4);
    \end{tikzpicture}
    \ =\
    \begin{tikzpicture}[centerzero]
        \draw[->] (-0.2,-0.4) to[out=45,in=down] (0.15,0) to[out=up,in=-45] (-0.2,0.4);
        \draw[wipe] (0.2,-0.4) to[out=135,in=down] (-0.15,0) to[out=up,in=225] (0.2,0.4);
        \draw[->] (0.2,-0.4) to[out=135,in=down] (-0.15,0) to[out=up,in=225] (0.2,0.4);
    \end{tikzpicture}
    \ ,\quad
    \begin{tikzpicture}[centerzero]
        \draw[->] (0.4,-0.4) -- (-0.4,0.4);
        \draw[wipe] (0,-0.4) to[out=135,in=down] (-0.32,0) to[out=up,in=225] (0,0.4);
        \draw[->] (0,-0.4) to[out=135,in=down] (-0.32,0) to[out=up,in=225] (0,0.4);
        \draw[wipe] (-0.4,-0.4) -- (0.4,0.4);
        \draw[->] (-0.4,-0.4) -- (0.4,0.4);
    \end{tikzpicture}
    \ =\
    \begin{tikzpicture}[centerzero]
        \draw[->] (0.4,-0.4) -- (-0.4,0.4);
        \draw[wipe] (0,-0.4) to[out=45,in=down] (0.32,0) to[out=up,in=-45] (0,0.4);
        \draw[->] (0,-0.4) to[out=45,in=down] (0.32,0) to[out=up,in=-45] (0,0.4);
        \draw[wipe] (-0.4,-0.4) -- (0.4,0.4);
        \draw[->] (-0.4,-0.4) -- (0.4,0.4);
    \end{tikzpicture}
    \ ,\quad
    \begin{tikzpicture}[centerzero]
        \draw[->] (-0.2,-0.4) to[out=45,in=down] (0.15,0) to[out=up,in=-45] (-0.2,0.4);
        \draw[wipe] (0.2,-0.4) to[out=135,in=down] (-0.15,0) to[out=up,in=225] (0.2,0.4);
        \draw[<-] (0.2,-0.4) to[out=135,in=down] (-0.15,0) to[out=up,in=225] (0.2,0.4);
    \end{tikzpicture}
    \ =\
    \begin{tikzpicture}[centerzero]
        \draw[->] (-0.2,-0.4) -- (-0.2,0.4);
        \draw[<-] (0.2,-0.4) -- (0.2,0.4);
    \end{tikzpicture}
    \ ,\quad
    \begin{tikzpicture}[centerzero]
        \draw[->] (0.2,-0.4) to[out=135,in=down] (-0.15,0) to[out=up,in=225] (0.2,0.4);
     \draw[wipe] (-0.2,-0.4) to[out=45,in=down] (0.15,0) to[out=up,in=-45] (-0.2,0.4);
        \draw[<-] (-0.2,-0.4) to[out=45,in=down] (0.15,0) to[out=up,in=-45] (-0.2,0.4);
    \end{tikzpicture}
    \ =\
    \begin{tikzpicture}[centerzero]
        \draw[<-] (-0.2,-0.4) -- (-0.2,0.4);
        \draw[->] (0.2,-0.4) -- (0.2,0.4);
    \end{tikzpicture}
    \ ,\quad
    \\ \label{oskein}
    \posupcross - \negupcross = (q-q^{-1})\
    \begin{tikzpicture}[centerzero]
        \draw[->] (-0.15,-0.2) -- (-0.15,0.2);
        \draw[->] (0.15,-0.2) -- (0.15,0.2);
    \end{tikzpicture}
    \ ,
    \\ \label{ocurlbub}
    \begin{tikzpicture}[centerzero,xscale=-1]
        \draw (0,-0.4) to[out=up,in=180] (0.25,0.15) to[out=0,in=up] (0.4,0);
        \draw[wipe] (0.25,-0.15) to[out=180,in=down] (0,0.4);
        \draw[->] (0.4,0) to[out=down,in=0] (0.25,-0.15) to[out=180,in=down] (0,0.4);
    \end{tikzpicture}
    = t\
    \begin{tikzpicture}[centerzero]
        \draw[->] (0,-0.4) -- (0,0.4);
    \end{tikzpicture}
    =
    \begin{tikzpicture}[centerzero]
        \draw[->] (0.4,0) to[out=down,in=0] (0.25,-0.15) to[out=180,in=down] (0,0.4);
        \draw[wipe] (0,-0.4) to[out=up,in=180] (0.25,0.15);
        \draw (0,-0.4) to[out=up,in=180] (0.25,0.15) to[out=0,in=up] (0.4,0);
    \end{tikzpicture}
    \ ,\qquad
    \rightbub = \frac{t-t^{-1}}{q-q^{-1}} 1_\one,
    \\ \label{adjunction}
    \begin{tikzpicture}[centerzero]
        \draw[->] (-0.3,-0.4) -- (-0.3,0) arc(180:0:0.15) arc(180:360:0.15) -- (0.3,0.4);
    \end{tikzpicture}
    =
    \begin{tikzpicture}[centerzero]
        \draw[->] (0,-0.4) -- (0,0.4);
    \end{tikzpicture}
    ,\qquad
    \begin{tikzpicture}[centerzero]
        \draw[<-] (0,-0.4) -- (0,0.4);
    \end{tikzpicture}
    =
    \begin{tikzpicture}[centerzero]
        \draw[->] (-0.3,0.4) -- (-0.3,0) arc(180:360:0.15) arc(180:0:0.15) -- (0.3,-0.4);
    \end{tikzpicture}.
\end{gather}
In \cref{obraid}, we have used the following morphisms:
\begin{equation} \label{lego}
    \posrightcross :=
    \begin{tikzpicture}[centerzero,xscale=-1]
        \draw[->] (-0.2,-0.3) \braidup (0.2,0.3);
        \draw[wipe] (-0.2,0.2)  \braidup (0.2,-0.2);
        \draw[->] (0.4,0.3) -- (0.4,0.1) to[out=down,in=right] (0.2,-0.2) to[out=left,in=right] (-0.2,0.2) to[out=left,in=up] (-0.4,-0.1) -- (-0.4,-0.3);
    \end{tikzpicture}
    \ ,\qquad
    \negleftcross :=
    \begin{tikzpicture}[centerzero,xscale=-1]
        \draw[->] (0.2,-0.3) \braidup (-0.2,0.3);
        \draw[wipe] (-0.2,-0.2) to[out=right,in=left] (0.2,0.2);
        \draw[->] (-0.4,0.3) -- (-0.4,0.1) to[out=down,in=left] (-0.2,-0.2) to[out=right,in=left] (0.2,0.2) to[out=right,in=up] (0.4,-0.1) -- (0.4,-0.3);
    \end{tikzpicture}
    \ .
\end{equation}

This category was first introduced in \cite[\S5.2]{Tur89}, where it was called the \emph{Hecke category}. In \cite[Definition 2.1]{QS19} it is called the \emph{quantized oriented Brauer category}.  It was studied in depth in \cite{Bru17}.  Note that, in \cite{Bru17}, the category is denoted $\OS(z,t)$, with the denominator $q-q^{-1}$ in \cref{ocurlbub} replaced by $z$.  Since we will only be interested in the case $z=q-q^{-1}$, we make this choice at the start and use the simpler notation $\OS(q,t)$ instead of $\OS(q-q^{-1},t)$.

We define the following morphisms:
\begin{equation} \label{windmill}
    \negrightcross:=
    \begin{tikzpicture}[centerzero]
        \draw[->] (-0.4,0.3) -- (-0.4,0.1) to[out=down,in=left] (-0.2,-0.2) to[out=right,in=left] (0.2,0.2) to[out=right,in=up] (0.4,-0.1) -- (0.4,-0.3);
        \draw[wipe] (0.2,-0.3) \braidup (-0.2,0.3);
        \draw[->] (0.2,-0.3) \braidup (-0.2,0.3);
    \end{tikzpicture}
    \ ,\qquad
    \posleftcross :=
    \begin{tikzpicture}[centerzero,xscale=-1]
        \draw[->] (-0.4,0.3) -- (-0.4,0.1) to[out=down,in=left] (-0.2,-0.2) to[out=right,in=left] (0.2,0.2) to[out=right,in=up] (0.4,-0.1) -- (0.4,-0.3);
        \draw[wipe] (0.2,-0.3) \braidup (-0.2,0.3);
        \draw[->] (0.2,-0.3) \braidup (-0.2,0.3);
    \end{tikzpicture}
    \ ,\qquad
    \posdowncross:=\begin{tikzpicture}[centerzero]
        \draw[->] (-0.4,0.3) -- (-0.4,0.1) to[out=down,in=left] (-0.2,-0.2) to[out=right,in=left] (0.2,0.2) to[out=right,in=up] (0.4,-0.1) -- (0.4,-0.3);
        \draw[wipe] (0.2,-0.3) \braidup (-0.2,0.3);
        \draw[<-] (0.2,-0.3) \braidup (-0.2,0.3);
    \end{tikzpicture}
    \ ,\qquad
    \negdowncross:=\begin{tikzpicture}[centerzero]
        \draw[<-] (0.2,-0.3) \braidup (-0.2,0.3);
        \draw[wipe] (-0.4,0.3) -- (-0.4,0.1) to[out=down,in=left] (-0.2,-0.2) to[out=right,in=left] (0.2,0.2) to[out=right,in=up] (0.4,-0.1) -- (0.4,-0.3);
        \draw[->] (-0.4,0.3) -- (-0.4,0.1) to[out=down,in=left] (-0.2,-0.2) to[out=right,in=left] (0.2,0.2) to[out=right,in=up] (0.4,-0.1) -- (0.4,-0.3);
    \end{tikzpicture}
    \ .
\end{equation}
The category $\OS(q,t)$ is strict pivotal.  Furthermore, the following relations hold for any orientation of the strands:
\begin{equation} \label{venom1}
    \begin{tikzpicture}[centerzero]
        \draw (-0.2,0.3) -- (-0.2,0.1) arc(180:360:0.2) -- (0.2,0.3);
        \draw[wipe] (-0.3,-0.3) \braidup (0,0.3);
        \draw (-0.3,-0.3) \braidup (0,0.3);
    \end{tikzpicture}
    =
    \begin{tikzpicture}[centerzero]
        \draw (-0.2,0.3) -- (-0.2,0.1) arc(180:360:0.2) -- (0.2,0.3);
        \draw[wipe] (0.3,-0.3) \braidup (0,0.3);
        \draw (0.3,-0.3) \braidup (0,0.3);
    \end{tikzpicture}
    \ ,\quad
    \begin{tikzpicture}[centerzero]
        \draw (-0.2,-0.3) -- (-0.2,-0.1) arc(180:0:0.2) -- (0.2,-0.3);
        \draw[wipe] (-0.3,0.3) \braiddown (0,-0.3);
        \draw (-0.3,0.3) \braiddown (0,-0.3);
    \end{tikzpicture}
    =
    \begin{tikzpicture}[centerzero]
        \draw (-0.2,-0.3) -- (-0.2,-0.1) arc(180:0:0.2) -- (0.2,-0.3);
        \draw[wipe] (0.3,0.3) \braiddown (0,-0.3);
        \draw (0.3,0.3) \braiddown (0,-0.3);
    \end{tikzpicture}
    \ ,\quad
    \begin{tikzpicture}[centerzero]
        \draw (-0.3,-0.3) \braidup (0,0.3);
        \draw[wipe] (-0.2,0.3) -- (-0.2,0.1) arc(180:360:0.2) -- (0.2,0.3);
        \draw (-0.2,0.3) -- (-0.2,0.1) arc(180:360:0.2) -- (0.2,0.3);
    \end{tikzpicture}
    =
    \begin{tikzpicture}[centerzero]
        \draw (0.3,-0.3) \braidup (0,0.3);
        \draw[wipe] (-0.2,0.3) -- (-0.2,0.1) arc(180:360:0.2) -- (0.2,0.3);
        \draw (-0.2,0.3) -- (-0.2,0.1) arc(180:360:0.2) -- (0.2,0.3);
    \end{tikzpicture}
    \ ,\quad
    \begin{tikzpicture}[centerzero]
        \draw (-0.3,0.3) \braiddown (0,-0.3);
        \draw[wipe] (-0.2,-0.3) -- (-0.2,-0.1) arc(180:0:0.2) -- (0.2,-0.3);
        \draw (-0.2,-0.3) -- (-0.2,-0.1) arc(180:0:0.2) -- (0.2,-0.3);
    \end{tikzpicture}
    =
    \begin{tikzpicture}[centerzero]
        \draw (0.3,0.3) \braiddown (0,-0.3);
        \draw[wipe] (-0.2,-0.3) -- (-0.2,-0.1) arc(180:0:0.2) -- (0.2,-0.3);
        \draw (-0.2,-0.3) -- (-0.2,-0.1) arc(180:0:0.2) -- (0.2,-0.3);
    \end{tikzpicture}
    \ ,\quad
    \begin{tikzpicture}[centerzero]
        \draw (-0.3,-0.4) -- (-0.3,0) arc(180:0:0.15) arc(180:360:0.15) -- (0.3,0.4);
    \end{tikzpicture}
    =
    \begin{tikzpicture}[centerzero]
        \draw (0,-0.4) -- (0,0.4);
    \end{tikzpicture}
    =
    \begin{tikzpicture}[centerzero]
        \draw (-0.3,0.4) -- (-0.3,0) arc(180:360:0.15) arc(180:0:0.15) -- (0.3,-0.4);
    \end{tikzpicture}
    \ ,\end{equation}
    \begin{equation} \label{venom2}
    \begin{tikzpicture}[centerzero]
        \draw (0.2,-0.4) to[out=135,in=down] (-0.15,0) to[out=up,in=225] (0.2,0.4);
        \draw[wipe] (-0.2,-0.4) to[out=45,in=down] (0.15,0) to[out=up,in=-45] (-0.2,0.4);
        \draw (-0.2,-0.4) to[out=45,in=down] (0.15,0) to[out=up,in=-45] (-0.2,0.4);
    \end{tikzpicture}
    \ =\
    \begin{tikzpicture}[centerzero]
        \draw (-0.2,-0.4) -- (-0.2,0.4);
        \draw (0.2,-0.4) -- (0.2,0.4);
    \end{tikzpicture}
    \ =\
    \begin{tikzpicture}[centerzero]
        \draw (-0.2,-0.4) to[out=45,in=down] (0.15,0) to[out=up,in=-45] (-0.2,0.4);
        \draw[wipe] (0.2,-0.4) to[out=135,in=down] (-0.15,0) to[out=up,in=225] (0.2,0.4);
        \draw (0.2,-0.4) to[out=135,in=down] (-0.15,0) to[out=up,in=225] (0.2,0.4);
    \end{tikzpicture}
    \ ,\quad
    \begin{tikzpicture}[centerzero]
        \draw (0.4,-0.4) -- (-0.4,0.4);
        \draw[wipe] (0,-0.4) to[out=135,in=down] (-0.32,0) to[out=up,in=225] (0,0.4);
        \draw (0,-0.4) to[out=135,in=down] (-0.32,0) to[out=up,in=225] (0,0.4);
        \draw[wipe] (-0.4,-0.4) -- (0.4,0.4);
        \draw (-0.4,-0.4) -- (0.4,0.4);
    \end{tikzpicture}
    \ =\
    \begin{tikzpicture}[centerzero]
        \draw (0.4,-0.4) -- (-0.4,0.4);
        \draw[wipe] (0,-0.4) to[out=45,in=down] (0.32,0) to[out=up,in=-45] (0,0.4);
        \draw (0,-0.4) to[out=45,in=down] (0.32,0) to[out=up,in=-45] (0,0.4);
        \draw[wipe] (-0.4,-0.4) -- (0.4,0.4);
        \draw (-0.4,-0.4) -- (0.4,0.4);
    \end{tikzpicture}
    \ ,\quad
    \Bbub=\frac{t-t^{-1}}{q-q^{-1}}.
\end{equation}
We also have
\begin{equation} \label{ocurlinv}
    \begin{tikzpicture}[centerzero,xscale=-1]
        \draw[->] (0.4,0) to[out=down,in=0] (0.25,-0.15) to[out=180,in=down] (0,0.4);
        \draw[wipe] (0,-0.4) to[out=up,in=180] (0.25,0.15) to[out=0,in=up] (0.4,0);
        \draw (0,-0.4) to[out=up,in=180] (0.25,0.15) to[out=0,in=up] (0.4,0);
    \end{tikzpicture}
    = t^{-1}\
    \begin{tikzpicture}[centerzero]
        \draw[->] (0,-0.4) -- (0,0.4);
    \end{tikzpicture}
    =
    \begin{tikzpicture}[centerzero]
        \draw (0,-0.4) to[out=up,in=180] (0.25,0.15) to[out=0,in=up] (0.4,0);
        \draw[wipe] (0.4,0) to[out=down,in=0] (0.25,-0.15) to[out=180,in=down] (0,0.4);
        \draw[->] (0.4,0) to[out=down,in=0] (0.25,-0.15) to[out=180,in=down] (0,0.4);
    \end{tikzpicture}
    \ .
\end{equation}

It is straightforward to verify that there is a $\kk$-linear isomorphism of monoidal categories
\begin{equation} \label{horizonOS}
    \Omega_{\updownarrow} \colon \OS(q,t) \to \OS(q^{-1},t^{-1})^\op
\end{equation}
acting on objects as $\upobj \mapsto \downobj$, $\downobj \mapsto \upobj$, and sending
\begin{equation}
    \posupcross \mapsto \negdowncross\, ,\quad
    \negupcross \mapsto \posdowncross\, ,\quad
    \rightcup \mapsto \rightcap\, ,\quad
    \rightcap \mapsto \rightcup\, ,\quad
    \leftcup \mapsto \leftcap\, ,\quad
    \leftcap \mapsto \leftcup\, .
\end{equation}
Intuitively, $\Omega_{\updownarrow}$ reflects diagrams in the horizontal axis and inverts $q$ and $t$.  We also have an isomorphism of $\kk$-linear monoidal categories
\begin{equation} \label{reverseOS}
    \Theta \colon \OS(q,t) \to \OS(q,t)
\end{equation}
that reverses the orientation of all strands. Finally, we have a $\kk$-linear isomorphism of monoidal categories
\begin{equation} \label{barOS}
    \Xi \colon \OS(q,t) \to \OS(q^{-1},t^{-1}),
\end{equation}
which we call the \emph{bar involution}, acting as the identity on objects and given on morphisms by
\[
    \posupcross \mapsto \negupcross\, ,\quad
    \negupcross \mapsto \posupcross\, ,\quad
    \rightcup \mapsto \rightcup\, ,\quad
    \rightcap \mapsto \rightcap\, ,\quad
    \leftcup \mapsto \leftcup\, ,\quad
    \leftcap \mapsto \leftcap\, .
\]
Intuitively, the bar involution flips crossings in diagrams and inverts $q$ and $t$.

Let $\word$ denote the set of objects for $\OS(q,t)$. Equivalently, writing the tensor product as juxtaposition, $\word$ is the set of words generated by $\upobj$ and $\downobj$.  It will be convenient to introduce thick strands:
\[
    \begin{tikzpicture}[centerzero]
        \draw[multi] (0,-0.3) \botlabel{\lambda} -- (0,0.3);
    \end{tikzpicture}
    := 1_\lambda, \qquad \lambda \in \word.
\]
The fact that $\OS(q,t)$ is a braided monoidal category gives us a natural interpretation for crossings of such strands.  For instance, if $\lambda = \downobj \upobj$ and $\mu = \upobj \upobj \downobj$, then
\[
    \begin{tikzpicture}[centerzero]
        \draw[multi] (0.4,-0.4) \botlabel{\mu} -- (0,0.4);
        \draw[wipe] (0,-0.4) -- (0.4,0.4);
        \draw[multi] (0,-0.4) \botlabel{\lambda} -- (0.4,0.4);
    \end{tikzpicture}
    =
    \begin{tikzpicture}[centerzero]
        \draw[->] (0.8,-0.4) -- (0,0.4);
        \draw[->] (1.2,-0.4) -- (0.4,0.4);
        \draw[<-] (1.6,-0.4) -- (0.8,0.4);
        \draw[wipe] (0,-0.4) -- (1.2,0.4);
        \draw[<-] (0,-0.4) -- (1.2,0.4);
        \draw[wipe] (0.4,-0.4) -- (1.6,0.4);
        \draw[->] (0.4,-0.4) -- (1.6,0.4);
    \end{tikzpicture}
    \ .
\]

%-----------------------------------------------------------
\subsection{The disoriented skein category\label{subsec:DS}}
%-----------------------------------------------------------

Let $(\cC,\otimes, \one)$ be a $\kk$-linear strict monoidal category, and let $\cM$ be a $\kk$-linear category.  Let $\cEnd(\cM)$ denote the strict monoidal category of $\kk$-linear endofunctors of $\cM$.  A \emph{right action} of $\cC$ on $\cM$ is a monoidal functor $\bA \colon \cC \to \cEnd(\cM)^\rev$, where $\cD^\rev$ denotes the \emph{reverse} of a monoidal category $\cD$, where we reverse the tensor product.  Given such a right action, we also say that $\cM$ is a \emph{right module category} over $\cC$, or \emph{$\cC$-module} for short; see \cite[\S 7.1]{ENGO15}.  For objects $C \in \cC$ and $M \in \cM$, we set
\begin{equation} \label{beaver}
    M \otimes C := \bA(C)(M).
\end{equation}
We say that the $\cC$-module $\cM$ is \emph{strict} if $\bA$ is a strict monoidal functor.  Equivalently, $\cM$ is strict if
\[
    M \otimes (C_1 \otimes C_2) = (M \otimes C_1) \otimes C_2,\qquad
    \text{for all } M \in \cM,\ C_1,C_2 \in \cC,
\]
and similarly for morphisms.

We now describe the notion of a \emph{presentation} of $\cC$-modules by generators and relations.  We restrict ourselves to the case where there are only generating morphisms (i.e., no nontrivial generating objects).  Let $\mathtt{M}$ be the set of generating morphisms.  These are formal morphisms whose domains and codomains are objects of $\cC$.  Let $\cM$ be the $\kk$-linear category whose objects are the objects of $\cC$ and whose morphisms are generated (as a $\kk$-linear category) by morphisms of $\cC$ and morphisms
\[
    f \otimes g \colon \dom(f) \otimes \dom(g) \to \codom(f) \otimes \codom(g),\qquad
    f \in \mathtt{M},\ g \in \Mor(\cC),
\]
modulo the relations of linearity in the second argument ($g$ above) and
\begin{equation} \label{deer}
    (1_Y \otimes g) \circ (f \otimes h)
    = f \otimes (g \circ h)
    = (f \otimes g) \circ (1_X \otimes h),
\end{equation}
for all $f \colon X \to Y$ in $\mathtt{M}$ and all composable morphisms $g,h$ in $\cC$.  The category $\cM$ has a natural structure of a strict $\cC$-module.  Now, let $\mathtt{R}$ be a set of relations on the morphisms in $\cM$.  Define $\hat{\mathtt{R}}$ to be the smallest subset of morphisms of $\cM$ containing $\mathtt{R}$ and closed under composition, acting on the right by morphisms in $\cC$, and taking $\kk$-linear combinations.  Then we define the \emph{$\cC$-module generated by the morphisms $\mathtt{M}$ subject to the relations $\mathtt{R}$} to be the quotient of $\cM$ by $\hat{\mathtt{R}}$.  This has the structure of a strict $\cC$-module induced from the $\cC$-module structure on $\cM$.

Using the usual string diagram notation for module categories, the relations \cref{deer} become
\begin{equation} \label{snow}

    \ ,\qquad
    f \in \mathtt{M},\ g,h \in \Mor(\cC).
\end{equation}

\begin{defin} \label{DSdef}
    The \emph{disoriented skein category} $\DS(q,t)$ is the $\OS(q,t)$-module generated by the morphisms
    \[
        \togupdown \colon \downobj \to \upobj,\qquad
        \togdownup \colon \upobj \to \downobj,
    \]
    which we call \emph{toggles}, subject to the relations
    \begin{gather} \label{DStoggles}
        %
        \ .
    \end{gather}
\end{defin}

\begin{lem}
    The following relations hold in $\DS(q,t)$:
    \begin{equation} \label{DScurls2}
        %
        \ .
    \end{equation}
\end{lem}

\begin{proof}
    Composing on the top of the last relation in \cref{DStoggles} with
    \(
        %
        \,
    \),
    then using the first relation in \cref{DStoggles} gives
    \[
        %
        \ .
    \]
    Multiplying both sides by $t$ then yields the first relation in \cref{DScurls2}.  The proof of the second relation in \cref{DScurls2} is similar.
\end{proof}

The next lemma shows that the relations \cref{DStoggles,DScurls,DScurls2} hold after flipping crossings and inverting $q$.

\begin{lem} \label{headlight}
    The following relations hold in $\DS(q,t)$:
    \begin{gather} \label{otto1}
        %
        \ .
    \end{gather}
\end{lem}

\begin{proof}
    The first relation in \cref{otto1} follows from the first relation in \cref{DScurls} after composing on the bottom with $\negupcross$ and using \cref{obraid}.  The proofs of the second, third, and fourth relations in \cref{otto1} are analogous.

    Let $z=q-q^{-1}$.  We have
    \[
        %
        \right).
    \]
    Subtracting and using the third relation in \cref{DStoggles} then yields \cref{otto2}.
\end{proof}

We note that a diagram representing a morphism in $\DS(q,t)$ can only have $\togupdown$ or $\togdownup$ appearing on the leftmost strand.  However, we define the following morphisms:
\begin{equation} \label{toggy}
    %
    \colon \lambda \upobj\to \lambda \downobj,
    \qquad \lambda \in \word.
\end{equation}
Using \cref{toggy} and the module category structure of $\DS(q,t)$, we may now draw diagrams with $\togdownup$ and $\togupdown$ appearing on any strand and interpret them as morphisms in $\DS(q,t)$.

\begin{lem}
    The following relations hold in $\DS(q,t)$ for any $\lambda,\mu\in\word$:
    \begin{gather} \label{dubtog}
        %
        \ .
    \end{gather}
\end{lem}

\begin{proof}
    The equalities \cref{dubtog,togcross} follow easily from \cref{toggy,DStoggles}.  To prove the first relation in \cref{intertog}, it suffices to prove that
    \[
        %
        \right\},
    \]
    where, in the braces, the strands in the cups and caps can have any orientation, and $\lambda,\mu$ run over all objects of $\DS(q,t)$.  For
    \(
        f =
        %
        \ .
    \]
    Composing on the top and bottom with
    \[
        %
        \ ,
    \]
    respectively, then using \cref{DStoggles}, yields the case
    \(
        f =
        %
    \).
    The remaining cases for $f$ follow easily from \cref{snow,venom1,venom2}.  This completes the proof of the first relation in \cref{intertog}.  The proof of the second relation then follows by composing the first relation on the top and bottom with
    \[
        %
        \ ,
    \]
    respectively, then using \cref{DStoggles}.

    To prove the first equality in \cref{togcupcap}, we have
    \[
        %
        \ .
    \]
    The remaining relations in \cref{togcupcap} are proved similarly.  Note also that one can apply $\Omega_\updownarrow$, defined in \cref{horizonDS} below, to obtain the second and fourth equalities from the first and third, respectively.
    \details{
        To prove the third equality in \cref{togcupcap}, we compute
        \[
            %
            \ .
        \]
        The fourth equality in \cref{togcupcap} then follows from applying $\Omega_\updownarrow$.
    }
\end{proof}

In light of \cref{intertog}, we define toggles at the same height by
\[
    %
    \ ,
\]
and similarly for the other toggles, or for more than two toggles at the same height.

\begin{rem}\label{nonmonoidal}
    We emphasize that $\DS(q,t)$ is \emph{not} a monoidal category. For example, the relations in \cref{DScurls,togcupcap} do not hold, in general, unless they are on the left side of a diagram.  For instance, since
    \[
        %
        \ ,
    \]
    we cannot apply the last relation in \cref{DScurls} to the curl with a toggle directly in this situation.
\end{rem}

We now describe three natural symmetries of the disoriented skein category.  First, it follows from \cref{headlight} that the isomorphism \cref{horizonOS} induces a $\kk$-linear isomorphism of categories
\begin{equation} \label{horizonDS}
    \Omega_\updownarrow \colon \DS(q,t) \to \DS(q^{-1},t^{-1})^\op,\qquad
    \togupdown \mapsto \togupdown,\quad
    \togdownup \mapsto \togdownup,
\end{equation}
such that the diagram
\[
    \begin{tikzcd}
        \DS(q,t) \times \OS(q,t)
        \arrow[r,"\otimes"]
        \arrow[d,"\Omega_\updownarrow \times \Omega_\updownarrow"]
        &
        \DS(q,t)
        \arrow[d,"\Omega_\updownarrow"]
        \\
        \DS(q^{-1},t^{-1})^\op \times \OS(q^{-1},t^{-1})^\op
        \arrow[r,"\otimes"]
        &
        \DS(q^{-1},t^{-1})^\op
    \end{tikzcd}
\]
commutes.

Second, \cref{reverseOS} induces an isomorphism of $\kk$-linear categories
\begin{equation}
    \Theta \colon \DS(q,t) \to \DS(q,t),\qquad
    \togupdown \mapsto \togdownup,\quad
    \togdownup \mapsto \togupdown,
\end{equation}
such that the following diagram commutes:
\[
    \begin{tikzcd}
        \DS(q,t) \times \OS(q,t)
        \arrow[r,"\otimes"]
        \arrow[d,"\Theta \times \Theta"]
        &
        \DS(q,t)
        \arrow[d,"\Theta"]
        \\
        \DS(q,t) \times \OS(q,t)
        \arrow[r,"\otimes"]
        &
        \DS(q,t)
    \end{tikzcd}
\]
Indeed, $\Theta$ respects the relations \cref{DScurls} by \cref{DScurls2}, and it respects the last relation in \cref{DStoggles} by \cref{toggy} and the $f = \togdownup$ case of the last relation in \cref{intertog}.

Finally, it follows from \cref{headlight} that the bar involution \cref{barOS} induces a $\kk$-linear isomorphism of categories
\begin{equation} \label{barDS}
    \Xi \colon \DS(q,t) \to \DS(q^{-1},t^{-1}),\qquad
    \togupdown \mapsto \togupdown,\quad
    \togdownup \mapsto \togdownup,
\end{equation}
such that the following diagram commutes:
\[
    \begin{tikzcd}
        \DS(q,t) \times \OS(q,t)
        \arrow[r,"\otimes"]
        \arrow[d,"\Xi \times \Xi"]
        &
        \DS(q,t)
        \arrow[d,"\Xi"]
        \\
        \DS(q^{-1},t^{-1}) \times \OS(q^{-1},t^{-1})
        \arrow[r,"\otimes"]
        &
        \DS(q^{-1},t^{-1})
    \end{tikzcd}
\]

Note that, $\Omega_\updownarrow$, $\Theta$, and $\Xi$ do not behave well with respect to toggles on strands that are not on the left, due to the crossings in \cref{toggy}.  However, the composites $\Theta \circ \Xi$, and $\Omega_\updownarrow \circ \Theta$ \emph{do} behave well: they act as $\togupdown \mapsto \togdownup$, $\togdownup \mapsto \togupdown$ on toggles in arbitrary position.  Similarly, $\Omega_\updownarrow \circ \Xi$ acts as $\togupdown \mapsto \togupdown$, $\togdownup \mapsto \togdownup$ on toggles in arbitrary position.

\begin{rem}
    The last relation in \cref{DStoggles} is a twisted analogue of the reflection equation appearing in the definition of braided module categories; see \cref{reflection}.  In the corresponding categories of modules over quantum enveloping algebras and their coideal subalgebras, this twist is handled in \cite{Kol20} by passing to an equivariantization of the category of modules.
\end{rem}

%-------------------------------------------------------
\subsection{The iquantum Brauer category\label{subsec:iqB}}
%-------------------------------------------------------

We now recall the definition of the iquantum Brauer category.  This category was defined in \cite[Def.~7.9]{ST19}, where it was called the \emph{quantum} or \emph{$q$-Brauer category}.  Its endomorphism algebras first appeared in \cite[Def.~2.1]{Mol03}.  The definition in \cite[Def.~7.9]{ST19} is as a module category over the subcategory of $\OS(q,t)$ consisting of only upward-oriented strands.  We start with a presentation as a $\kk$-linear category.  Then, in \cref{sec:iqBmodstruct}, we endow the iquantum Brauer category with the structure of an $\OS(q,t)$-module category, extending the module category structure from \cite[Def.~7.9]{ST19}.  Below and throughout the paper, we let $\N$ denote the set of nonnegative integers.

\begin{defin} \label{iqBdef}
    The \emph{iquantum Brauer category} $\iqB(q,t) = \iqB_\kk(q,t)$ is the $\kk$-linear category with objects $\Bobj_r$, $r \in \N$, and generating morphisms
    \begin{gather} \label{Bcupcap}
        \begin{tikzpicture}[centerzero]
            \draw (-0.15,0.2) -- (-0.15,0.1) arc(180:360:0.15) -- (0.15,0.2);
            \draw[multi] (0.45,-0.2) \botlabel{r} -- (0.45,0.2);
        \end{tikzpicture}
        \colon \Bobj_r \to \Bobj_{r+2},
        \qquad
        \begin{tikzpicture}[centerzero]
            \draw (-0.15,-0.2) -- (-0.15,-0.1) arc(180:0:0.15) -- (0.15,-0.2);
            \draw[multi] (0.45,-0.2) \botlabel{r} -- (0.45,0.2);
        \end{tikzpicture}
        \colon \Bobj_{r+2} \to \Bobj_r,\qquad
        r \in \N,
        \\ \label{Bcross}
        \thickstrand{r} \poscross\ \thickstrand{s} \colon \Bobj_{r+s+2} \to \Bobj_{r+s+2},
        \qquad
        \thickstrand{r} \negcross\ \thickstrand{s} \colon \Bobj_{r+s+2} \to \Bobj_{r+s+2},\qquad
        r,s \in \N,
    \end{gather}
    subject to relations that we describe below.  We denote the identity morphism of $\Bobj_r$ by the thick strand $\thickstrand{r}$.  We define juxtaposition of thick strands by
    \begin{equation} \label{vacuum}
        \thickstrand{r_1} \thickstrand{r_2} \cdots \thickstrand{r_r} := \thickstrand{r_1+r_2+\dotsb+r_k},\qquad r_1,r_2,\dotsc,r_k \in \N.
    \end{equation}
    We use \cref{vacuum} even when these strands are part of a larger diagram that involves cups, caps, or crossings.

    We impose the following relations on morphisms, for all $r,s \in \N$:
    \begin{gather} \label{Bbraid}
,
    \end{gather}
    where, in \cref{Bcommute}, $f$ runs over all generating morphisms, $g$ runs over the generating morphisms \cref{Bcross}, and we interpret the juxtaposition using \cref{vacuum}.  For example, when
    \[
        f =
        %
        \right)
        \ .
    \end{multline*}
    This concludes the definition of $\iqB(q,t)$.
\end{defin}

We define cups and caps appearing at the same height by
\begin{equation} \label{mountain}
    \underbrace{\Bcup\ \Bcup \cdots \Bcup}_r
    :=
    \underbrace{
        %
    \ .
\end{equation}
A similar argument (or, alternatively, composing the second and third relations in \cref{Bcurl} with the appropriate crossings, then using the first two relations in \cref{Bbraid}) gives
\begin{equation} \label{eggnog}
    %
    \ .
\end{equation}
Note that the third relation in \cref{Bhumps} is obtained from the second by reflecting the diagrams in the horizontal axis.  The next result shows that the diagrammatic relation obtained by reflecting the first relation in \cref{Bhumps} also holds.

\begin{lem}
    In $\iqB(q,t)$,
    \begin{equation} \label{camel}
        %
        \ ,
    \]
    the result follows.
\end{proof}

\begin{lem}
    There is a $\kk$-linear isomorphism of categories
    \begin{equation} \label{barB}
        \Omega_\updownarrow \colon \iqB(q,t) \mapsto \iqB(q^{-1},t^{-1})^\op
    \end{equation}
    sending
    \[
        %
        \ ,\quad
        \thickstrand{r} \poscross\ \thickstrand{s} \mapsto \thickstrand{r} \negcross\ \thickstrand{s}
        \ ,\quad
        \thickstrand{r} \negcross\ \thickstrand{s} \mapsto \thickstrand{r} \poscross\ \thickstrand{s}\ .
    \]
\end{lem}

Intuitively, $\Omega_\updownarrow$ reflects diagrams in a horizontal axis and inverts $q$ and $t$.

\begin{proof}
    Using \cref{Bcurl2,eggnog,camel}, it is straightforward to verify that $\Omega_\updownarrow$ preserves the defining relations of $\iqB(q,t)$.  Since $\Omega_\updownarrow$ squares to the identity, it is an isomorphism.
\end{proof}

The following proposition implies that the relations \cref{Bhumps} continue to hold after all crossings are flipped.

\begin{prop}
    In $\iqB(q,t)$,
    \begin{equation} \label{Bhumps2}

            \ ,
    \end{equation}
\end{prop}

\begin{proof}
    The second and third relations in \cref{Bhumps2} follow from applying the isomorphism $\Omega_{\updownarrow}$ to the third and second relations, respectively, in \cref{Bhumps}.  Thus, it remains to prove the first relation in \cref{Bhumps2}.

    Let $z=q-q^{-1}$.  We repeatedly use the skein relation \cref{Bskein} to flip crossings:
    \[
        %
        \ .
    \]
    Adding together all the extra terms that appeared above gives $z$ times
    \[
        - qt\
        %
        \overset{\cref{Bskein}}{=} 0.
    \]
    Thus, the first relation in \cref{Bhumps2} is satisfied.
\end{proof}

\begin{lem} \label{lem:bar}
    There is a $\kk$-linear isomorphism of categories
    \[
        \Xi \colon \iqB(q,t) \to \iqB(q^{-1},t^{-1}),
    \]
    which we call the \emph{bar involution}, sending
    \[
        %
        , \quad
        \thickstrand{r} \poscross\ \thickstrand{s} \mapsto \thickstrand{r} \negcross\ \thickstrand{s}
        , \quad
        \thickstrand{r} \negcross\ \thickstrand{s} \mapsto \thickstrand{r} \poscross\ \thickstrand{s}
        .
    \]
\end{lem}

\begin{proof}
    Using \cref{Bcurl2,eggnog,Bhumps2}, it is straightforward to verify that the given map preserves the defining relations of $\iqB(q,t)$.  Since it squares to the identity, it is an isomorphism.
\end{proof}

\begin{rem}
    In \cite[Lem.~3.2]{CS24}, a bar involution was constructed for the iquantum Brauer algebras, which are isomorphic to endomorphism algebras of $\iqB(q,t)$; see \cref{river}. The bar involution in \cref{lem:bar} is a generalization of this involution from the algebra level to the category level.
\end{rem}

\begin{rem}
    For the reader familiar with the string diagram calculus for monoidal categories, we want to emphasize that the iquantum Brauer category is \emph{not} monoidal.  In particular, we may not, in general, use horizontal juxtaposition, which is the tensor product for monoidal categories.  This is why we have been careful in \cref{vacuum,Bcommute} to define some \emph{particular} horizontal juxtapositions; otherwise these would not be defined.
\end{rem}

\begin{rem}
    Taking $\kk = \Z[q,q^{-1}]$ and $t=q^m$ for some $m \in \Z$, the coefficient $(t-t^{-1})/(q-q^{-1})$ appearing in \cref{Bcurl} becomes the quantum integer $[m]$; see \cref{qint}.  Then, specializing $q=1$, the iquantum Brauer category becomes isomorphic to the usual Brauer category; see \cref{Guilin}.   Note, however, that the usual Brauer category is monoidal.
\end{rem}

%==========================================================================================
\section{Module category structure on the iquantum Brauer category\label{sec:iqBmodstruct}}
%==========================================================================================

The goal of the current section is to endow $\iqB(q,t)$ with the structure of a strict right module category over $\OS(q,t)$.  As noted in \cref{subsec:DS}, this amounts to defining a strict monoidal functor
\[
    \bA \colon \OS(q,t) \to \cEnd(\iqB(q,t))^\rev.
\]

We begin by defining functors that describe the action of the objects of $\OS(q,t)$.  Recall the convention \cref{vacuum} for juxtaposition of strands in $\iqB(q,t)$.  We define the following thick crossings:
\[
    \begin{tikzpicture}[centerzero]
        \draw[multi] (0,0.3) -- (0.6,-0.3) \botlabel{s};
        \draw[wipe] (0,-0.3) to (0.6,0.3);
        \draw[multi] (0,-0.3) \botlabel{r} -- (0.6,0.3);
    \end{tikzpicture}
    :=
    \begin{tikzpicture}[baseline=4mm]
        \draw (-.8,1) to (.2,0);
        \draw (-.2,1) to (.8,0);
        \draw[wipe] (-.8,0) to (.2,1);
        \draw[wipe] (-.2,0) to (.8,1);
        \draw (-.8,0) to (.2,1);
        \draw (-.2,0) to (.8,1);
        \node at (-.45,0){$\cdots$};
        \node at (-.45,1){$\cdots$};
        \node at (.45,0){$\cdots$};
        \node at (.45,1){$\cdots$};
        \node at (-.45,-.3){$\underbrace{\hspace{.02in}}_{r}$};
        \node at (.45,-.3){$\underbrace{\hspace{.02in}}_{s}$};
    \end{tikzpicture},
    \qquad
    \begin{tikzpicture}[centerzero]
        \draw[multi] (0,-0.3) \botlabel{r} to (0.6,0.3);
        \draw[wipe] (0,0.3) to (0.6,-0.3);
        \draw[multi] (0,.3) to (0.6,-0.3) \botlabel{s};
    \end{tikzpicture}
    :=
    \begin{tikzpicture}[baseline=4mm]
        \draw (-.8,0) to (.2,1);
        \draw (-.2,0) to (.8,1);
        \draw[wipe] (-.8,1) to (.2,0);
        \draw[wipe] (-.2,1) to (.8,0);
        \draw (-.8,1) to (.2,0);
        \draw (-.2,1) to (.8,0);
        \node at (-.45,0){$\cdots$};
        \node at (-.45,1){$\cdots$};
        \node at (.45,0){$\cdots$};
        \node at (.45,1){$\cdots$};
        \node at (-.45,-.3){$\underbrace{\hspace{.02in}}_{r}$};
        \node at (.45,-.3){$\underbrace{\hspace{.02in}}_{s}$};
    \end{tikzpicture}
    ,\qquad r,s \in \N.
\]
We then define $\kk$-linear endofunctors $\bA_k$, $k \in \N$, of $\iqB(q,t)$ given on objects by
\begin{equation}
    \bA_k (\Bobj_r) = \Bobj_{r+k},\qquad r \in \N,
\end{equation}
and on morphisms by
\begin{equation} \label{alpaca}
    \bA_k
    \left(
        \begin{tikzpicture}[centerzero]
            \draw[multi] (0,-0.5) \botlabel{r} -- (0,0.5) \toplabel{s};
            \coupon{0,0}{f};
        \end{tikzpicture}
    \right)
    =
    \begin{tikzpicture}[centerzero]
        \draw[multi] (0,-0.5) \botlabel{r} -- (0,0.5) \toplabel{s};
        \draw[multi] (0.4,-0.5) \botlabel{k} -- (0.4,0.5);
        \coupon{0,0}{f};
    \end{tikzpicture}
    \ ,
\end{equation}
for any morphism $f \colon \Bobj_r \to \Bobj_s$ in $\iqB(q,t)$.  It is straightforward to verify that these define endofunctors of $\iqB(q,t)$, for example, that they respect the defining relations of $\iqB(q,t)$.

\begin{prop}
    We have natural transformations
    \begin{gather} \label{cheese1}
        \bA(\posupcross) \colon \bA_2 \to \bA_2,\quad
        \bA(\negupcross) \colon \bA_2 \to \bA_2,
        \\ \label{cheese2}
        \bA(\rightcup) \colon \id \to \bA_2,\quad
        \bA(\rightcap) \colon \bA_2 \to \id,\quad
        \bA(\leftcup) \colon \id \to \bA_2,\quad
        \bA(\leftcap) \colon \bA_2 \to \id,
    \end{gather}
    with components given as follows:
    \begin{gather} \label{meat1}
        \bA(\posupcross)_{\Bobj_r} = \thickstrand{r} \poscross,\quad
        \bA(\negupcross)_{\Bobj_r} = \thickstrand{r} \negcross,
        \\ \label{meat2}
        \bA(\rightcup)_{\Bobj_r} =
        \begin{tikzpicture}[centerzero,cscale]
            \draw[multi] (0,-0.6) \botlabel{r} -- (0,-0.2) \braidup (-0.4,0.2);
            \draw[wipe] (-0.4,-0.2) \braidup (0,0.2);
            \draw (-0.4,0.6) \braiddown (-0.8,0.2)  -- (-0.8,-0.2) to[out=down,in=down,looseness=2] (-0.4,-0.2) \braidup (0,0.2) -- (0,0.6);
            \draw[wipe] (-0.4,0.2) \braidup (-0.8,0.6);
            \draw[multi] (-0.4,0.2) \braidup (-0.8,0.6);
        \end{tikzpicture}
        \ ,\quad
        \bA(\leftcup)_{\Bobj_r} = q^{-1}t\
        \begin{tikzpicture}[centerzero,cscale]
            \draw[multi] (-0.4,0.2) \braidup (-0.8,0.6);
            \draw[wipe] (-0.4,0.6) \braiddown (-0.8,0.2);
            \draw (-0.4,0.6) \braiddown (-0.8,0.2)  -- (-0.8,-0.2) to[out=down,in=down,looseness=2] (-0.4,-0.2) \braidup (0,0.2) -- (0,0.6);
            \draw[wipe] (0,-0.2) \braidup (-0.4,0.2);
            \draw[multi] (0,-0.6) \botlabel{r} -- (0,-0.2) \braidup (-0.4,0.2);
        \end{tikzpicture}
        \ ,\quad
        \bA(\leftcap)_{\Bobj_r} =
        \begin{tikzpicture}[centerzero,cscale]
            \draw[multi] (-0.4,-0.2) \braidup (0,0.2) -- (0,0.6);
            \draw[wipe] (-0.4,0.2) \braiddown (0,-0.2);
            \draw (-0.4,-0.6) \braidup (-0.8,-0.2) -- (-0.8,0.2) to[out=up,in=up,looseness=2] (-0.4,0.2) \braiddown (0,-0.2) -- (0,-0.6);
            \draw[wipe] (-0.8,-0.6) \braidup (-0.4,-0.2);
            \draw[multi] (-0.8,-0.6) \botlabel{r} \braidup (-0.4,-0.2);
        \end{tikzpicture}
        \ ,\quad
        \bA(\rightcap)_{\Bobj_r} = qt^{-1}\,
        \begin{tikzpicture}[centerzero,cscale]
            \draw[multi] (-0.8,-0.6) \botlabel{r} \braidup (-0.4,-0.2);
            \draw[wipe] (-0.4,-0.6) \braidup (-0.8,-0.2);
            \draw (-0.4,-0.6) \braidup (-0.8,-0.2) -- (-0.8,0.2) to[out=up,in=up,looseness=2] (-0.4,0.2) \braiddown (0,-0.2) -- (0,-0.6);
            \draw[wipe] (-0.4,-0.2) \braidup (0,0.2);
            \draw[multi] (-0.4,-0.2) \braidup (0,0.2) -- (0,0.6);
        \end{tikzpicture}
        \ .
    \end{gather}
\end{prop}

\begin{proof}
    To show that $\bA(\posupcross)$ is a natural transformation, we must show that, for $r,s \in \N$, and $f \colon \Bobj_r \to \Bobj_s$ a morphism in $\iqB(q,t)$, the diagram
    \begin{equation}
        \xymatrix{
            \bA_2(\Bobj_s) \ar[rr]^{\thickstrand{s} \poscross} &&
            \bA_2(\Bobj_s)
            \\
            \bA_2(\Bobj_r) \ar[rr]_{\thickstrand{r}\poscross} \ar[u]^{\bA_2(f)} &&
            \bA_2(\Bobj_r) \ar[u]_{\bA_2(f)}
        }
    \end{equation}
    commutes.  This follows immediately from \cref{Bcommute}.  Hence $\bA(\posupcross)$ is a natural transformation.  The proof that $\bA(\negupcross)$ is a natural transformation is similar.

    To show that $\bA(\rightcup)$ is a natural transformation, we must show that, for $f$ a generating morphism of $\iqB(q,t)$, the diagram
    \[
        \begin{tikzcd}
            \Bobj_s \arrow{rr}{\bA(\rightcup)_{\Bobj_s}} & &
            \bA_2(\Bobj_s)
            \\
            \Bobj_r \arrow{rr}{\bA(\rightcup)_{\Bobj_r}} \arrow{u}{f} & &
            \bA_2(\Bobj_r) \arrow[swap]{u}{\bA_2(f)}
        \end{tikzcd}
    \]
    commutes.  In other words, we must show that
    \begin{equation} \label{rogue}
        \begin{tikzpicture}[centerzero]
            \draw[multi] (0,-1) -- (0,-0.6) \braidup (-0.4,-0.2);
            \draw[wipe] (0,-0.2) \braiddown (-0.4,-0.6);
            \draw (0,1) -- (0,-0.2) \braiddown (-0.4,-0.6) to[out=down,in=down,looseness=2] (-0.8,-0.6) -- (-0.8,-0.2) \braidup (-0.4,0.2) -- (-0.4,1);
            \draw[wipe] (-0.4,-0.2) \braidup (-0.8,0.2);
            \draw[multi] (-0.4,-0.2) \braidup (-0.8,0.2) -- (-0.8,1);
            \coupon{-0.8,0.5}{f};
        \end{tikzpicture}
        \ =\
        \begin{tikzpicture}[centerzero]
            \draw[multi] (0,-1) -- (0,0.2) \braidup (-0.4,0.6);
            \draw[wipe] (0,0.6) \braiddown (-0.4,0.2);
            \draw (0,1) -- (0,0.6) \braiddown (-0.4,0.2) to[out=down,in=down,looseness=2] (-0.8,0.2) -- (-0.8,0.6) \braidup (-0.4,1);
            \draw[wipe] (-0.4,0.6) \braidup (-0.8,1);
            \draw[multi] (-0.4,0.6) \braidup (-0.8,1);
            \coupon{0,-0.5}{f};
        \end{tikzpicture}
        \ .
    \end{equation}
    When
    \(
        f =
        \begin{tikzpicture}[centerzero]
            \draw (-0.15,0.2) -- (-0.15,0.1) arc(180:360:0.15) -- (0.15,0.2);
            \draw[multi] (0.45,-0.2) \botlabel{r} -- (0.45,0.2);
        \end{tikzpicture},
    \)
    this is the third relation in \cref{Bhumps}.  When
    \(
        f =
        \begin{tikzpicture}[centerzero]
            \draw (-0.15,-0.2) -- (-0.15,-0.1) arc(180:0:0.15) -- (0.15,-0.2);
            \draw[multi] (0.45,-0.2) \botlabel{r} -- (0.45,0.2);
        \end{tikzpicture},
    \)
    it is \cref{camel}.  Finally, when $f = \thickstrand{r} \negcross \thickstrand{s}$, \cref{rogue} follows from \cref{Bbraid,Bcommute}.  The proofs that $\bA(\leftcup)$, $\bA(\rightcap)$, and $\bA(\leftcap)$ are natural transformations are analogous.
\end{proof}

We will show in \cref{camping} that $\bA$ yields a well-defined functor from $\OS(q,t)$.  First, we extend $\bA$ to compositions and tensor products of generating morphisms by requiring that $\bA$ commute with these two operations.  We will denote horizontal composition of natural transformations by $*$ and the identity natural transformation of a functor $\bH$ by $\id_{\bH}$.  

\begin{lem}
    For $n \in \N$,
    \begin{gather} \label{train1}
        \bA(\posrightcross)_{\Bobj_r}
        = \thickstrand{r} \poscross,
        \quad
        \bA(\negleftcross)_{\Bobj_r}
        = \thickstrand{r} \negcross,
        \quad
        \bA(\posdowncross)_{\Bobj_r}
        = \thickstrand{r} \poscross,
        \quad
        \bA(\negdowncross)_{\Bobj_r}
        = \thickstrand{r} \negcross,
        \\ \label{train2}
        \bA(\negrightcross)_{\Bobj_r}
        =
        \begin{tikzpicture}[centerzero,cscale]
            \draw[multi] (0,-1.2) \botlabel{r} -- (0,1.2);
            \draw (0.8,-1.2) -- (0.8,-0.3) \braidup (0.4,0.3) -- (0.4,1.2);
            \draw[wipe] (0.4,-0.3) \braidup (0.8,0.3);
            \draw (0.4,-1.2) -- (0.4,-0.3) \braidup (0.8,0.3) -- (0.8,1.2);
        \end{tikzpicture}
        - (q^2-1)t^{-1}
        \begin{tikzpicture}[centerzero,cscale]
            \draw (0.8,-1.2) -- (0.8,-0.8) \braidup (0.4,-0.4) to[out=up,in=up,looseness=2] (0,-0.4) -- (0,-0.8);
            \draw (0.4,1.2) \braiddown (0,0.8) -- (0,0.4) to[out=down,in=down,looseness=2] (0.4,0.4);
            \draw[wipe] (0.4,-0.8) \braidup (0.8,-0.4) -- (0.8,0.4) \braidup (0.4,0.8) \braidup (0,1.2);
            \draw[multi] (0,-1.2) \botlabel{r} \braidup (0.4,-0.8) \braidup (0.8,-0.4) -- (0.8,0.4) \braidup (0.4,0.8) \braidup (0,1.2);
            \draw[wipe] (0.4,-1.2) \braidup (0,-0.8);
            \draw (0.4,-1.2) \braidup (0,-0.8);
            \draw[wipe] (0.4,0.4) \braidup (0.8,0.8) -- (0.8,1.2);
            \draw (0.4,0.4) \braidup (0.8,0.8) -- (0.8,1.2);
        \end{tikzpicture}
        \ ,\qquad
        \bA(\posleftcross)_{\Bobj_r}
        =
        \begin{tikzpicture}[centerzero,cscale]
            \draw[multi] (0,-1.2) \botlabel{r} -- (0,1.2);
            \draw (0.4,-1.2) -- (0.4,-0.3) \braidup (0.8,0.3) -- (0.8,1.2);
            \draw[wipe] (0.8,-0.3) \braidup (0.4,0.3);
            \draw (0.8,-1.2) -- (0.8,-0.3) \braidup (0.4,0.3) -- (0.4,1.2);
        \end{tikzpicture}
        + (q^{-2}-1) t
        \begin{tikzpicture}[centerzero,cscale]
            \draw (0.4,0.4) \braidup (0.8,0.8) -- (0.8,1.2);
            \draw[wipe] (0.4,-1.2) \braidup (0,-0.8);
            \draw (0.4,-1.2) \braidup (0,-0.8);
            \draw[wipe] (0,-1.2) \braidup (0.4,-0.8) \braidup (0.8,-0.4) -- (0.8,0.4) \braidup (0.4,0.8) \braidup (0,1.2);
            \draw[multi] (0,-1.2) \botlabel{r} \braidup (0.4,-0.8) \braidup (0.8,-0.4) -- (0.8,0.4) \braidup (0.4,0.8) \braidup (0,1.2);
            \draw[wipe] (0.4,1.2) \braiddown (0,0.8) -- (0,0.4) to[out=down,in=down,looseness=2] (0.4,0.4);
            \draw (0.4,1.2) \braiddown (0,0.8) -- (0,0.4) to[out=down,in=down,looseness=2] (0.4,0.4);
            \draw[wipe] (0.8,-1.2) -- (0.8,-0.8) \braidup (0.4,-0.4) to[out=up,in=up,looseness=2] (0,-0.4) -- (0,-0.8);
            \draw (0.8,-1.2) -- (0.8,-0.8) \braidup (0.4,-0.4) to[out=up,in=up,looseness=2] (0,-0.4) -- (0,-0.8);
        \end{tikzpicture}
        \ .
    \end{gather}
\end{lem}

\begin{proof}
    We first compute
    \begin{gather*}
        \bA( \downarrow \uparrow \rightcap )_{\Bobj_r}
        = \big( \bA(\rightcap) * \bA(\uparrow) * \bA(\downarrow) \big)_{\Bobj_r}
        = \big( \bA(\rightcap) * \id_{\bA_2} \big)_{\Bobj_r}
        = \bA( \rightcap )_{\Bobj_{r+2}}
        = qt^{-1}\,
        \begin{tikzpicture}[centerzero,cscale]
            \draw[multi] (-0.8,-0.6) \braidup (-0.4,-0.2);
            \draw[wipe] (-0.4,-0.6) \braidup (-0.8,-0.2);
            \draw (-0.4,-0.6) \braidup (-0.8,-0.2) -- (-0.8,0.2) to[out=up,in=up,looseness=2] (-0.4,0.2) \braiddown (0,-0.2) -- (0,-0.6);
            \draw[wipe] (-0.4,-0.2) \braidup (0,0.2);
            \draw[multi] (-0.4,-0.2) \braidup (0,0.2) -- (0,0.6) \toplabel{r+2};
        \end{tikzpicture}
        \, ,
        \\
        \bA(\downarrow \posupcross \downarrow)_{\Bobj_r}
        = \big( \id_{\bA_1} * \bA(\posupcross) * \id_{\bA_1} \big)_{\Bobj_r}
        = \bA_1 \big( \bA(\posupcross)_{\Bobj_{r+1}} \big)
        = \bA_1 \left( \thickstrand{r+1} \poscross \right)
        \overset{\cref{alpaca}}{=} \thickstrand{r+1} \poscross\ \ \strand\, ,
        \\
        \bA(\rightcup \uparrow \downarrow)_{\Bobj_r}
        = \big( \id_{\bA_2} * \bA(\rightcup) \big)_{\Bobj_r}
        = \bA_2 \big( \bA(\rightcup)_{\Bobj_r} \big)
        = \bA_2
        \left(
            \begin{tikzpicture}[centerzero,cscale]
                \draw[multi] (0,-0.6) \botlabel{r} -- (0,-0.2) \braidup (-0.4,0.2);
                \draw[wipe] (-0.4,-0.2) \braidup (0,0.2);
                \draw (-0.4,0.6) \braiddown (-0.8,0.2)  -- (-0.8,-0.2) to[out=down,in=down,looseness=2] (-0.4,-0.2) \braidup (0,0.2) -- (0,0.6);
                \draw[wipe] (-0.4,0.2) \braidup (-0.8,0.6);
                \draw[multi] (-0.4,0.2) \braidup (-0.8,0.6);
            \end{tikzpicture}
        \right)
        \overset{\cref{alpaca}}{=}
        \begin{tikzpicture}[centerzero,cscale]
            \draw[multi] (0,-0.6) \botlabel{r} -- (0,-0.2) \braidup (-0.4,0.2);
            \draw[wipe] (-0.4,-0.2) \braidup (0,0.2);
            \draw (-0.4,0.6) \braiddown (-0.8,0.2)  -- (-0.8,-0.2) to[out=down,in=down,looseness=2] (-0.4,-0.2) \braidup (0,0.2) -- (0,0.6);
            \draw[wipe] (-0.4,0.2) \braidup (-0.8,0.6);
            \draw[multi] (-0.4,0.2) \braidup (-0.8,0.6);
            \draw (0.4,-0.6) -- (0.4,0.6);
            \draw (0.8,-0.6) -- (0.8,0.6);
        \end{tikzpicture}
        \ .
    \end{gather*}
    Thus,
    \begin{multline*}
        \bA(\posrightcross)_{\Bobj_r}
        \overset{\cref{lego}}{=}
        \bA
        \left(
            \begin{tikzpicture}[centerzero,xscale=-1]
                \draw[->] (-0.2,-0.3) \braidup (0.2,0.3);
                \draw[wipe] (-0.2,0.2)  \braidup (0.2,-0.2);
                \draw[->] (0.4,0.3) -- (0.4,0.1) to[out=down,in=right] (0.2,-0.2) to[out=left,in=right] (-0.2,0.2) to[out=left,in=up] (-0.4,-0.1) -- (-0.4,-0.3);
            \end{tikzpicture}
        \right)_{\Bobj_r}
        =
        \bA( \downarrow \uparrow \rightcap )_{\Bobj_r} \circ \bA(\downarrow \posupcross \downarrow)_{\Bobj_r} \circ \bA(\rightcup \uparrow \downarrow)_{\Bobj_r}
        \\
        = qt^{-1}\
        \begin{tikzpicture}[centerzero,cscale]
            \draw (0.8,-2.2) -- (0.8,0.6) \braidup (0.4,1) \braidup (0,1.4) \braidup (-0.4,1.8) to[out=up,in=up,looseness=2] (-0.8,1.8) -- (-0.8,0.6);
            \draw[wipe] (0.4,0.6) \braidup (0.8,1);
            \draw (0.4,-2.2) -- (0.4,-1) \braidup (0,-0.6) \braidup (0.4,-0.2) -- (0.4,0.6) \braidup (0.8,1) -- (0.8,2.2);
            \draw[wipe] (0.4,1.4) \braiddown (0,1);
            \draw (0.4,2.2) -- (0.4,1.4) \braiddown (0,1) -- (0,0.2) \braiddown (-0.4,-0.2) -- (-0.4,-1) \braiddown (-0.8,-1.4) -- (-0.8,-1.8) to[out=down,in=down,looseness=2] (-0.4,-1.8);
            \draw[wipe] (0,-2.2) -- (0,-1.8) \braidup (-0.4,-1.4) \braidup (-0.8,-1) -- (-0.8,0.2) \braidup (-0.4,0.6) -- (-0.4,1.4) \braidup (0,1.8) -- (0,2.2);
            \draw[multi] (0,-2.2) \botlabel{r} -- (0,-1.8) \braidup (-0.4,-1.4) \braidup (-0.8,-1) -- (-0.8,0.2) \braidup (-0.4,0.6) -- (-0.4,1.4) \braidup (0,1.8) -- (0,2.2);
            \draw[wipe] (-0.4,-1.8) \braidup (0,-1.4) -- (0,-1) \braidup (0.4,-0.6) \braidup (0,-0.2) \braidup (-0.4,0.2) \braidup (-0.8,0.6);
            \draw (-0.4,-1.8) \braidup (0,-1.4) -- (0,-1) \braidup (0.4,-0.6) \braidup (0,-0.2) \braidup (-0.4,0.2) \braidup (-0.8,0.6);
        \end{tikzpicture}
        \overset{\cref{Bbraid}}{=} qt^{-1}\
        \begin{tikzpicture}[centerzero,cscale]
            \draw (0.8,-1.4) -- (0.8,-0.6) \braidup (0.4,-0.2) \braidup (0,0.2) \braidup (-0.4,0.6) -- (-0.4,1) to[out=up,in=up,looseness=2] (-0.8,1) -- (-0.8,-0.6);
            \draw[wipe] (0,0.6) \braiddown (-0.4,0.2);
            \draw (0.4,1.4) -- (0.4,1) \braiddown (0,0.6) \braiddown (-0.4,0.2) -- (-0.4,-0.6) \braiddown (-0.8,-1) to[out=down,in=down,looseness=2] (-0.4,-1);
            \draw[wipe] (-0.4,-1) \braidup (-0.8,-0.6);
            \draw (-0.4,-1)  \braidup (-0.8,-0.6);
            \draw[wipe] (0.4,-0.6) \braidup (0.8,-0.2);
            \draw (0.4,-1.4) -- (0.4,-0.6) \braidup (0.8,-0.2) -- (0.8,1.4);
            \draw[wipe] (0,-0.2) \braidup (0.4,0.2) -- (0.4,0.6) \braidup (0,1);
            \draw[multi] (0,-1.4) \botlabel{r} -- (0,-0.2) \braidup (0.4,0.2) -- (0.4,0.6) \braidup (0,1) -- (0,1.4);
        \end{tikzpicture}
        \overset{\cref{eggnog}}{\underset{\cref{Bcurl}}{=}}
        \begin{tikzpicture}[centerzero,cscale]
            \draw (0.8,-1.4) -- (0.8,-0.6) \braidup (0.4,-0.2) \braidup (0,0.2) -- (0,0.6) \braidup (0.4,1) -- (0.4,1.4);
            \draw[wipe] (0.4,-0.6) \braidup (0.8,-0.2);
            \draw (0.4,-1.4) -- (0.4,-0.6) \braidup (0.8,-0.2) -- (0.8,1.4);
            \draw[wipe] (0,-0.2) \braidup (0.4,0.2) -- (0.4,0.6) \braidup (0,1);
            \draw[multi] (0,-1.4) \botlabel{r} -- (0,-0.2) \braidup (0.4,0.2) -- (0.4,0.6) \braidup (0,1) -- (0,1.4);
        \end{tikzpicture}
        \overset{\cref{Bbraid}}{=}
        \begin{tikzpicture}[centerzero,cscale]
            \draw[multi] (0,-1.4) \botlabel{r} -- (0,1.4);
            \draw (0.8,-1.4) -- (0.8,-0.2) \braidup (0.4,0.2) -- (0.4,1.4);
            \draw[wipe] (0.4,-0.2) \braidup (0.8,0.2);
            \draw (0.4,-1.4) -- (0.4,-0.2) \braidup (0.8,0.2) -- (0.8,1.4);
        \end{tikzpicture}
        \ .
    \end{multline*}
    A similar computation shows that the second equality in \cref{train1} holds.  Then we have
    \begin{multline*}
        \bA(\posdowncross)_{\Bobj_r}
        = \bA
        \left(
            \begin{tikzpicture}[centerzero,xscale=-1]
                \draw[->] (0.4,0.3) -- (0.4,0.1) to[out=down,in=right] (0.2,-0.2) to[out=left,in=right] (-0.2,0.2) to[out=left,in=up] (-0.4,-0.1) -- (-0.4,-0.3);
                \draw[wipe] (-0.2,-0.3) \braidup (0.2,0.3);
                \draw[<-] (-0.2,-0.3) \braidup (0.2,0.3);
            \end{tikzpicture}
        \right)_{\Bobj_r}
        =
        \bA( \downarrow \downarrow \rightcap )_{\Bobj_r} \circ \bA(\downarrow \posrightcross \downarrow)_{\Bobj_r} \circ \bA(\rightcup \downarrow \downarrow)_{\Bobj_r}
        \\
        = \bA( \downarrow \downarrow \rightcap )_{\Bobj_r} \circ \left( \thickstrand{r+1} \poscross\ \ \strand \right) \circ \bA(\rightcup \downarrow \downarrow)_{\Bobj_r}
        = \thickstrand{r} \poscross,
    \end{multline*}
    where the third equality above follows from the first equality in \cref{train1}, and then the final equality above follows from our earlier computation of $\bA(\posrightcross)_{\Bobj_r}$.  The proof of the final equality in \cref{train1} is analogous, as are the relations \cref{train2}.
\end{proof}

\begin{theo} \label{camping}
    We have a strict monoidal functor $\bA \colon \OS(q,t) \to \cEnd(\iqB(q,t))^\rev$ given on objects by $\upobj, \downobj \mapsto \bA_1$, and on morphisms by \cref{cheese1,cheese2}.
\end{theo}

\begin{proof}
    We must show that $\bA$ respects the relations \cref{obraid,oskein,ocurlbub,adjunction}.  The first three equalities in \cref{obraid} follow immediately from \cref{Bbraid}.  The last two equalities in \cref{obraid} follow from \cref{train1}.  The fact that $\bA$ respects \cref{oskein} follows from \cref{Bskein}.

    For the first equality in \cref{ocurlbub}, we compute
    \[
        \bA
        \left(
            \begin{tikzpicture}[centerzero,xscale=-1]
                \draw (0,-0.4) to[out=up,in=180] (0.25,0.15) to[out=0,in=up] (0.4,0);
                \draw[wipe] (0.25,-0.15) to[out=180,in=down] (0,0.4);
                \draw[->] (0.4,0) to[out=down,in=0] (0.25,-0.15) to[out=180,in=down] (0,0.4);
            \end{tikzpicture}
        \right)_{\Bobj_r}
        =
        \bA(\leftcap \uparrow)_{\Bobj_r} \circ \bA(\downarrow \posupcross)_{\Bobj_r} \circ \bA(\rightcup \uparrow)_{\Bobj_r}
        =
        \begin{tikzpicture}[centerzero,cscale]
            \draw (0.4,-1.4) -- (0.4,-0.2) \braidup (0,0.2) -- (0,0.6);
            \draw[wipe] (0.4,0.2) \braiddown (0,-0.2);
            \draw (0.4,1.4) -- (0.4,0.2) \braiddown (0,-0.2) -- (0,-0.6);
            \draw[multi] (0,1.4) -- (0,1) \braiddown (-0.4,0.6);
            \draw[multi] (0,-1.4) \botlabel{r} -- (0,-1) \braidup (-0.4,-0.6);
            \draw[wipe] (0,0.6) \braidup (-0.4,1) to[out=up,in=up,looseness=2] (-0.8,1) -- (-0.8,0.6) \braiddown (-0.4,0.2) -- (-0.4,-0.2) \braiddown (-0.8,-0.6) -- (-0.8,-1) to[out=down,in=down,looseness=2] (-0.4,-1) \braidup (0,-0.6);
            \draw (0,0.6) \braidup (-0.4,1) to[out=up,in=up,looseness=2] (-0.8,1) -- (-0.8,0.6) \braiddown (-0.4,0.2) -- (-0.4,-0.2) \braiddown (-0.8,-0.6) -- (-0.8,-1) to[out=down,in=down,looseness=2] (-0.4,-1) \braidup (0,-0.6);
            \draw[wipe] (-0.4,0.6) \braiddown (-0.8,0.2) -- (-0.8,-0.2) \braiddown (-0.4,-0.6);
            \draw[multi] (-0.4,0.6) \braiddown (-0.8,0.2) -- (-0.8,-0.2) \braiddown (-0.4,-0.6);
        \end{tikzpicture}
        \overset{\cref{Bbraid}}{\underset{\cref{Bcommute}}{=}}
        \begin{tikzpicture}[centerzero,cscale]
            \draw[multi] (-0.4,-1) \botlabel{r} \braidup (0,-0.6) -- (0,0.6) \braidup (-0.4,1);
            \draw[wipe] (0,-1) \braidup (-0.4,-0.6);
            \draw (0,-1) \braidup (-0.4,-0.6) -- (-0.4,-0.2) \braidup (-0.8,0.2) to[out=up,in=up,looseness=2] (-1.2,0.2) -- (-1.2,-0.2) to[out=down,in=down,looseness=2] (-0.8,-0.2);
            \draw[wipe] (-0.8,-0.2) \braidup (-0.4,0.2) -- (-0.4,0.6) \braidup (-0,1);
            \draw (-0.8,-0.2) \braidup (-0.4,0.2) -- (-0.4,0.6) \braidup (-0,1);
        \end{tikzpicture}
        \overset{\cref{Bcurl}}{\underset{\cref{Bbraid}}{=}} t\
        \begin{tikzpicture}[centerzero,cscale]
            \draw[multi] (0,-1) \botlabel{r+1} -- (0,1);
        \end{tikzpicture}
        = t \bA
        \left(
            \begin{tikzpicture}[centerzero]
                \draw[->] (0,-0.4) -- (0,0.4);
            \end{tikzpicture}
        \right)_{\Bobj_r},
    \]
    as desired.  A similar computation shows that $\bA$ respects the second equality in \cref{ocurlbub}.  For the third equality in \cref{ocurlbub}, we have
    \[
        \bA \left( \rightbub \right)_{\Bobj_r}
        = \bA(\rightcap)_{\Bobj_r} \circ \bA(\leftcup)_{\Bobj_r}
        =
        \begin{tikzpicture}[centerzero,cscale]
            \draw (-0.8,-0.4) -- (-0.8,-0.8) to[out=down,in=down,looseness=2] (-0.4,-0.8) \braidup (0,-0.4) -- (0,0.4) \braidup (-0.4,0.8) to[out=up,in=up,looseness=2] (-0.8,0.8) -- (-0.8,0.4);
            \draw[wipe] (0,-1.2) \botlabel{r} -- (0,-0.8) \braidup (-0.4,-0.4) \braidup (-0.8,0) \braidup (-0.4,0.4) \braidup (0,0.8) -- (0,1.2);
            \draw[multi] (0,-1.2) \botlabel{r} -- (0,-0.8) \braidup (-0.4,-0.4) \braidup (-0.8,0) \braidup (-0.4,0.4) \braidup (0,0.8) -- (0,1.2);
            \draw[wipe] (-0.8,-0.4) \braidup (-0.4,0) \braidup (-0.8,0.4);
            \draw (-0.8,-0.4) \braidup (-0.4,0) \braidup (-0.8,0.4);
        \end{tikzpicture}
        \overset{\cref{Bbraid}}{\underset{\cref{Bcurl}}{=}} \frac{t-t^{-1}}{q-q^{-1}}\,
        \begin{tikzpicture}[centerzero,cscale]
            \draw[multi] (0,-0.8) \botlabel{r} -- (0,0.8);
        \end{tikzpicture}
        = \frac{t-t^{-1}}{q-q^{-1}} \bA(1_\one)_{\Bobj_r},
    \]
    as desired.

    For the first relation in \cref{adjunction}, we compute
    \[
        \bA
        \left(
            \begin{tikzpicture}[centerzero]
                \draw[->] (-0.3,-0.4) -- (-0.3,0) arc(180:0:0.15) arc(180:360:0.15) -- (0.3,0.4);
            \end{tikzpicture}
        \right)_{\Bobj_r}
        =
        \bA(\rightcap \uparrow)_{\Bobj_r} \circ \bA(\uparrow \rightcup)_{\Bobj_r}
        = qt^{-1}\
        \begin{tikzpicture}[centerzero,cscale]
            \draw (0.4,-1.6) -- (0.4,-0.8) \braidup (0,-0.4) -- (0,0);
            \draw[multi] (0,-1.6) \botlabel{r} -- (0,-1.2) \braidup (-0.4,-0.8) -- (-0.4,-0.4);
            \draw[wipe] (-0.4,-1.2) \braidup (0,-0.8) \braidup (0.4,-0.4);
            \draw (-0.8,0.8) -- (-0.8,1.2) to[out=up,in=up,looseness=2] (-0.4,1.2) \braiddown (0,0.8) -- (0,0.4) \braiddown (-0.4,0) \braiddown (-0.8,-0.4) -- (-0.8,-1.2) to[out=down,in=down,looseness=2] (-0.4,-1.2) \braidup (0,-0.8) \braidup (0.4,-0.4) -- (0.4,1.6);
            \draw[wipe] (0,1.2) \braiddown (-0.4,0.8) \braiddown (-0.8,0.4) -- (-0.8,0) \braiddown (-0.4,-0.4);
            \draw[multi] (0,1.6) -- (0,1.2) \braiddown (-0.4,0.8) \braiddown (-0.8,0.4) -- (-0.8,0) \braiddown (-0.4,-0.4);
            \draw[wipe] (0,0) \braidup (-0.4,0.4) \braidup (-0.8,0.8);
            \draw (0,0) \braidup (-0.4,0.4) \braidup (-0.8,0.8);
        \end{tikzpicture}
        \overset{\cref{Bbraid}}{\underset{\cref{Bcommute}}{=}} qt^{-1}\
        \begin{tikzpicture}[centerzero,cscale]
            \draw[multi] (0,-1.2) \botlabel{r} \braidup (0.4,-0.8) -- (0.4,0.8) \braidup (0,1.2);
            \draw[wipe] (0.4,-1.2) \braidup (0,-0.8);
            \draw (0.4,-1.2) \braidup (0,-0.8) -- (0,-0.4) \braidup (-0.4,0);
            \draw[wipe] (0.4,1.2) \braiddown (0,0.8) -- (0,0) \braiddown (-0.4,-0.4);
            \draw (0.4,1.2) \braiddown (0,0.8) -- (0,0) \braiddown (-0.4,-0.4) to[out=down,in=down,looseness=2] (-0.8,-0.4) -- (-0.8,0) \braidup (-0.4,0.4) to[out=up,in=up,looseness=2] (-0.8,0.4);
            \draw[wipe] (-0.4,0) \braidup (-0.8,0.4);
            \draw (-0.4,0) \braidup (-0.8,0.4);
        \end{tikzpicture}
        \overset{\substack{\cref{eggnog} \\ \cref{Bcurl}}}{\underset{\cref{Bbraid}}{=}}
        \begin{tikzpicture}[centerzero,cscale]
            \draw[multi] (-0.4,-1) \botlabel{r+1} -- (-0.4,1);
        \end{tikzpicture}
        = \bA
        \left(
            \begin{tikzpicture}[centerzero]
                \draw[->] (0,-0.4) -- (0,0.4);
            \end{tikzpicture}
        \right)_{\Bobj_r}.
    \]
    The proof that $\bA$ respects the second relation in \cref{adjunction} is similar.
\end{proof}

Written in module-theoretic (as opposed to representation-theoretic) notation, as in \cref{beaver}, the action defined in \cref{camping} is as follows.  On objects,
\[
    \Bobj_r \otimes \upobj := \Bobj_{r+1},\qquad
    \Bobj_r \otimes \downobj := \Bobj_{r+1},\qquad
    r \in \N,
\]
and, on morphisms,
\begin{gather*}
    \begin{tikzpicture}[centerzero]
        \draw[multi] (0,-0.4) -- (0,0.4);
        \coupon{0,0}{f};
    \end{tikzpicture}
    \otimes \posupcross
    :=
    \begin{tikzpicture}[centerzero]
        \draw[multi] (0,-0.4) -- (0,0.4);
        \coupon{0,0}{f};
        \draw (0.6,-0.4) \braidup (0.3,0.4);
        \draw[wipe] (0.3,-0.4) \braidup (0.6,0.4);
        \draw (0.3,-0.4) \braidup (0.6,0.4);
    \end{tikzpicture}
    ,\qquad
    \begin{tikzpicture}[centerzero]
        \draw[multi] (0,-0.4) -- (0,0.4);
        \coupon{0,0}{f};
    \end{tikzpicture}
    \otimes \negupcross
    :=
    \begin{tikzpicture}[centerzero]
        \draw[multi] (0,-0.4) -- (0,0.4);
        \coupon{0,0}{f};
        \draw (0.3,-0.4) \braidup (0.6,0.4);
        \draw[wipe] (0.6,-0.4) \braidup (0.3,0.4);
        \draw (0.6,-0.4) \braidup (0.3,0.4);
    \end{tikzpicture}
    ,
    \\
    \begin{tikzpicture}[centerzero]
        \draw[multi] (0,-0.4) -- (0,0.4);
        \coupon{0,0}{f};
    \end{tikzpicture}
    \otimes \rightcup :=
    \begin{tikzpicture}[centerzero]
        \draw (0.3,0.7) -- (0.3,0.3) \braiddown (-0.5,-0.4);
        \draw[wipe] (0,-0.7) -- (0,0.7);
        \draw[multi] (0,-0.7) -- (0,0.7);
        \draw[wipe] (-0.2,-0.4) \braidup (0.6,0.1);
        \draw (-0.5,-0.4) to[out=down,in=down,looseness=2] (-0.2,-0.4) \braidup (0.6,0.1) -- (0.6,0.7);
        \coupon{0,0.4}{f};
    \end{tikzpicture}
    =
    \begin{tikzpicture}[centerzero]
        \draw (0.3,0.7) -- (0.3,0.6) \braiddown (-0.5,-0.1);
        \draw[wipe] (0,-0.7) -- (0,0.7);
        \draw[multi] (0,-0.7) -- (0,0.7);
        \draw[wipe] (-0.2,-0.1) \braidup (0.6,0.4);
        \draw (-0.5,-0.1) to[out=down,in=down,looseness=2] (-0.2,-0.1) \braidup (0.6,0.4) -- (0.6,0.7);
        \coupon{0,-0.4}{f};
    \end{tikzpicture}
    \ ,\quad
    \begin{tikzpicture}[centerzero]
        \draw[multi] (0,-0.4) -- (0,0.4);
        \coupon{0,0}{f};
    \end{tikzpicture}
    \otimes \leftcup := q^{-1}t\
    \begin{tikzpicture}[centerzero]
        \draw (-0.5,-0.4) to[out=down,in=down,looseness=2] (-0.2,-0.4) \braidup (0.6,0.1) -- (0.6,0.7);
        \draw[wipe] (0,-0.7) -- (0,0.7);
        \draw[multi] (0,-0.7) -- (0,0.7);
        \draw[wipe] (0.3,0.3) \braiddown (-0.5,-0.4);
        \draw (0.3,0.7) -- (0.3,0.3) \braiddown (-0.5,-0.4);
        \coupon{0,0.4}{f};
    \end{tikzpicture}
    = q^{-1}t\
    \begin{tikzpicture}[centerzero]
        \draw (-0.5,-0.1) to[out=down,in=down,looseness=2] (-0.2,-0.1) \braidup (0.6,0.4) -- (0.6,0.7);
        \draw[wipe] (0,-0.7) -- (0,0.7);
        \draw[multi] (0,-0.7) -- (0,0.7);
        \draw[wipe] (0.3,0.7) -- (0.3,0.6) \braiddown (-0.5,-0.1);
        \draw (0.3,0.7) -- (0.3,0.6) \braiddown (-0.5,-0.1);
        \coupon{0,-0.4}{f};
    \end{tikzpicture}
    \ ,
    \\
    \begin{tikzpicture}[centerzero]
        \draw[multi] (0,-0.4) -- (0,0.4);
        \coupon{0,0}{f};
    \end{tikzpicture}
    \otimes \leftcap :=
    \begin{tikzpicture}[centerzero,yscale=-1]
        \draw (0.3,0.7) -- (0.3,0.3) \braiddown (-0.5,-0.4);
        \draw[wipe] (0,-0.7) -- (0,0.7);
        \draw[multi] (0,-0.7) -- (0,0.7);
        \draw[wipe] (-0.2,-0.4) \braidup (0.6,0.1);
        \draw (-0.5,-0.4) to[out=down,in=down,looseness=2] (-0.2,-0.4) \braidup (0.6,0.1) -- (0.6,0.7);
        \coupon{0,0.4}{f};
    \end{tikzpicture}
    =
    \begin{tikzpicture}[centerzero,yscale=-1]
        \draw (0.3,0.7) -- (0.3,0.6) \braiddown (-0.5,-0.1);
        \draw[wipe] (0,-0.7) -- (0,0.7);
        \draw[multi] (0,-0.7) -- (0,0.7);
        \draw[wipe] (-0.2,-0.1) \braidup (0.6,0.4);
        \draw (-0.5,-0.1) to[out=down,in=down,looseness=2] (-0.2,-0.1) \braidup (0.6,0.4) -- (0.6,0.7);
        \coupon{0,-0.4}{f};
    \end{tikzpicture}
    \ ,\quad
    \begin{tikzpicture}[centerzero]
        \draw[multi] (0,-0.4) -- (0,0.4);
        \coupon{0,0}{f};
    \end{tikzpicture}
    \otimes \rightcap := qt^{-1}\
    \begin{tikzpicture}[centerzero,yscale=-1]
        \draw (-0.5,-0.4) to[out=down,in=down,looseness=2] (-0.2,-0.4) \braidup (0.6,0.1) -- (0.6,0.7);
        \draw[wipe] (0,-0.7) -- (0,0.7);
        \draw[multi] (0,-0.7) -- (0,0.7);
        \draw[wipe] (0.3,0.3) \braiddown (-0.5,-0.4);
        \draw (0.3,0.7) -- (0.3,0.3) \braiddown (-0.5,-0.4);
        \coupon{0,0.4}{f};
    \end{tikzpicture}
    = qt^{-1}\
    \begin{tikzpicture}[centerzero,yscale=-1]
        \draw (-0.5,-0.1) to[out=down,in=down,looseness=2] (-0.2,-0.1) \braidup (0.6,0.4) -- (0.6,0.7);
        \draw[wipe] (0,-0.7) -- (0,0.7);
        \draw[multi] (0,-0.7) -- (0,0.7);
        \draw[wipe] (0.3,0.7) -- (0.3,0.6) \braiddown (-0.5,-0.1);
        \draw (0.3,0.7) -- (0.3,0.6) \braiddown (-0.5,-0.1);
        \coupon{0,-0.4}{f};
    \end{tikzpicture}
    \ ,
\end{gather*}
for $f$ a morphism in $\iqB(q,t)$.

%============================================================================
\section{Equivalence of the disoriented skein and iquantum Brauer categories}
%============================================================================

In this section, we will show that the disoriented skein category $\DS(q,t)$ and the iquantum Brauer category $\iqB(q,t)$ are equivalent as $\OS(q,t)$-module categories.

%------------------------------------------
\subsection{Morphisms of module categories}
%------------------------------------------

Recall the definition of strict right module categories from \cref{subsec:DS}.  Let $(\cC,\otimes,\one)$ be a $\kk$-linear strict monoidal category, and let $\cM$, $\cN$ be strict $\cC$-modules.  A \emph{morphism of $\cC$-modules} from $\cM$ to $\cN$ is a pair $(\bH,\omega)$, where $\bH$ is a $\kk$-linear functor from $\cM$ to $\cN$ and $\omega$ is a natural isomorphism with components
\[
    \omega_{M,C} \colon \bH(M \otimes C) \xrightarrow{\cong} \bH(M) \otimes C,\qquad M \in \cM,\ C \in \cC,
\]
such that the diagram
\begin{equation} \label{buffalo}
    \begin{tikzcd}
        & \bH (M \otimes C \otimes D) \arrow[dl,"\omega_{M \otimes C, D}"'] \ar[dr,"\omega_{M, C \otimes D}"] &
        \\
        \bH(M \otimes C) \otimes D \arrow[rr,"\omega_{M,C} \otimes 1_D"] & & \bH(M) \otimes C \otimes D
    \end{tikzcd}
\end{equation}
commutes for all $M \in \cM$ and $C,D \in \mathcal{C}$.  (See \cite[Def.~7.2.1]{ENGO15} for a more general definition, where the module categories are not required to be strict.)  We say that a morphism of $\cC$-modules is \emph{strict} if $\omega_{M,C}$ is the identity morphism for all $M \in \cM$ and $C \in \cC$.  An \emph{equivalence of $\cC$-modules} is a morphism of $\cC$-modules that is also an equivalence of categories.

%---------------------------------------------------------------------------------------
\subsection{Functor from the disoriented skein category to the iquantum Brauer category}
%---------------------------------------------------------------------------------------

\begin{prop}
    There is a unique strict morphism of $\OS(q,t)$-modules $\bF \colon \DS(q,t) \to \iqB(q,t)$ given on objects by
    \[
        \bF(\one) = \Bobj_0
    \]
    and on morphisms by
    \[
        \bF(\togupdown) = 1_{\Bobj_1} = \bF(\togdownup).
    \]
\end{prop}

\begin{proof}
    Uniqueness is clear, since $\togupdown$ and $\togdownup$ generate $\DS(q,t)$ as an $\OS(q,t)$-module.  For existence, we verify that $\bF$ respects the relations \cref{DStoggles,DScurls}.  We first compute
    \[
        \bF(\posupcross) = \bA(\posupcross)_{\Bobj_0}
        \overset{\cref{meat1}}{=} \poscross.
    \]
    Similarly, using \cref{meat1,meat2,train1}, we have
    \begin{gather*}
        \bF(\posdowncross) = \bF(\posrightcross) = \bF(\posleftcross) = \poscross,\qquad
        \bF(\negupcross) = \bF(\negdowncross) = \bF(\negrightcross) = \bF(\negleftcross) = \negcross,
        \\
        \bF(\togdownup \uparrow) = \thickstrand{2} = \bF(\togdownup \downarrow),\qquad
        \bF(\leftcap) = \Bcap\, .
    \end{gather*}
    Thus,
    \[
        \bF
        \left(
            \begin{tikzpicture}[anchorbase]
                \draw[<-] (-0.2,0.2) to[out=down,in=135] (0.2,-0.3);
                \draw[wipe] (-0.2,-0.3) to[out=45,in=down] (0.2,0.2);
                \draw[->] (-0.2,-0.3) to[out=45,in=down] (0.2,0.2) arc(0:180:0.2);
                \opendot{-0.2,0.2};
            \end{tikzpicture}
        \right)
        =
        \begin{tikzpicture}[anchorbase]
            \draw (-0.2,0.2) to[out=down,in=135] (0.2,-0.3);
            \draw[wipe] (-0.2,-0.3) to[out=45,in=down] (0.2,0.2);
            \draw (-0.2,-0.3) to[out=45,in=down] (0.2,0.2) arc(0:180:0.2);
        \end{tikzpicture}
        \overset{\cref{Bcurl}}{=} q\
        \begin{tikzpicture}[anchorbase]
            \draw (-0.2,-0.3) -- (-0.2,0.1) arc(180:0:0.2) -- (0.2,-0.3);
        \end{tikzpicture}
        = q
        \bF
        \left(
            \begin{tikzpicture}[anchorbase]
                \draw[->] (0.2,-0.3) -- (0.2,0.1) arc(0:180:0.2) -- (-0.2,0);
                \draw[<-] (-0.2,0) -- (-0.2,-0.3);
                \opendot{-0.2,0};
            \end{tikzpicture}
        \right),
    \]
    proving that $\bF$ respects the first relation in \cref{DScurls}.  Similarly,
    \[
        \bF
        \left(
            \begin{tikzpicture}[centerzero]
                \draw[<-] (-0.2,-0.6) -- (-0.2,-0.4) \braidup (0.2,0);
                \draw[wipe] (0.2,-0.6) -- (0.2,-0.4) \braidup (-0.2,0);
                \draw[<-] (0.2,-0.6) -- (0.2,-0.4) \braidup (-0.2,0);
                \draw[->] (-0.2,0) \braidup (0.2,0.4) -- (0.2,0.6);
                \draw[wipe] (0.2,0) \braidup (-0.2,0.4) -- (-0.2,0.6);
                \draw[->] (0.2,0) \braidup (-0.2,0.4) -- (-0.2,0.6);
                \opendot{-0.2,-0.4};
                \opendot{-0.2,0};
            \end{tikzpicture}
        \right)
        =
        \begin{tikzpicture}[centerzero]
            \draw (0.2,-0.6) -- (0.2,-0.4) \braidup (-0.2,0) \braidup (0.2,0.4) -- (0.2,0.6);
            \draw[wipe] (-0.2,-0.4) \braidup (0.2,0) \braidup (-0.2,0.4);
            \draw (-0.2,-0.6) -- (-0.2,-0.4) \braidup (0.2,0) \braidup (-0.2,0.4) -- (-0.2,0.6);
        \end{tikzpicture}
        \overset{\cref{Bbraid}}{=}
        \begin{tikzpicture}[centerzero]
            \draw (-0.2,-0.6) -- (-0.2,-0.4) \braidup (0.2,0) \braidup (-0.2,0.4) -- (-0.2,0.6);
            \draw[wipe] (0.2,-0.6) -- (0.2,-0.4) \braidup (-0.2,0) \braidup (0.2,0.4) -- (0.2,0.6);
            \draw (0.2,-0.6) -- (0.2,-0.4) \braidup (-0.2,0) \braidup (0.2,0.4) -- (0.2,0.6);
        \end{tikzpicture}
        = \bF
        \left(
            \begin{tikzpicture}[centerzero]
                \draw[<-] (-0.2,-0.6) -- (-0.2,-0.4) \braidup (0.2,0);
                \draw[wipe] (0.2,-0.6) -- (0.2,-0.4) \braidup (-0.2,0);
                \draw[<-] (0.2,-0.6) -- (0.2,-0.4) \braidup (-0.2,0);
                \draw[->] (-0.2,0) \braidup (0.2,0.4) -- (0.2,0.6);
                \draw[wipe] (0.2,0) \braidup (-0.2,0.4) -- (-0.2,0.6);
                \draw[->] (0.2,0) \braidup (-0.2,0.4) -- (-0.2,0.6);
                \opendot{-0.2,0.4};
                \opendot{-0.2,0};
            \end{tikzpicture}
        \right),
    \]
    proving that $\bF$ respects the last relation in \cref{DStoggles}.  The remaining relations follow by analogous arguments and are omitted for brevity.
\end{proof}

For $\lambda \in \word$, let $\ell(\lambda)$ denote the length of $\lambda$.  The following result allows us to easily compute the image under $\bF$ of any morphism in $\DS(q,t)$.

\begin{lem}
    For all $\lambda,\mu \in \word$, we have
    \begin{gather} \label{water1}
        \bF \left( \thickstrand{\lambda} \posupcross \thickstrand{\mu} \right)
        = \bF \left( \thickstrand{\lambda} \posdowncross \thickstrand{\mu} \right)
        = \bF \left( \thickstrand{\lambda} \posrightcross \thickstrand{\mu} \right)
        = \bF \left( \thickstrand{\lambda} \posleftcross \thickstrand{\mu} \right)
        = \thickstrand{\ell(\lambda)} \poscross \thickstrand{\ell(\mu)},
        \\ \label{water2}
        \bF \left( \thickstrand{\lambda} \negupcross \thickstrand{\mu} \right)
        = \bF \left( \thickstrand{\lambda} \negdowncross \thickstrand{\mu} \right)
        = \bF \left( \thickstrand{\lambda} \negrightcross \thickstrand{\mu} \right)
        = \bF \left( \thickstrand{\lambda} \negleftcross \thickstrand{\mu} \right)
        = \thickstrand{\ell(\lambda)} \negcross \thickstrand{\ell(\mu)},
        \\ \label{water3}
        \bF \left( \thickstrand{\lambda} \rightcup \thickstrand{\mu} \right)
        =
        \begin{tikzpicture}[centerzero,cscale]
            \draw[multi] (0,-0.6) -- (0,-0.2) \braidup (-0.4,0.2);
            \draw[wipe] (-0.4,-0.2) \braidup (0,0.2);
            \draw (-0.4,0.6) \braiddown (-0.8,0.2)  -- (-0.8,-0.2) to[out=down,in=down,looseness=2] (-0.4,-0.2) \braidup (0,0.2) -- (0,0.6);
            \draw[wipe] (-0.4,0.2) \braidup (-0.8,0.6);
            \draw[multi] (-0.4,0.2) \braidup (-0.8,0.6) \toplabel{\ell(\lambda)};
            \draw[multi] (0.4,-0.6) -- (0.4,0.6) \toplabel{\ell(\mu)};
        \end{tikzpicture}
        ,\qquad
        \bF \left( \thickstrand{\lambda} \leftcup \thickstrand{\mu} \right)
        = q^{-1} t
        \begin{tikzpicture}[centerzero,cscale]
            \draw[multi] (-0.4,0.2) \braidup (-0.8,0.6) \toplabel{\ell(\lambda)};
            \draw[wipe] (-0.4,0.6) \braiddown (-0.8,0.2);
            \draw (-0.4,0.6) \braiddown (-0.8,0.2)  -- (-0.8,-0.2) to[out=down,in=down,looseness=2] (-0.4,-0.2) \braidup (0,0.2) -- (0,0.6);
            \draw[wipe] (0,-0.2) \braidup (-0.4,0.2);
            \draw[multi] (0,-0.6) -- (0,-0.2) \braidup (-0.4,0.2);
            \draw[multi] (0.4,-0.6) -- (0.4,0.6) \toplabel{\ell(\mu)};
        \end{tikzpicture}
        \ ,
        \\ \label{water4}
        \bF \left( \thickstrand{\lambda} \leftcap \thickstrand{\mu} \right)
        =
        \begin{tikzpicture}[centerzero,cscale]
            \draw[multi] (-0.4,-0.2) \braidup (0,0.2) -- (0,0.6);
            \draw[wipe] (-0.4,0.2) \braiddown (0,-0.2);
            \draw (-0.4,-0.6) \braidup (-0.8,-0.2) -- (-0.8,0.2) to[out=up,in=up,looseness=2] (-0.4,0.2) \braiddown (0,-0.2) -- (0,-0.6);
            \draw[wipe] (-0.8,-0.6) \braidup (-0.4,-0.2);
            \draw[multi] (-0.8,-0.6) \botlabel{\ell(\lambda)} \braidup (-0.4,-0.2);
            \draw[multi] (0.4,-0.6) \botlabel{\ell(\mu)} -- (0.4,0.6);
        \end{tikzpicture}
        \ ,\qquad
        \bF \left( \thickstrand{\lambda} \rightcap \thickstrand{\mu} \right)
        = qt^{-1}\,
        \begin{tikzpicture}[centerzero,cscale]
            \draw[multi] (-0.8,-0.6) \botlabel{\ell(\lambda)} \braidup (-0.4,-0.2);
            \draw[wipe] (-0.4,-0.6) \braidup (-0.8,-0.2);
            \draw (-0.4,-0.6) \braidup (-0.8,-0.2) -- (-0.8,0.2) to[out=up,in=up,looseness=2] (-0.4,0.2) \braiddown (0,-0.2) -- (0,-0.6);
            \draw[wipe] (-0.4,-0.2) \braidup (0,0.2);
            \draw[multi] (-0.4,-0.2) \braidup (0,0.2) -- (0,0.6);
            \draw[multi] (0.4,-0.6) \botlabel{\ell(\mu)} -- (0.4,0.6);
        \end{tikzpicture}
        \ ,
        \\ \label{water5}
        \bF
        \left(
            \begin{tikzpicture}[centerzero]
                \draw[multi] (-0.2,-0.3) \botlabel{\lambda} -- (-0.2,0.3);
                \draw[->] (0.2,0) -- (0.2,0.3);
                \draw[->] (0.2,0) -- (0.2,-0.3);
                \draw[multi] (0.6,-0.3) \botlabel{\mu} -- (0.6,0.3);
                \opendot{0.2,0};
            \end{tikzpicture}
        \right)
        = \thickstrand{\ell(\lambda)+\ell(\mu)+1}
        = \bF
        \left(
            \begin{tikzpicture}[centerzero]
                \draw[multi] (-0.2,-0.3) \botlabel{\lambda} -- (-0.2,0.3);
                \draw[<-] (0.2,0) -- (0.2,0.3);
                \draw[<-] (0.2,0) -- (0.2,-0.3);
                \draw[multi] (0.6,-0.3) \botlabel{\mu} -- (0.6,0.3);
                \opendot{0.2,0};
            \end{tikzpicture}
        \right).
    \end{gather}
\end{lem}

\begin{proof}
    Equations \cref{water1,water2,water3,water4} follow from \cref{meat1,meat2,train1,train2}, while \cref{water5} follows from \cref{toggy,water1,water2}.
\end{proof}

%---------------------------------------------------------------------------------------
\subsection{Functor from the iquantum Brauer category to the disoriented skein category}
%---------------------------------------------------------------------------------------

For $r \in \N$, define the following morphisms in $\OS(q,t)$ and $\DS(q,t)$:
\[
    \begin{tikzpicture}[centerzero]
        \draw[multi,->] (0,-0.2) \botlabel{r} -- (0,0.2);
    \end{tikzpicture}
    :=
    \underbrace{
        \upstrand \cdots \upstrand
    }_r
    \ ,\qquad
    \begin{tikzpicture}[centerzero]
        \draw[multi,<-] (0,-0.2) \botlabel{r} -- (0,0.2);
    \end{tikzpicture}
    :=
    \underbrace{
        \downstrand \cdots \downstrand
    }_r
    \ .
\]

\begin{prop}
    There is a unique $\kk$-linear functor $\bG \colon \iqB(q,t)\to \DS(q,t)$ given by
    \begin{equation}\label{eq:bGi}
        \begin{gathered}
            \Bobj_r \mapsto \upobj^{\otimes r},
            \qquad
            \begin{tikzpicture}[centerzero]
                \draw (0.2,-0.3) -- (0.2,0.1) arc(0:180:0.2) -- (-0.2,0) -- (-0.2,-0.3);
                \draw[multi] (0.6,-0.3) \botlabel{r} -- (0.6,0.4);
            \end{tikzpicture}
            \mapsto
            \begin{tikzpicture}[centerzero]
                \draw[->] (0.2,-0.3) -- (0.2,0.1) arc(0:180:0.2) -- (-0.2,0);
                \draw[<-] (-0.2,0) -- (-0.2,-0.3);
                \opendot{-0.2,0};
                \draw[multi,->] (0.6,-0.3) \botlabel{r} -- (0.6,0.4);
            \end{tikzpicture}
            ,\qquad
            \begin{tikzpicture}[centerzero]
                \draw (0.2,0.3) -- (0.2,-0.1) arc(360:180:0.2) -- (-0.2,0) -- (-0.2,0.3);
                \draw[multi] (0.6,-0.4) \botlabel{r} -- (0.6,0.3);
            \end{tikzpicture}
            \mapsto
            \begin{tikzpicture}[centerzero]
                \draw[<-] (0.2,0.3) -- (0.2,-0.1) arc(360:180:0.2) -- (-0.2,0);
                \draw[->] (-0.2,0) -- (-0.2,0.3);
                \opendot{-0.2,0};
                \draw[multi,->] (0.6,-0.4) \botlabel{r} -- (0.6,0.3);
            \end{tikzpicture}
            ,
            \\
           \thickstrand{r}\  \poscross\ \thickstrand{s}\mapsto \thickstrandup{r}\  \posupcross\ \thickstrandup{s},
           \qquad
           \thickstrand{r}\  \negcross\ \thickstrand{s}\mapsto \thickstrandup{r}\  \negupcross\ \thickstrandup{s}.
       \end{gathered}
   \end{equation}
   for all $r,s \in \N$.
\end{prop}

\begin{proof}
    It suffices to check that relations \cref{Bbraid,Bskein,Bcurl,Bhumps,Bcommute} are preserved under $\bG$.  The relations \cref{Bbraid}, \cref{Bskein}, and \cref{Bcommute} follow from \cref{Bbraid}, \cref{oskein}, and \cref{snow}.   The first and last relations in \cref{Bcurl} follow from \cref{DStoggles,ocurlbub}, while the second and third relations in \cref{Bcurl} follow from \cref{DScurls}.

    For the first relation in \cref{Bhumps}, we compute
    \begin{multline*}
        \bG
        \left(
            \begin{tikzpicture}[centerzero,cscale]
                \draw[multi] (0,-1.2) \botlabel{r} -- (0,1.2);
                \draw (-0.4,-1.2) -- (-0.4,0) \braidup (-0.8,0.4) \braidup (-1.2,0.8) to[out=up,in=up,looseness=2] (-1.6,0.8) -- (-1.6,0);
                \draw[wipe] (-0.4,1.2) -- (-0.4,0.4) \braiddown (-0.8,0) -- (-0.8,-0.4) \braiddown (-1.2,-0.8) to[out=down,in=down,looseness=2] (-1.6,-0.8) -- (-1.6,-0.4) \braidup (-1.2,0) -- (-1.2,0.4) \braidup (-0.8,0.8) -- (-0.8,1.2);
                \draw (-0.4,1.2) -- (-0.4,0.4) \braiddown (-0.8,0) -- (-0.8,-0.4) \braiddown (-1.2,-0.8) to[out=down,in=down,looseness=2] (-1.6,-0.8) -- (-1.6,-0.4) \braidup (-1.2,0) -- (-1.2,0.4) \braidup (-0.8,0.8) -- (-0.8,1.2);
                \draw[wipe] (-1.6,0) \braiddown (-1.2,-0.4) \braiddown (-0.8,-0.8);
                \draw (-1.6,0) \braiddown (-1.2,-0.4) \braiddown (-0.8,-0.8) -- (-0.8,-1.2);
            \end{tikzpicture}
        \right)
        =
        \begin{tikzpicture}[centerzero,cscale]
            \draw[multi] (0,-1.2) \botlabel{r} -- (0,1.2);
            \draw[->] (-0.4,-1.2) -- (-0.4,0) \braidup (-0.8,0.4) \braidup (-1.2,0.8) to[out=up,in=up,looseness=2] (-1.6,0.8) -- (-1.6,0.4);
            \draw[wipe] (-0.4,1.2) -- (-0.4,0.4) \braiddown (-0.8,0) -- (-0.8,-0.4) \braiddown (-1.2,-0.8) to[out=down,in=down,looseness=2] (-1.6,-0.8) -- (-1.6,-0.4) \braidup (-1.2,0) -- (-1.2,0.4) \braidup (-0.8,0.8) -- (-0.8,1.2);
            \draw[<->] (-0.4,1.2) -- (-0.4,0.4) \braiddown (-0.8,0) -- (-0.8,-0.4) \braiddown (-1.2,-0.8) to[out=down,in=down,looseness=2] (-1.6,-0.8) -- (-1.6,-0.4) \braidup (-1.2,0) -- (-1.2,0.4) \braidup (-0.8,0.8) -- (-0.8,1.2);
            \draw[wipe] (-1.6,0.4) -- (-1.6,0) \braiddown (-1.2,-0.4) \braiddown (-0.8,-0.8);
            \draw[<-] (-1.6,0.4) -- (-1.6,0) \braiddown (-1.2,-0.4) \braiddown (-0.8,-0.8) -- (-0.8,-1.2);
            \opendot{-1.6,0.4};
            \opendot{-1.6,-0.8};
        \end{tikzpicture}
        \overset{\cref{togcupcap}}{\underset{\cref{togcross}}{=}} q^{-1} t\
        \begin{tikzpicture}[centerzero,cscale]
            \draw[multi] (0,-1.6) \botlabel{r} -- (0,1.6);
            \draw[<->] (-0.4,1.6) -- (-0.4,0.4) \braiddown (-0.8,0) -- (-0.8,-0.4) \braiddown (-1.2,-0.8) to[out=down,in=down,looseness=2] (-1.6,-0.8) -- (-1.6,-0.4) \braidup (-1.2,0) -- (-1.2,0.4) \braidup (-0.8,0.8) -- (-0.8,1.6);
            \draw[->] (-0.4,-1.6) -- (-0.4,-1.2);
            \draw[wipe] (-0.4,0) \braidup (-0.8,0.4) \braidup (-1.2,0.8);
            \draw[<-] (-0.4,-1.2) -- (-0.4,0) \braidup (-0.8,0.4) \braidup (-1.2,0.8) to[out=up,in=up,looseness=2] (-1.6,0.8) -- (-1.6,0.4);
            \draw[wipe] (-1.6,0.4) -- (-1.6,0) \braiddown (-1.2,-0.4) \braiddown (-0.8,-0.8);
            \draw (-1.6,0.4) -- (-1.6,0) \braiddown (-1.2,-0.4) \braiddown (-0.8,-0.8) -- (-0.8,-1.6);
            \opendot{-0.4,-1.2};
            \opendot{-1.6,-0.8};
        \end{tikzpicture}
        \\
        \overset{\substack{\cref{togcupcap} \\ \cref{togcross}}}{\underset{\cref{intertog}}{=}}
        \begin{tikzpicture}[centerzero,cscale]
            \draw[multi] (0,-1.6) \botlabel{r} -- (0,1.6);
            \draw[->] (-1.2,-0.8) to[out=down,in=down,looseness=2] (-1.6,-0.8) -- (-1.6,-0.4) \braidup (-1.2,0) -- (-1.2,0.4) \braidup (-0.8,0.8) -- (-0.8,1.6);
            \draw[->] (-0.4,-1.6) -- (-0.4,-1.2);
            \draw[wipe] (-0.4,0) \braidup (-0.8,0.4) \braidup (-1.2,0.8);
            \draw[<-] (-0.4,-1.2) -- (-0.4,0) \braidup (-0.8,0.4) \braidup (-1.2,0.8) to[out=up,in=up,looseness=2] (-1.6,0.8) -- (-1.6,0.4);
            \draw[wipe] (-1.6,0.4) -- (-1.6,0) \braiddown (-1.2,-0.4) \braiddown (-0.8,-0.8);
            \draw (-1.6,0.4) -- (-1.6,0) \braiddown (-1.2,-0.4) \braiddown (-0.8,-0.8) -- (-0.8,-1.6);
            \draw[wipe] (-0.4,1.6) -- (-0.4,0.4) \braiddown (-0.8,0) -- (-0.8,-0.4) \braiddown (-1.2,-0.8);
            \draw[<-] (-0.4,1.6) -- (-0.4,0.4) \braiddown (-0.8,0) -- (-0.8,-0.4) \braiddown (-1.2,-0.8);
            \opendot{-0.4,-1.2};
            \opendot{-0.4,1.2};
        \end{tikzpicture}
        \overset{\cref{venom1}}{\underset{\cref{venom2}}{=}}
        \begin{tikzpicture}[centerzero,cscale]
            \draw[multi] (0,-1.2) \botlabel{r} -- (0,1.2);
            \draw[->] (-0.4,-1.2) -- (-0.4,-0.8);
            \draw[<-] (-0.4,-0.8) -- (-0.4,-0.6) to[out=up,in=up,looseness=2] (-0.8,-0.6) -- (-0.8,-1.2);
            \draw[<->] (-0.4,1.2) -- (-0.4,0.8) -- (-0.4,0.6) to[out=down,in=down,looseness=2] (-0.8,0.6) -- (-0.8,1.2);
            \opendot{-0.4,-0.8};
            \opendot{-0.4,0.8};
        \end{tikzpicture}
        \overset{\cref{togcupcap}}{=}
        \begin{tikzpicture}[centerzero,cscale]
            \draw[multi] (0,-1.2) \botlabel{r} -- (0,1.2);
            \draw[->] (-0.4,-1.2) -- (-0.4,-0.8) -- (-0.4,-0.6) to[out=up,in=up,looseness=2] (-0.8,-0.6) -- (-0.8,-0.8);
            \draw[<-] (-0.8,-0.8) -- (-0.8,-1.2);
            \draw[<->] (-0.4,1.2) -- (-0.4,0.8) -- (-0.4,0.6) to[out=down,in=down,looseness=2] (-0.8,0.6) -- (-0.8,1.2);
            \opendot{-0.8,-0.8};
            \opendot{-0.8,0.8};
        \end{tikzpicture}
        =
        \bG
        \left(
            \begin{tikzpicture}[centerzero,cscale]
                \draw[multi] (0,-1.2) \botlabel{r} -- (0,1.2);
                \draw (-0.4,-1.2) -- (-0.4,-0.8) -- (-0.4,-0.6) to[out=up,in=up,looseness=2] (-0.8,-0.6) -- (-0.8,-0.8) -- (-0.8,-1.2);
                \draw (-0.4,1.2) -- (-0.4,0.8) -- (-0.4,0.6) to[out=down,in=down,looseness=2] (-0.8,0.6) -- (-0.8,1.2);
            \end{tikzpicture}
        \right).
    \end{multline*}
    Hence, the first relation in \cref{Bhumps} is preserved under $\bG$. The other relations in \cref{Bhumps} follow similarly.
\end{proof}

\begin{rem} \label{laser}
    Note that $\bG$ is \emph{not} a \emph{strict} morphism of $\OS(q,t)$-modules.  For example,
    \[
        \bG( \Bobj_0 \otimes \downobj)
        = \bG(\Bobj_1)
        = \upobj
        \ne \downobj
        = \bG(\Bobj_0) \otimes \downobj.
    \]
    Nevertheless, we will see in \cref{ginger} that $\bG$ can be made into a morphism of $\OS(q,t)$-modules.
\end{rem}

%---------------------------------------------------------
\subsection{Equivalence of categories\label{subsec:equiv}}
%---------------------------------------------------------

In this subsection, we show that the functors $\bG$ and $\bF$ are equivalences of module categories between $\iqB(q,t)$ and $\DS(q,t)$.  We first observe that $\bF \bG=\text{id}_{\iqB(q,t)}$ by definition.  In the other direction, we construct a natural isomorphism $\eta\colon \id_{\DS(q,t)} \Rightarrow \bG \bF$.

Note that
\[
    \bG \bF(\lambda) = \upobj^{\ell(\lambda)},\qquad \lambda \in \word.
\]
For any $\lambda \in \word$, we construct an isomorphism   $\eta_\lambda\colon \lambda \to \bG \bF(\lambda)$ in $\DS(q,t)$ recursively by
\begin{equation} \label{eta}
    \eta_\varnothing := 1_\varnothing,\qquad
    \eta_\upobj := \upstrand,\qquad
    \eta_\downobj := \togupdown,\qquad
    \eta_{\lambda \mu}
    :=
    \begin{tikzpicture}[anchorbase]
        \draw[multi,->] (0,-0.5) -- (0,0.5);
        \coupon{0,0}{\eta_\lambda};
        \draw[multi,->] (0.6,-0.5) -- (0.6,0.5);
        \coupon{0.6,0}{\eta_\mu};
    \end{tikzpicture}
    ,\qquad
    \lambda,\mu \in \word.
\end{equation}
For example, we have $\eta_{\upobj \downobj \downobj \upobj} = \upstrand \togupdown \togupdown \upstrand$ and $\eta_{\upobj \downobj \downobj \upobj}^{-1} = \upstrand \togdownup \togdownup \upstrand$.

\begin{theo}\label{equivthm}
    The functor $\bF \colon \DS(q,t) \to \iqB(q,t)$ is a strict equivalence of $\OS(q,t)$-modules.
\end{theo}

\begin{proof}
    To see that $\eta$ is indeed a natural transformation, we must show that
    \begin{equation} \label{peanut}
        \begin{tikzpicture}[centerzero]
            \draw[multi,->] (0,-0.7) \botlabel{\lambda} -- (0,0.7) \toplabel{\bG \bF(\mu)};
            \coupon{0,-0.35}{f};
            \coupon{0,0.2}{\eta_\mu};
        \end{tikzpicture}
        =
        \begin{tikzpicture}[centerzero]
            \draw[multi,->] (0,-0.7) \botlabel{\lambda} -- (0,0.7) \toplabel{\bG \bF(\mu)};
            \coupon{0,-0.35}{\eta_\lambda};
            \coupon{0,0.2}{\bG \bF(f)};
        \end{tikzpicture}
    \end{equation}
    for all morphisms $f \colon \lambda \to \mu$ in $\DS(q,t)$.  It suffices to check this for
    \[
        f \in
        \left\{
            \togupdown \thickstrand{\lambda},
            \togdownup\thickstrand{\lambda},
            \thickstrand{\lambda}\posupcross\thickstrand{\mu},
            \thickstrand{\lambda}\negupcross\thickstrand{\mu},
            \thickstrand{\lambda}\rightcup\thickstrand{\mu},
            \thickstrand{\lambda}\rightcap\thickstrand{\mu},
            \thickstrand{\lambda}\leftcup\thickstrand{\mu},
            \thickstrand{\lambda}\leftcap\thickstrand{\mu}
            : \lambda,\mu \in \word
        \right\},
    \]
    since these generate $\DS(q,t)$ as a $\kk$-linear category.  For

    For $f = \togupdown \thickstrand{\lambda}$, we have $\bG \bF (f) = \thickstrandup{\ell(\lambda)}$, and so \cref{peanut} follows immediately from the definition of $\eta$.  The case $f = \togdownup \thickstrand{\lambda}$ is similar, using \cref{DStoggles}.

    For $f = \thickstrand{\lambda}\posupcross\thickstrand{\mu}$, we have
    \[
        \begin{tikzpicture}[centerzero]
            \draw[multi,->] (0,-0.7) -- (0,0.7);
            \coupon{0,-0.35}{f};
            \coupon{0,0.2}{\eta_{\lambda \upobj \upobj \mu}};
        \end{tikzpicture}
        =
        \begin{tikzpicture}[centerzero]
            \draw[multi,->] (0,-0.7) -- (0,0.7);
            \draw[->] (0.6,-0.7) -- (0.6,-0.6) \braidup (0.3,-0.2) -- (0.3,0.7);
            \draw[wipe] (0.3,-0.7) -- (0.3,-0.6) \braidup (0.6,-0.2) -- (0.6,0.7);
            \draw[->] (0.3,-0.7) -- (0.3,-0.6) \braidup (0.6,-0.2) -- (0.6,0.7);
            \draw[multi,->] (0.9,-0.7) -- (0.9,0.7);
            \coupon{0,0.2}{\eta_\lambda};
            \coupon{0.9,0.2}{\eta_\mu};
        \end{tikzpicture}
        \overset{\cref{intertog}}{=}
        \begin{tikzpicture}[centerzero]
            \draw[multi,->] (0,-0.7) -- (0,0.7);
            \draw[->] (0.6,-0.7) -- (0.6,0.2) \braidup (0.3,0.6) -- (0.3,0.7);
            \draw[wipe] (0.3,-0.7) -- (0.3,0.2) \braidup (0.6,0.6) -- (0.6,0.7);
            \draw[->] (0.3,-0.7) -- (0.3,0.2) \braidup (0.6,0.6) -- (0.6,0.7);
            \draw[multi,->] (0.9,-0.7) -- (0.9,0.7);
            \coupon{0,-0.3}{\eta_\lambda};
            \coupon{0.9,-0.3}{\eta_\mu};
        \end{tikzpicture}
        =
        \begin{tikzpicture}[centerzero]
            \draw[multi,->] (0,-0.7) -- (0,0.7);
            \coupon{0,-0.35}{\eta_{\lambda \upobj \upobj \mu}};
            \coupon{0,0.2}{\bG \bF(f)};
        \end{tikzpicture}
        \ .
    \]
    The case $f = \thickstrand{\lambda}\negupcross\thickstrand{\mu}$ is analogous.

    When $f = \thickstrand{\lambda}\rightcup\thickstrand{\mu}$,
    \[
        \begin{tikzpicture}[centerzero]
            \draw[multi,->] (0,-0.9) -- (0,0.9);
            \coupon{0,-0.4}{\eta_{\lambda \mu}};
            \coupon{0,0.3}{\bG \bF(f)};
        \end{tikzpicture}
        =
        \begin{tikzpicture}[centerzero]
            \draw[<-] (0.3,0.9) -- (0.3,0.8) \braiddown (-0.6,0) -- (-0.6,-0.1) to[out=down,in=down,looseness=2] (-0.3,-0.1);
            \draw[wipe] (0,-0.9) -- (0,0.9);
            \draw[multi,->] (0,-0.9) -- (0,0.9);
            \draw[wipe] (-0.3,-0.1) to[out=up,in=down] (0.6,0.7) -- (0.6,0.9);
            \draw[->] (-0.3,-0.1) to[out=up,in=down] (0.6,0.7) -- (0.6,0.9);
            \opendot{-0.6,0};
            \draw[multi,->] (0.9,-0.9) -- (0.9,0.9);
            \coupon{0,-0.5}{\eta_\lambda};
            \coupon{0.9,-0.5}{\eta_\mu};
        \end{tikzpicture}
        \overset{\cref{togcross}}{=}
        \begin{tikzpicture}[centerzero]
            \draw[multi,->] (0,-0.9) -- (0,0.9);
            \draw[wipe] (0.3,0.9) -- (0.3,0.8) \braiddown (-0.6,0) -- (-0.6,-0.1) to[out=down,in=down,looseness=2] (-0.3,-0.1) to[out=up,in=down] (0.6,0.7) -- (0.6,0.9);
            \draw[<->] (0.3,0.9) -- (0.3,0.8) \braiddown (-0.6,0) -- (-0.6,-0.1) to[out=down,in=down,looseness=2] (-0.3,-0.1) to[out=up,in=down] (0.6,0.7) -- (0.6,0.9);
            \opendot{0.28,0.7};
            \draw[multi,->] (0.9,-0.9) -- (0.9,0.9);
            \coupon{0,-0.5}{\eta_\lambda};
            \coupon{0.9,-0.5}{\eta_\mu};
        \end{tikzpicture}
        \overset{\cref{venom1}}{\underset{\cref{venom2}}{=}}
        \begin{tikzpicture}[centerzero]
            \draw[multi,->] (0,-0.9) -- (0,0.9);
            \draw[<->] (0.3,0.9) -- (0.3,0.4) to[out=down,in=down,looseness=2] (0.6,0.4) -- (0.6,0.9);
            \opendot{0.3,0.6};
            \draw[multi,->] (0.9,-0.9) -- (0.9,0.9);
            \coupon{0,-0.5}{\eta_\lambda};
            \coupon{0.9,-0.5}{\eta_\mu};
        \end{tikzpicture}
        \overset{\cref{intertog}}{=}
        \begin{tikzpicture}[centerzero]
            \draw[multi,->] (-0.1,-0.9) -- (-0.1,0.9);
            \draw[<->] (0.3,0.9) -- (0.3,-0.4) to[out=down,in=down,looseness=2] (0.6,-0.4) -- (0.6,0.9);
            \opendot{0.3,0.3};
            \draw[multi,->] (1,-0.9) -- (1,0.9);
            \coupon{-0.1,0.3}{\eta_\lambda};
            \coupon{1,0.3}{\eta_\mu};
        \end{tikzpicture}
        =
        \begin{tikzpicture}[centerzero]
            \draw[multi,->] (0,-0.9) -- (0,0.9);
            \coupon{0,-0.4}{f};
            \coupon{0,0.3}{\eta_{\lambda \downobj \upobj \mu}};
        \end{tikzpicture}
        \ .
    \]
    The remaining cases are analogous.

    Since the components of $\eta$ are isomorphisms by \cref{dubtog}, it follows that $\eta \colon \id_{\DS(q,t)} \Rightarrow \bG \bF$ is a natural isomorphism.  Since $\bF \bG = \id_{\iqB(q,t)}$, we conclude that $\bG$ and $\bF$ are equivalences of categories.
\end{proof}

\begin{cor} \label{ginger}
    We have an equivalence of $\OS(q,t)$-modules $(\bG,\omega)$ from $\iqB(q,t)$ to $\DS(q,t)$, where $\omega$ has components
    \[
        \omega_{\Bobj_r,\lambda} \colon \bG(\Bobj_r \otimes \lambda)
        \xrightarrow{\eta_{\bG(\Bobj_r) \otimes \lambda}^{-1}}
        \bG(\Bobj_r) \otimes \lambda.
    \]
    for $r \in \N$, $\lambda \in \word$.
\end{cor}

\begin{proof}
    The fact that $\omega$ is a natural isomorphism follows from the fact that $\eta$ is.  It remains to show that the analogue of diagram \cref{buffalo} commutes.  To verify this, we compute that
    \[
        (\omega_{\Bobj_r,\lambda} \otimes 1_\mu) \circ \omega_{\Bobj_r \otimes \lambda, \mu}
        = (\eta_{\bG(\Bobj_r) \otimes \lambda}^{-1} \otimes 1_\mu) \circ \eta_{\bG(\Bobj_r \otimes \lambda) \otimes \mu}^{-1}
        \overset{\cref{eta}}{=} \eta_{\bG(\Bobj_r) \otimes \lambda \otimes \mu}^{-1}
        = \omega_{\Bobj_r,\lambda \otimes \mu},
    \]
    as desired.
\end{proof}

%=============================================
\section{The iquantum enveloping superalgebra}
%=============================================

In this section we collect some important facts about the quantum symmetric pairs of interest to us, together with their representation theory.  Throughout this section, we work over the field $\kk = \C(q)$, where $q$ is an indeterminate.

When working with superspaces, we denote the parity of a homogeneous element $v$ by $\bar{v} \in \Z_2$.  When we write expressions involving $\bar{v}$, we implicitly assume that $v$ is homogeneous.  For a statement $P$, we define
\[
    \delta_P
    :=
    \begin{cases}
        1 & \text{if $P$ is true}, \\
        0 & \text{if $P$ is false}.
    \end{cases}
\]
In particular, we have the usual Kronecker delta $\delta_{i,j} := \delta_{i=j}$.   

%----------------------------------------------------------------
\subsection{The quantum enveloping superalgebra\label{subsec:Uq}}
%----------------------------------------------------------------

Fix $m,n \in \N$ and let
\[
    \Vset = \Vset(m|n) = \{m,m-1,\dotsc,1-n\}.
\]
Define a parity function
\[
    p \colon \Vset \to \{0,1\} \subseteq \Z,\qquad
    p(i) =
    \begin{cases}
        1 & \text{if } i \le 0, \\
        0 & \text{if } i > 0.
    \end{cases}
\]
The general linear Lie superalgebra $\fg = \fgl(m|n,\C)$ has a basis given by the unit matrices $E_{ij}$, $i,j \in \Vset$, where the parity of $E_{ij}$ is 
\[
    \bar{E}_{ij} = p(i) + p(j) \mod 2.
\]
Here, and in what follows, parities are always considered modulo $2$.

Let $\fh$ be the Cartan subalgebra of $\fg$ consisting of diagonal matrices.  Then the dual space $\fh^*$ has basis $\epsilon_i$, $i \in \Vset$, given by $\epsilon_i(E_{jj}) = \delta_{ij}$.  We define a symmetric bilinear form on the weight lattice $P = \bigoplus_{i \in \Vset} \Z \epsilon_i$ by
\[
    (\epsilon_i, \epsilon_j) = (-1)^{p(i)} \delta_{ij}.
\]
We set the parity of $\epsilon_i$ to be $\overline{\epsilon_i} = p(i)$.  We then extend parity to a homomorphism of additive groups $P \to \Z_2$, $\lambda \mapsto \bar{\lambda}$.

Let
\[
    I = \{ i \in \Z : 1-n \le i \le m-1 \}
\]
and choose the set of simple roots $\Pi = \{ \alpha_i : i \in I \}$, where
\[
    \alpha_i = \epsilon_i - \epsilon_{i+1}, \qquad i \in I.
\]
It follows that $\overline{\alpha_0} = 1$ and $\overline{\alpha_i} = 0$ for $i \in I$, $i \ne 0$.  The Dynkin diagram associated to this choice of simple roots is:
\begin{equation} \label{Dynkin}
    \begin{tikzpicture}[anchorbase]
        \draw (0,0) -- (1,0);
        \node at (1.5,0) {$\cdots$};
        \draw (2,0) -- (3,0) -- (4,0) -- (5,0) -- (6,0);
        \node at (6.5,0) {$\cdots$};
        \draw (7,0) -- (8,0);
        \filldraw[draw=black,fill=white] (0,0) circle (0.1);
        \filldraw[draw=black,fill=white] (3,0) circle (0.1);
        \filldraw[draw=black,fill=white] (4,0) circle (0.1);
        \filldraw[draw=black,fill=white] (5,0) circle (0.1);
        \filldraw[draw=black,fill=white] (8,0) circle (0.1);
        \draw (3.93,-0.07) -- (4.07,0.07);
        \draw (4.07,-0.07) -- (3.93,0.07);
        \node at (0,-0.3) {\strandlabel{1-n}};
        \node at (3,-0.3) {\strandlabel{-1}};
        \node at (4,-0.3) {\strandlabel{0}};
        \node at (5,-0.3) {\strandlabel{1}};
        \node at (8,-0.3) {\strandlabel{m-1}};
    \end{tikzpicture}
\end{equation}
In the above diagram, the crossed node corresponds to an odd isotropic simple root and the other nodes correspond to even simple roots.  We have the coweight lattice $P^\vee = \bigoplus_{i \in \Vset} \Z \epsilon_i^\vee$, with the pairing
\[
    \langle\ ,\ \rangle \colon P^\vee \times P \to \Z,\qquad
    \langle \epsilon_i^\vee, \epsilon_j \rangle = \delta_{i,j}.
\]
The set of simple coroots is $\Pi^\vee = \{h_i : i \in I\}$, where
\[
    h_i = \epsilon_i^\vee - (-1)^{\delta_{i,0}} \epsilon_{i+1}^\vee,\qquad i \in I.
\]
We set
\[
    a_{ij} = \langle h_i, \alpha_j \rangle,
    \qquad i,j \in I.
\]

We let $X = \Z \Pi$ and $Y = \Z \Pi^\vee$ denote the root lattice and coroot lattice, respectively.  We define
\[
    d_i
    =
    \begin{cases}
        1 & \text{if } i > 0, \\
        -1 & \text{if } i \le 0,
    \end{cases}
    \qquad i \in I.
\]
(The $d_i$ are uniquely determined by the condition $(\alpha_i, \alpha_j) = d_i a_{ij}$ for all $i,j \in I$.)  Then set $q_i = q^{d_i}$ for all $i \in I$.   We define the quantum integers
\begin{equation} \label{qint}
    [a]=\frac{q^{a}-q^{-a}}{q-q^{-1}},\qquad a \in \Z.
\end{equation}

Following \cite[Th.~10.5.1]{Yam94}, we define the quantum enveloping superalgebra $\Ualg = U_q(\fg)$ to be the unital associative superalgebra over $\C(q)$ with generators
\[
    E_i, F_i, \quad i \in I, \qquad
    K_h,\quad h \in Y,
\]
of parities
\[
    \overline{E_i} = \overline{F_i} = \overline{\alpha_i},\quad i \in I, \qquad
    \overline{K_h} = 0,\quad h \in Y,
\]
subject to the following relations for $h,h' \in Y$, $i,j \in I$:
\begin{gather}
    K_0 = 1,\qquad
    K_h K_h = K_{h+h'},
    \\ \label{KEF}
    K_h E_i = q^{\langle h, \alpha_i \rangle} E_i K_h,\qquad
    K_h F_i = q^{- \langle h, \alpha_i \rangle} F_i K_h,
    \\ \label{skunk}
    [E_i,F_j]
    = \delta_{i,j} \frac{K_i - K_i^{-1}}{q_i-q_i^{-1}},\qquad
    \text{where } K_i := K_{h_i}^{d_i},
    \\
    [E_i,E_j] = 0,\qquad
    [F_i,F_j] = 0,\qquad
    \text{ for } a_{ij}=0
    \\ \label{racoon}
    E_i^2 E_j - [2] E_i E_j E_i + E_j E_i^2 = 0 = F_i^2 F_j - [2] F_i F_j F_i + F_j F_i^2,\qquad
    \text{for } |a_{ij}|=1,\ i \ne 0,
    \\
    [[[E_{-1},E_0]_{q_0},E_1]_{q_1}],E_0]=0,\quad
    [[[F_{-1},F_0]_{q_0},F_1]_{q_1}],F_0]=0,\qquad \text{if } n \ge 1,
\end{gather}
where
\[
    [X,Y]_a := XY-(-1)^{\bar{X} \bar{Y}} a YX, \quad a \in \kk,
    \qquad \text{and} \qquad [X,Y] := [X,Y]_1.
\]

Then $\Ualg$ is a Hopf superalgebra with coproduct $\Delta$, counit $\varepsilon$, and antipode $S$ given by
\begin{gather} \label{comult}
    \Delta(E_i) = E_i \otimes 1 + K_i \otimes E_i,\quad
    \Delta(F_i) = F_i \otimes K_i^{-1} + 1 \otimes F_i,\quad
    \Delta(K_h) = K_h \otimes K_h,
    \\
    \varepsilon(E_i) = 0,\qquad
    \varepsilon(F_i) = 0,\qquad
    \varepsilon(K_h) = 1,
    \\ \label{antipode}
    S(E_i) = -K_i^{-1} E_i,\qquad
    S(F_i) = -F_i K_i,\qquad
    S(K_h) = K_{-h}.
\end{gather}

Let $\sigma$ be the Hopf superalgebra automorphism of $\Ualg$ given by
\begin{equation} \label{sigma}
    \sigma(E_i) = (-1)^{\delta_{i=m-1 \ge 0}} E_i,\quad
    \sigma(F_i) = (-1)^{\delta_{i=m-1 \ge 0}} F_i,\quad
    \sigma(K_h) = K_h,\qquad
    i \in I,\ h \in Y.
\end{equation}
We define
\begin{equation} \label{Uisigma}
    \Us = \Ualg \rtimes \langle \sigma \rangle
\end{equation}
to be the superalgebra obtained from $\Ualg$ by adjoining an even generator $\sigma$, subject to the relations
\[
    \sigma^2 = 1,\quad
    \sigma x \sigma = \sigma(x),\qquad x \in \Ualg.
\]
We can extend the Hopf superalgebra structure of $\Ualg$ to $\Us$ by defining
\[
    \Delta(\sigma) = \sigma \otimes \sigma,\quad
    \varepsilon(\sigma) = 1,\quad
    S(\sigma) = \sigma.
\]
The reason for introducing $\sigma$ will become apparent when we consider the iquantum enveloping superalgebra; see \cref{pavement}.
\details{
    In general, if $\Gamma$ is a finite group acting on a Hopf superalgebra $H$ by even automorphisms, then one can form the crossed product $H \rtimes \Gamma$ Hopf superalgebra.  This is equal to $H \otimes \kk \Gamma$ as a coalgebra, with $S$ acting diagonally, and with multiplication determined by each factor $H$ and $\Gamma$ being a sub-superalgebra and $\gamma x \gamma^{-1} = \gamma(x)$ for all $g \in \Gamma$ and $x \in H$.
}

For $i \ne 0$, we have an automorphism of superalgebras $T_i \colon \Ualg \to \Ualg$ given by
\begin{equation} \label{Tidef}
    \begin{gathered}
        T_i (E_i) = -F_i K_{i},\qquad
        T_i (F_i) = - K_{i}^{-1}E_i,
        \\
        T_i(E_j) = E_j,\qquad T_i(F_j)=F_j, \qquad \text{for} \quad |j-i|>1,
        \\
        T_i(E_j) = E_i E_j - q^{(\alpha_i,\alpha_j)} E_j E_i,\qquad
        T_i(F_j) = F_i F_j - q^{-(\alpha_i,\alpha_j)} F_j F_i,
        \qquad \text{for} \quad |j-i|=1,
        \\
        T_i (K_h) = K_{s_i(h)},\quad h \in Y,
    \end{gathered}
\end{equation}
where $s_i$ is the reflection corresponding to the simple root $\alpha_i$.  The $T_i$ are super versions of the automorphisms $T_{i,1}''$ of \cite[\S 37.1.3]{Lus10} defining a braid group action; see \cite[Prop.~7.4.1]{Yam99} and \cite[Th.~3.4]{She25}.

%-------------------------------------------
\subsection{The natural module and its dual}
%-------------------------------------------

We now recall the quantum analogue of the natural module and its dual.  The quantum analogue $V^+$ of the natural module has basis $\{v_j^+ : j \in \Vset\}$, with the action of $\Us$ given by
\begin{equation} \label{panda+}
    \begin{gathered}
        E_i v_j^+ = \delta_{i+1,j} v_{j-1}^+,\qquad
        F_i v_j^+ = \delta_{i,j} v_{j+1}^+,\qquad
        K_h v_j^+ = q^{\langle h, \epsilon_j \rangle} v_j^+,
        \\
        K_i v_j^+ = q^{(-1)^{\delta_{j \le 0}}(\delta_{i,j} - \delta_{i+1,j})} v_j^+,\qquad
        \sigma v_j^+ = (-1)^{\delta_{j=m>0}} v_j^+.
    \end{gathered}
\end{equation}
We remind the reader that $K_i = K_{h_i}^{d_i}$; see \cref{skunk}.
\details{
    In terms of the elementary matrices, we can write the associated representation $\rho_+$ as
    \[
        \rho_+(E_i) = E_{i,i+1},\qquad
        \rho_+(F_i) = E_{i+1,i},\qquad
        \rho_+(K_i) = q^{d_i} E_{i,i} + q^{-(-1)^{\delta_{i,0}}d_i} E_{i+1,i+1}.
    \]
    Then, to verify \cref{skunk}, we compute
    \begin{multline*}
        \rho_+ \left( E_i F_j - (-1)^{\delta_{i,0} \delta_{j,0}} F_j E_i \right)
        = E_{i,i+1} E_{j+1,j} - (-1)^{\delta_{i,0} \delta_{j,0}} E_{j+1,j} E_{i,i+1}
        \\
        = \delta_{i,j} \left( E_{i,i} - (-1)^{\delta_{i,0}} E_{i+1,i+1} \right)
        = \delta_{i,j} \rho_+ \left( \frac{K_i-K_i^{-1}}{q_i-q_i^{-1}} \right).
    \end{multline*}
}
The dual module $V^-$ has basis $\{ v_j^- : j \in \Vset \}$, with the action of $\Us$ given by
\begin{equation} \label{panda-}
    \begin{gathered}
        E_i v_j^- = \delta_{i,j} v_{j+1}^-,\qquad
        F_i v_j^- = \delta_{i+1,j} (-1)^{\delta_{i,0}} v_{j-1}^-,\qquad
        K_h v_j^- = q^{-\langle h, \epsilon_j \rangle} v_j^-,
        \\
        K_i v_j^- = q^{(-1)^{\delta_{j > 0}}(\delta_{i,j} - \delta_{i+1,j})}  v_j^-,\qquad
        \sigma v_j^- = (-1)^{\delta_{j=m>0}} v_j^-.
    \end{gathered}
\end{equation}
\details{
    We can write the associated representation $\rho_-$ as
    \[
        \rho_-(E_i) = E_{i+1,i},\qquad
        \rho_-(F_i) = (-1)^{\delta_{i,0}} E_{i,i+1},\qquad
        \rho(K_i) = q^{-d_i} E_{i,i} + q^{(-1)^{\delta_{i,0}}d_i} E_{i+1,i+1}.
    \]
    Then, to verify \cref{skunk}, we compute
    \begin{multline*}
        \rho_- \left( E_i F_j - (-1)^{\delta_{i,0} \delta_{j,0}} F_j E_i \right)
        = (-1)^{\delta_{i,0}} E_{i+1,i} E_{j,j+1} - (-1)^{\delta_{i,0} + \delta_{i,0} \delta_{j,0}} E_{j,j+1} E_{i+1,i}
        \\
        = \delta_{i,j} \left( (-1)^{\delta_{i,0}} E_{i+1,i+1} - E_{i,i} \right)
        = \delta_{i,j} \rho_- \left( \frac{K_i-K_i^{-1}}{q_i-q^{-1}} \right).
    \end{multline*}
    Note that $V^-$ is obtained from $V^+$ via composing with the algebra automorphism $E_i\mapsto F_i$, $F_i\mapsto(-1)^{\delta_{i,0}} E_i$, $K_h\mapsto K_h^{-1}$, $\sigma\mapsto \sigma$.
}

We have a homomorphism of $\Us$-modules
\begin{equation} \label{ev+}
    \ev^+ \colon V^+ \otimes V^- \to \triv,\qquad
    \ev^+(v_i^+ \otimes v_j^-) = \delta_{i,j} (-q)^{|i|},
\end{equation}
where $\triv$ denotes the trivial $\Us$-module.
\details{
    We use the equality
    \[
        \ev^+(x v_i^+ \otimes v_j^-)
        = (-1)^{\bar{x} p(i)} \ev^+(v_i^+ \otimes S(x) v_j^-)
        \qquad \text{for all } i,j \in \Vset,\ x \in \Ualg.
    \]
    to verify that \cref{ev+} is the correct form.  It follows from weight considerations that $\ev^+(v_i^+ \otimes v_j^-) = 0$ when $i \ne j$.  Thus, we need only verify the equality in \cref{ev+} when $i=j$.  We normalize $\ev^+$ so that $\ev^+(v_m^+ \otimes v_m^-) = (-q)^m$. Then, for $2-n \le j \le m$, we have
    \begin{align*}
        \ev^+(v_{j-1}^+ \otimes v_{j-1}^-)
        \ &\overset{\mathclap{\cref{panda+}}}{=}\ \ev^+(E_{j-1} v_j^+ \otimes v_{j-1}^-) \\
        &= \ev^+ (v_j^+ \otimes S(E_{j-1}) v_{j-1}^-) \\
        \ &\overset{\mathclap{\cref{antipode}}}{=}\ -\ev^+(v_j^+ \otimes K_{j-1}^{-1} E_{j-1} v_{j-1}^-) \\
        \ &\overset{\mathclap{\cref{panda-}}}{=}\ - \ev^+(v_j^+ \otimes K_{j-1}^{-1} v_j^-) \\
        \ &\overset{\mathclap{\cref{panda-}}}{=}\ - q^{d_{j-1} \langle h_{j-1}, \epsilon_j \rangle} \ev^+(v_j^+ \otimes v_j^-) \\
        &= - q^{-(-1)^{\delta_{j,1}}d_{j-1}} \ev^+(v_j^+ \otimes v_j^-) \\
        &= - q^{(-1)^{\delta_{j > 0}}} \ev^+(v_j^+ \otimes v_j^-).
    \end{align*}
    This proves by induction that \cref{ev+} holds for all $i=j \in \Vset$.
}
Since $V^+$ and $V^-$ are simple, the homomorphism space $\Hom_\Ualg(V^+ \otimes V^-,\triv)$ is one-dimensional, spanned by $\ev^+$.  We also have a homomorphism of $\Us$-modules
\begin{equation} \label{ev-}
    \ev^- \colon V^- \otimes V^+ \to \triv,\qquad
    \ev^- (v_i^- \otimes v_j^+) = \delta_{i,j} (-1)^{i+p(i)} q^{m+n-|i-1|},
\end{equation}
and $\Hom_\Ualg(V^- \otimes V^+,\triv)$ is one-dimensional, spanned by $\ev^-$.
\details{
    As for $\ev^+$, we need only consider the equation in \cref{ev-} when $i=j$.  We normalize $\ev^-$ so that $\ev^-(v_{1-n}^- \otimes v_{1-n}^+) = (-1)^{1-n+p(1-n)}q^m$.  For $1-n \le j \le m-1$, we have
    \begin{align*}
        \ev^-(v_{j+1}^- \otimes v_{j+1}^+)
        \ &\overset{\mathclap{\cref{panda-}}}{=}\ \ev^-(E_j v_j^- \otimes v_{j+1}^+) \\
        &= (-1)^{\delta_{j,0}} \ev^-(v_j^- \otimes S(E_j) v_{j+1}^+) \\
        \ &\overset{\mathclap{\cref{antipode}}}{=}\ - (-1)^{\delta_{j,0}} \ev^-(v_j^- \otimes K_j^{-1} E_j v_{j+1}^+) \\
        \ &\overset{\mathclap{\cref{panda+}}}{=}\ - (-1)^{\delta_{j,0}} \ev^-(v_j^- \otimes K_j^{-1} v_j^+) \\
        \ &\overset{\mathclap{\cref{panda+}}}{=}\ - (-1)^{\delta_{j,0}} q^{-d_j \langle h_j, \epsilon_j \rangle} \ev^-(v_j^- \otimes v_j^+) \\
        &= - (-1)^{\delta_{j,0}} q^{-d_j} \ev^-(v_j^- \otimes v_j^+) \\
        &= - (-1)^{\delta_{j,0}} q^{(-1)^{\delta_{j>0}}} \ev^-(v_j^- \otimes v_j^+).
    \end{align*}
    This proves by induction that \cref{ev-} holds for all $i=j \in \Vset$.
}

For a basis $B^\pm$ of $V^\pm$, we let $\{ v^\vee : v \in B^\pm \}$ denote the dual basis of $V^\mp$ determined by
\[
    \ev^\mp( v^\vee \otimes w ) = \delta_{v,w},\qquad v,w \in B^\pm.
\]
It follows from \cref{ev+,ev-} that the bases dual to $\{v_i^+ : i \in \Vset\}$ and $\{v_i^- : i \in \Vset\}$ are given by
\begin{equation} \label{button}
    (v_i^+)^\vee = (-1)^{i+p(i)} q^{|i-1|-m-n} v_i^-,\qquad
    (v_i^-)^\vee = (-q)^{-|i|} v_i^+.
\end{equation}
We have $\Us$-module homomorphisms
\begin{equation} \label{coev}
    \coev_\pm \colon \triv \to V^\mp \otimes V^\pm,\qquad
    1 \mapsto \sum_{v \in B^\mp} v \otimes v^\vee.
\end{equation}
This definition is independent of the basis $B^\mp$.

Let $\Us$-tmod denote the category of tensor $\Us$-modules.  By definition, this is the full subcategory of the category of $\Us$-modules whose objects are finite direct sums of summands of tensor products of $V^+$ and $V^-$.  The category $\Us$-tmod is braided monoidal, where the braiding is given by the universal $R$-matrix.  Define
\[
    p(i,j) := \delta_{i,j} (1-2p(i))
    =
    \begin{cases}
        0, & i \ne j, \\
        1, & i=j > 0, \\
        -1, & i=j \le 0,
    \end{cases}
    \qquad i,j \in \Vset.
\]
Then the action of the braiding on $V^+ \otimes V^+$ is
\begin{equation} \label{Vbraid}
    \begin{gathered}
        T_{++} \colon V^+ \otimes V^+ \to V^+ \otimes V^+,
        \\
        T_{++}(v_i^+ \otimes v_j^+)
        = (-1)^{p(i)p(j)} q^{p(i,j)} v_j^+ \otimes v_i^+ + \delta_{i<j} (q-q^{-1}) v_i^+ \otimes v_j^+.
%        \begin{dcases}
%            (-1)^{p(i) p(j)} v_j^+ \otimes v_i^+ & \text{if } i>j, \\
%            (-1)^{p(i)} q^{1-2 p(i)} v_i^+ \otimes v_i^+ & \text{if } i=j, \\
%            (-1)^{p(i) p(j)} v_j^+ \otimes v_i^+ + (q-q^{-1}) v_i^+ \otimes v_j^+ & \text{if } i<j.
%        \end{dcases}
    \end{gathered}
\end{equation}
See, for example, \cite[(1)]{Mit06}.  One can verify by direct computation that
\begin{equation} \label{Vbraidinv}
    T_{++}^{-1}(v_i^+ \otimes v_j^+)
    =
    (-1)^{p(i)p(j)} q^{-p(i,j)} v_j^+ \otimes v_i^+ - \delta_{i>j} (q-q^{-1}) v_i^+ \otimes v_j^+.
\end{equation}
\details{
    It follows from the explicit formula for the universal $R$-matrix given in \cite[Th.~10.6.1]{Yam94} that it commutes with the action of $\sigma \otimes \sigma$.
}

%------------------------------------------------
\subsection{The iquantum enveloping superalgebra}
%------------------------------------------------

Following \cite[\S 7.1]{SW24}, we consider the following super Satake diagram of type AI-II:
\begin{equation} \label{satake}
    \begin{tikzpicture}[anchorbase,xscale=-1]
        \draw (0,0) -- (1,0);
        \node at (1.5,0) {$\cdots$};
        \draw (2,0) -- (8,0);
        \node at (8.5,0) {$\cdots$};
        \draw (9,0) -- (11,0);
        \filldraw[fill=white,draw=black] (0,0) circle (0.1);
        \filldraw[fill=white,draw=black] (3,0) circle (0.1);
        \filldraw[draw=black,fill=white] (4,0) circle (0.1);
        \filldraw[fill=black] (5,0) circle (0.1);
        \filldraw[fill=white,draw=black] (6,0) circle (0.1);
        \filldraw[fill=black] (7,0) circle (0.1);
        \filldraw[fill=white,draw=black] (10,0) circle (0.1);
        \filldraw[fill=black] (11,0) circle (0.1);
        \draw (3.93,-0.07) -- (4.07,0.07);
        \draw (4.07,-0.07) -- (3.93,0.07);
        \node at (0,-0.3) {\strandlabel{m-1}};
        \node at (3,-0.3) {\strandlabel{1}};
        \node at (4,-0.3) {\strandlabel{0}};
        \node at (5,-0.3) {\strandlabel{-1}};
        \node at (6,-0.3) {\strandlabel{-2}};
        \node at (7,-0.3) {\strandlabel{-3}};
        \node at (11,-0.3) {\strandlabel{1-2n}};
    \end{tikzpicture}
\end{equation}
where $I_\bu = \{1-2a : 1 \leq a \leq n\}$ and $I_\circ = I \backslash I_\bu$.  The diagram automorphism $\tau$ that is part of the data of the Satake diagram is the identity in our case.  The diagram \cref{satake} corresponds to an involution on $\fgl(m|2n)$; this is the involution denoted $\theta(I,\tau)$ in \cite[(2.5)]{SW24}. Note that the underlying Dynkin diagram of \cref{satake} is \cref{Dynkin}, with $n$ replaced by $2n$.  In the extreme case where $n=0$, we obtain a Satake diagram of type AI; when $m=0$, we obtain a Satake diagram of type AII; see \cite[Table~4]{BW18}.

The following definition, given in \cite[Def.~4.1]{SW24}, was inspired by \cite[\S4]{Let99} and \cite[Def.~5.1]{Kol14}, which treated the non-super case.

\begin{defin}\label{def:iQG}
    Fix $\va_i \in \C(q)^\times$ for $i \in I_\circ$.  The corresponding \emph{iquantum enveloping superalgebra} $\Ui$ of type AI-II is the $\C(q)$-subalgebra of $\Ualg$ generated by the elements
    \begin{equation} \label{tubeless}
        \begin{gathered}
             B_i = F_i + \va_i E_{i}K_i^{-1},\qquad 1 \le i \le m-1,\\
             B_0 = F_0 + \va_0 T_{-1}(E_0)K_0^{-1}, \\
             B_{2k} = F_{2k} + \va_{2k} T_{2k+1} \big( T_{2k-1}(E_{2k}) \big) K_{2k}^{-1},\qquad 1-n \le k \le -1,\\
             E_{2k+1}, F_{2k+1}, K_{2k+1}^{\pm 1},\qquad -n \le k \le -1.
        \end{gathered}
    \end{equation}
    We define $\Uis$ to be the subalgebra of $\Us$ generated by $\Ui$ and $\sigma$.
\end{defin}

The restriction of the automorphism $\sigma$ of \cref{sigma} to $\Ui$ is given by
\begin{gather}\label{Binv1}
    \sigma(B_i) = (-1)^{\delta_{i=m-1 \ge 0}} B_i,\qquad i \in I_\circ,
    \\ \label{Binv2}
    \sigma(E_i) = E_i,\qquad
    \sigma(F_i) = F_i,\qquad
    \sigma(K_i^{\pm 1}) = K_i^{\pm 1},\qquad
    i \in I_\bu.
\end{gather}
The braid group actions appearing in \cref{tubeless} can be computed explicitly, using \cref{Tidef}:
\begin{equation} \label{iodine1}
    T_{-1}(E_0) = E_{-1} E_0 - q E_0 E_{-1},
\end{equation}
\begin{multline} \label{iodine2}
    T_{2k+1}(T_{2k-1}(E_k))
    = E_{2k-1}E_{2k+1}E_{2k} - q E_{2k-1}E_{2k}E_{2k+1} \\
    - qE_{2k+1}E_{2k}E_{2k-1} + q^2E_{2k}E_{2k+1}E_{2k-1}.
\end{multline}
\details{
    We have $(\alpha_0,\alpha_{-1})=1$ and $(\alpha_{2k-1},\alpha_{2k}) = (\alpha_{2k+1},\alpha_{2k}) = 1$, for $1-n \le k \le -1$.  Thus, \cref{iodine1} follows immediately and, for \cref{iodine2}, we compute
    \begin{align*}
        T_{2k+1}T_{2k-1}(E_{2k})
        &= T_{2k+1}(E_{2k-1}E_{2k}-qE_{2k}E_{2k+1})
        \\
        &= T_{2k+1}(E_{2k-1})T_{2k+1}(E_{2k})-qT_{2k+1}(E_{2k})T_{2k+1}(E_{2k-1})
        \\
        &= E_{2k-1}(E_{2k+1}E_{2k}-qE_{2k}E_{2k+1})-q(E_{2k+1}E_{2k}-qE_{2k}E_{2k+1}) E_{2k-1}
        \\
        &= E_{2k-1}E_{2k+1}E_{2k}-q E_{2k-1}E_{2k}E_{2k+1}-qE_{2k+1}E_{2k}E_{2k-1}+q^2E_{2k}E_{2k+1}E_{2k-1}.
    \end{align*}
}

By \cite[Prop.~4.2]{SW24}, $\Ui$ is a right coideal subalgebra of $\Ualg$ since
\[
    \Delta(\Ui) \subseteq \Ui \otimes \Ualg.
\]
Thus $\Uis$ is also a right coideal subalgebra of $\Us$.  Let $\Uis$-tmod denote the full subcategory of $\Ui$-modules whose objects are restrictions of objects in $\Us\tmod$.  Then $\Uis\tmod$ is naturally a right module category over $\Us\tmod$. Given $M \in \Uis\tmod$ and $N \in \Us\tmod$, the bifunctor
\[
    \otimes \colon \Uis\tmod \times \Us\tmod \to \Uis\tmod
\]
sends $(M,N)$ to $M \otimes N := M \otimes_\kk N$, which is naturally a $\Uis$-module via $\Delta$.

We note that the superalgebra $\Ui$ is a quantum analogue of the enveloping superalgebra of the orthosymplectic Lie superalgebra $\osp(m|2n,\C)$; see \cite[Ex.~2.16 and Rem.~7.8]{SW24} and the proof of \cref{fullthm}.

\begin{rem}
    The superalgebra $\Ui$ is sometimes called an \emph{iquantum supergroup}.  Let us explain how \cref{def:iQG} is obtained from \cite[Def.~4.1]{SW24}.  First note that, as defined in \cite[\S 4.1]{SW24}, $\Ui$ also depends on a set of parameters $\{\kappa_i \in \C(q) : i \in I_\circ\}$.  However, the condition \cite[(4.10)]{SW24} forces all $\kappa_i=0$ in our case.  The diagram automorphism $\tau$ appearing in \cite[Def.~4.1]{SW24} is the identity for us.  The $w_\bullet$ there is the longest element of the Weyl group associated to the subdiagram of \cref{satake} containing the black nodes.  Thus, for us, $T_{w_\bullet} = T_{-1} T_{-3} \dotsm T_{1-2n}$ and the $T_{1-2j}$, $1 \le j \le n$, pairwise commute.  For $i \ne j$, we have $T_i(E_j) = E_j$ whenever $a_{ij}=0$.  Thus
    \[
        T_{w_\bu}(E_0) = T_{-1}(E_0),\quad
        T_{w_\bu}(E_{2k}) = T_{2k+1} (T_{2k-1} (E_{2k})),\qquad
        -1 \le k \le 1-n.
    \]
\end{rem}

\begin{rem}\label{revolution}
    Up to isomorphisms of coideal subalgebras, $\Ui$ is independent of the parameters $\va_i$.  More precisely, for two choices $\va = (\va_i)_{i \in I_\circ}$ and $\va' = (\va'_i)_{i \in I_\circ}$, there is a Hopf algebra automorphism of $\Ualg$ sending $\Ui$ for the choice $\va$ to $\Ui$ for the choice $\va'$.  See \cite[Lem.~2.51 and Rem.~2.5.2]{Wat21} for details.  While the treatment there is for the non-super case, the same automorphism applies in our setting.  Because of this independence, one loses no generality in choosing specific parameters.  The choice
    \begin{equation}\label{oilers}
        \va_i
        =
        \begin{cases}
            q^{-1} & \text{if } 0 < i < m-1, \\
            -q^{-1} & \text{if } i \in \{0,-2,-4,\dotsc,2-2n\},
        \end{cases}
    \end{equation}
    simplifies some of the formulas in the current paper.  However, we work with general parameters since it does not involve substantially more work and allows the reader to make whatever choice is more convenient for their purposes.
\end{rem}

\begin{rem}\label{pavement}
    The passage from $\Ui$ to $\Uis$ is a quantum analogue of the passage from $\osp(m|2n)$ to $\mathrm{OSp}(m|2n)$.  The extra generator $\sigma$ corresponds to an element of $\mathrm{O}(m) \times \mathrm{Sp}(2n)$ that has determinant $-1$ and hence is not in the identity component.  This will be needed in order for the incarnation functor defined in \cref{sec:fullness} to be full.
\end{rem}

\begin{rem}
    When $m$ and $n$ are both nonzero, there exist various distinct super Satake diagrams corresponding to the supersymmetric pair $(\fgl(m|2n), \osp(m|2n))$; see \cite[Ex.~2.16]{SW24} for a complete list of possible choices. Although we have made a specific choice in \cref{satake}, we expect that all the results in this manuscript can be developed in a parallel fashion for other choices of Satake diagrams.  The precise relationship between the iquantum enveloping algebras arising from different Satake diagrams remains unclear, though they are expected to be isomorphic as superalgebras.  A uniform framework encompassing all choices could potentially be realized via a general theory of the universal $K$-matrix in the super case, following the approach of \cite{BW18,BK19} in the non-super case; see also \cref{reflection}.  A step in this direction was taken in \cite[\S 6]{SW24}, where the quasi $K$-matrix was constructed for general quantum supersymmetric pairs. Motivated by this philosophy, we formulate many of our arguments in \cref{sec:incar} and beyond in a way that avoids reliance on any particular choice of Satake diagram.
\end{rem}

%-------------------------------------------------------------
\subsection{Modules over the iquantum enveloping superalgebra}
%-------------------------------------------------------------

We have a restriction functor
\[
    \Res \colon \Us\tmod \to \Uis\tmod.
\]
When there is no chance for confusion, we will sometimes continue to denote $\Res(V^\pm)$ by $V^\pm$.  Direct computation using \cref{panda+,panda-,iodine1,iodine2} shows that the actions of $B_i$, $i\in I_\circ$, on $V^+$ and $V^-$ are given by
\begin{equation} \label{sloth}
    B_i v^+_j
    =
    \begin{cases}
        v_{j+1}^+ & \text{if } j=i, \\
        - q \va_i v_{j-3}^+ & \text{if } i<0,\ j=i+2, \\
        - q \va_0 v_{-1}^+ & \text{if } i=0,\ j=1, \\
        q \va_i v_{j-1}^+ & \text{if } i>0,\ j=i+1, \\
        0 & \text{otherwise},
    \end{cases}
    \quad
    B_i v^-_j
    =
    \begin{cases}
        (-1)^{\delta_{i,0}} v_{j-1}^- & \text{if } j=i+1, \\
        - q \va_i v_{j+3}^- & \text{if } i<0,\ j=i-1, \\
        - q \va_0 v_1^- & \text{if } i=0,\ j=-1, \\
        q \va_i v_{j+1}^- & \text{if } i>0,\ j=i, \\
        0 & \text{otherwise}.
    \end{cases}
\end{equation}
\details{
    We first consider the action on $V^+$.  If $i>0$, then
    \[
        B_i(v_j^+)
        \overset{\cref{tubeless}}{=} (F_i + \va_i E_i K_i^{-1}) v_j^+
        \overset{\cref{panda+}}{=} \delta_{i,j} v_{j+1}^+ + \delta_{i+1,j} q \va_i v_{j-1}^+.
    \]
    If $i=0$, we have
    \[
        B_0 v_j^+
        \overset{\cref{tubeless}}{\underset{\cref{iodine1}}{=}} (F_0 + \va_0 E_{-1} E_0 K_0^{-1}) v_j^+
        \overset{\cref{panda+}}{=} \delta_{j,0} v_1^+ + \delta_{j,1} q \va_0 v_{-1}^+,
    \]
    where, in the first equality, we used the fact that $E_0 E_{-1}$ acts as zero on $V^+$ by \cref{panda+}.
    For $1-n \le k \le -1$, we have
    \[
        B_{2k} v_j^+
        \overset{\cref{tubeless}}{\underset{\cref{iodine2}}{=}} (F_{2k} - q \va_{2k} E_{2k-1} E_{2k} E_{2k+1} K_{2k}^{-1})v_j^+
        \overset{\cref{panda+}}{=} \delta_{2k,j} v_{j+1}^+ - \delta_{2k+2,j} q \va_{2k} v_{2k-1}^+,
    \]
    where, in the first equality, we used the fact that $E_{2k-1} E_{2k+1} E_{2k}$, $E_{2k+1} E_{2k} E_{2k-1}$, and $E_{2k} E_{2k+1} E_{2k-1}$ act as zero on $V^+$ by \cref{panda+}.

    Now we consider the action on $V^-$.  If $i>0$, then
    \[
        B_i(v_j^+)
        \overset{\cref{tubeless}}{=} (F_i + \va_i E_i K_i^{-1}) v_j^-
        \overset{\cref{panda-}}{=} \delta_{i+1,j} v_{j-1}^- + \delta_{i,j} q \va_i v_{j+1}^-.
    \]
    If $i=0$, we have
    \[
        B_0 v_j^+
        \overset{\cref{tubeless}}{\underset{\cref{iodine1}}{=}} (F_0 - q \va_0 E_0 E_{-1} K_0^{-1}) v_j^-
        \overset{\cref{panda-}}{=} - \delta_{j,1} v_0^- - \delta_{j,-1} q \va_0 v_1^-,
    \]
    where, in the first equality, we used the fact that $E_{-1} E_0$ acts as zero on $V^-$ by \cref{panda-}.  For $1-n \le k \le -1$, we have
    \[
        B_{2k} v_j^-
        \overset{\cref{tubeless}}{\underset{\cref{iodine2}}{=}} (F_{2k} - q \va_{2k} E_{2k+1} E_{2k} E_{2k-1} K_{2k}^{-1})v_j^-
        \overset{\cref{panda-}}{=} \delta_{2k+1,j} v_{j-1}^- - \delta_{2k-1,j} q \va_{2k} v_{2k+2}^-,
    \]
    where, in the first equality, we used the fact that $E_{2k-1} E_{2k+1} E_{2k}$, $E_{2k-1} E_{2k} E_{2k+1}$, and $E_{2k} E_{2k+1} E_{2k-1}$ act as zero on $V^+$ by \cref{panda+}.
}

Define the involution
\begin{equation} \label{pump}
    \phi \colon \Vset \to \Vset,\quad
    \phi(i)
    = i - \delta_{i \le 0} (-1)^i
    =
    \begin{cases}
        i, & i>0, \\
        i-1, & i \in \{0,-2,-4,\dotsc,2-2n\}, \\
        i+1, & i \in \{-1,-3,-5,\dotsc,1-2n\}.
    \end{cases}
\end{equation}
For $i \in \Vset$, define
\begin{equation} \label{kowloon}
    \tva_i
    :=
    \begin{cases}
        \prod_{k=1}^{i-1} (q \va_k)^{-1}, & i > 0, \\
        \prod_{k = \lfloor \frac{i+1}{2} \rfloor}^0 (-q \va_{2k}), & i \le 0.
    \end{cases}
\end{equation}
Note that $\tva_{\phi(i)} = \tva_i$ for all $i \in \Vset.$

\begin{lem} \label{varphi}
    The linear map $\varphi \colon V^- \to V^+$ determined by
    \begin{equation}\label{varphidef}
        \varphi(v_i^-) = \tva_i v_{\phi(i)}^+
    \end{equation}
    is a $\Uis$-module isomorphism.  Its inverse is determined by $\varphi^{-1}(v_i^+) = \tva_i^{-1} v_{\phi(i)}^-$.
\end{lem}

\begin{proof}
    This follows from a direct computation using \cref{sloth}.
\end{proof}

\details{
    We must show that $\varphi$ intertwines the actions of the generators \cref{tubeless} of $\Ui$.  For $i>0$, we have
    \begin{multline*}
        \varphi(B_i v_j^-)
        \overset{\cref{sloth}}{=} \delta_{i,j} q \va_i \varphi(v_{j+1}^-) + \delta_{i+1,j} \varphi(v_{j-1}^-)
        \overset{\cref{varphidef}}{=} \delta_{i,j} \tva_j v_{j+1}^+ + \delta_{i+1,j} \tva_{j-1} v_{j-1}^+
        \\
        = \tva_j \left( \delta_{i,j} v_{j+1}^+ + \delta_{i+1,j} q \va_{j-1} v_{j-1}^+ \right)
        \overset{\cref{sloth}}{=} \tva_j B_i v_j^+
        \overset{\cref{varphidef}}{=} B_i \varphi(v_j^-).
    \end{multline*}
    Next, we compute
    \begin{multline*}
        \varphi(B_0 v_j^-)
        \overset{\cref{sloth}}{=} - \delta_{j,1} \varphi(v_0^-) - \delta_{j,-1} q \va_0 \varphi(v_1^-)
        \overset{\cref{varphidef}}{=} - \delta_{j,1} \tva_0 v_{-1}^+ - \delta_{j,-1} q \va_0 v_1^+
        \\
        = \delta_{j,1} B_0 v_1^+ - \delta_{j,-1} q \va_0 B_0 v_0^+
        = \delta_{j,1} B_0 \varphi(v_j^-) + \delta_{j,-1} B_0 \varphi(v_j^-)
        = B_0(\varphi(v_j^-)).
    \end{multline*}
    Now, for $1-n \le k \le -1$, we have
    \begin{multline*}
        \varphi(B_{2k}v_j^-)
        \overset{\cref{sloth}}{=} \varphi \big( \delta_{2k+1,j} v_{j-1}^- + \delta_{2k-1,j} \tva_{2k} v_{j+3}^- \big)
        \overset{\cref{varphidef}}{=} \delta_{2k+1,j} \tva_{2k} v_{j-2}^+ + \delta_{2k-1,j} \tva_{2k} v_{j+2}^+
        \\
        \overset{\cref{sloth}}{=} \delta_{2k+1,j} \tva_j B_{2k} v_{j+1}^+ + \delta_{2k-1,j} \tva_j B_{2k} v_{j+1}^+
        \overset{\cref{varphidef}}{=} B_{2k}\varphi(v_j^-).
    \end{multline*}
    For $-n \le k \le -1$, we have
    \begin{gather*}
        \varphi(E_{2k+1} v_j^-)
        \overset{\cref{panda-}}{=} \delta_{2k+1,j} \varphi(v_{j+1}^-)
        \overset{\cref{varphidef}}{=} \delta_{2k+1,j} \tva_{j+1} v_j^+
        \overset{\cref{panda+}}{=} \tva_j E_{2k+1} v_{j+1}^+
        = E_{2k+1} \varphi(v_j^-),
        \\
        \varphi(F_{2k+1} v_j^-)
        \overset{\cref{panda-}}{=} \delta_{2k+2,j} \varphi(v_{j-1}^-)
        \overset{\cref{varphidef}}{=} \delta_{2k+2,j} \tva_{j-1} v_j^+
        \overset{\cref{panda+}}{=} \tva_j F_{2k+1} v_{j-1}^+
        = F_{2k+1} \varphi(v_j^-),
     \end{gather*}
     where we have used the fact that $\tva_{2k+2} = \tva_{2k+1}$.  Finally,
     \begin{multline*}
        \varphi(K_{2k+1} v_j^-)
        \overset{\cref{panda-}}{=} \varphi \left( \delta_{2k+1,j} q v_j^- + \delta_{2k+2,j} q^{-1} v_j^- + \delta_{j \ne 2k+1,2k+2} v_j^- \right)
        \\
        \overset{\cref{varphidef}}{=} \tva_j \left( \delta_{2k+1,j} q v_{j+1}^+ + \delta_{2k+2,j} q^{-1} v_{j-1}^+ + \delta_{j \ne 2k+1,2k+2} v_{\phi(j)}^+ \right)
        \\
        \overset{\cref{panda+}}{=} \tva_j \left( \delta_{2k+1,j} K_{2k+1} v_{j+1}^+ + \delta_{2k+2,j} K_{2k+1} v_{j-1}^+ + \delta_{j \ne 2k+1,2k+2} K_{2k+1} v_{\phi(j)}^+ \right)
        \overset{\cref{varphidef}}{=} K_{2k+1} \varphi(v_j^-).
    \end{multline*}

    It is clear that $\varphi$ intertwines the action of $\sigma$.
}

\begin{lem}
    The $\Uis$-modules $V^+$ and $V^-$ are simple.
\end{lem}

\begin{proof}
    By \cref{varphi}, it suffices to show that $V^+$ is simple.  Suppose $M$ is a nonzero submodule of $V^+$.  Since $V^+$ is generated by $v_m^+$, to show that $M=V^+$, it suffices to show that $v_m^+ \in M$.  Let $v$ be a nonzero element of $M$.  If $v=v_m^+$, we are done.  Otherwise, multiplying by a nonzero scalar if necessary, we may assume
    \[
        v = v_i^+ + \sum_{j>i} c_j v_j^+
    \]
    for some $i \in \Vset$, $i<m$, and $c_j \in \kk$.  If $i<m-1$, then, by \cref{sloth}, we have
    \[
        B_{m-1} B_{m-2} \dotsm B_{i+1} B_i v = v_m^+,
    \]
    where we interpret $B_{2k-1}$ as $F_{2k-1}$ for $k \le 0$.  On the other hand, if $i=m-1$, then $v = v_{m-1}^+ + c_m v_m^+$.  First suppose $c_m \ne 0$.  If $m >0$, we have
    \[
        (1-\sigma)v = 2 v_m^+.
    \]
    If $m=0$, we have
    \[
        B_{-1} v
        = v_0^+.
    \]
    Finally, if $c_m = 0$, then $B_{m-1} v = B_{m-1} v_{m-1}^+ = v_m^+$.  In all cases, we have shown that $v_m^+ \in M$.
\end{proof}

%==============================================
\section{Incarnation functors\label{sec:incar}}
%==============================================

In this section we relate the diagrammatic categories of \cref{sec:diagrammcats} to the representation theory of $\Ui$.  Throughout this section, we work over the field $\kk = \C(q)$.

%--------------------------------------------
\subsection{The oriented incarnation functor}
%--------------------------------------------

The following proposition is known to experts.  However, since we could not locate a proof in the literature, we have included one.

\begin{prop} \label{oincarnate}
    There is a unique $\kk$-linear monoidal functor $\bR_\OS \colon \OS(q,q^{m-n}) \to \Us\tmod$ given on objects by $\upobj \mapsto V^+$, $\downobj \mapsto V^-$, and on morphisms by
    \[
        \posupcross \mapsto T_{++},\quad
        \negupcross \mapsto T_{++}^{-1},\quad
        \rightcap \mapsto \ev^+,\quad
        \rightcup \mapsto \coev_+,\quad
        \leftcap \mapsto \ev^-,\quad
        \leftcup \mapsto \coev_-.
    \]
\end{prop}

\begin{proof}
    We must show that $\bR_\OS$ respects the relations \cref{obraid,oskein,ocurlbub,adjunction}.  The first three equalities 
    in \cref{obraid}
        follow immediately from the fact that $T_{++}$ is the action of the braiding on $V^+ \otimes V^+$.  Furthermore, it follows from \cref{lego} that $\bR_\OS(\posrightcross)$ and $\bR_\OS(\negleftcross)$ are the components of the braiding on $V^+ \otimes V^-$ and $V^- \otimes V^+$, respectively.  Indeed, if $X$ and $Y$ are objects in a rigid braided monoidal category, and $X^*$ is the dual of $X$, then naturality of the braiding implies that
    \[
        \begin{tikzpicture}[anchorbase,scale=1.3]
            \draw (0.3,-0.4) \botlabel{Y} to[out=up,in=down] (-0.3,0.4);
            \draw[wipe] (-0.3,-0.4) to[out=up,in=down] (0,0) arc(180:0:0.15) to[out=down,in=up] (0,-0.4);
            \draw (-0.3,-0.4) \botlabel{X} to[out=up,in=down] (0,0) arc(180:0:0.15) to[out=down,in=up] (0,-0.4) \botlabel{X^*};
        \end{tikzpicture}
        =
        \begin{tikzpicture}[anchorbase,scale=1.3]
            \draw (-0.3,-0.4) \botlabel{X} -- (-0.3,-0.2) arc(180:0:0.15) -- (0,-0.4) \botlabel{X^*};
            \draw (0.3,-0.4) \botlabel{Y} -- (0.3,0.4);
        \end{tikzpicture}
    \]
    Tensoring on the left with $X^*$, then composing on the bottom with
    \[
        \begin{tikzpicture}[scale=1.3]
            \draw (-0.3,0.4) \toplabel{X^*} -- (-0.3,0.2) arc(180:360:0.15) -- (0,0.4) \toplabel{X};
            \draw (0.3,-0.2) -- (0.6,0.4) \toplabel{Y};
            \draw[wipe] (0.6,-0.2) -- (0.3,0.4);
            \draw (0.6,-0.2) -- (0.3,0.4) \toplabel{X^*};
        \end{tikzpicture}
    \]
    and using the adjunction relation, gives
    \[
        \begin{tikzpicture}[centerzero]
            \draw (-0.3,-0.3) \botlabel{Y} -- (0.3,0.3);
            \draw[wipe] (0.3,-0.3) -- (-0.3,0.3);
            \draw (0.3,-0.3) \botlabel{X^*} -- (-0.3,0.3);
        \end{tikzpicture}
        =
        \begin{tikzpicture}[centerzero,xscale=-1.4]
            \draw (-0.1,-0.3) \botlabel{Y} \braidup (0.1,0.3);
            \draw[wipe] (-0.2,0.2)  \braidup (0.2,-0.2);
            \draw (0.4,0.3) -- (0.4,0.1) to[out=down,in=right] (0.2,-0.2) to[out=left,in=right] (-0.2,0.2) to[out=left,in=up] (-0.4,-0.1) -- (-0.4,-0.3) \botlabel{X^*};
        \end{tikzpicture}.
    \]
    This shows that $\bR_\OS(\posrightcross)$ is the braiding on $V^+ \otimes V^-$.  The argument for $\bR_\OS(\negleftcross)$ is analogous.  Then the last two equalities in \cref{obraid} follow from the properties of a braided monoidal category.

    The fact that $\bR_\OS$ respects \cref{oskein} follows easily from \cref{Vbraid,Vbraidinv}.
    \details{
        For $i \in \Vset$, we have
        \begin{multline*}
            \bR_\OS \left( \posupcross - \negupcross \right)(v_i^+ \otimes v_i^+)
            = (T_{++} - T_{++}^{-1})(v_i^+ \otimes v_i^+)
            \\
            = (-1)^{p(i)} ( q^{1-2p(i)} - q^{2p(i)-1} ) v_i^+ \otimes v_i^+
            = (q-q^{-1}) v_i^+ \otimes v_i^+
            = (q-q^{-1}) \bR_\OS
            \left(
                \begin{tikzpicture}[centerzero]
                    \draw[->] (-0.15,-0.2) -- (-0.15,0.2);
                    \draw[->] (0.15,-0.2) -- (0.15,0.2);
                \end{tikzpicture}
            \right)
            (v_i^+ \otimes v_i^+).
        \end{multline*}
        That
        \(
            \bR_\OS \left( \posupcross - \negupcross \right)(v_i^+ \otimes v_j^+) = (q-q^{-1}) \bR_\OS
            \left(
                \begin{tikzpicture}[centerzero]
                    \draw[->] (-0.15,-0.2) -- (-0.15,0.2);
                    \draw[->] (0.15,-0.2) -- (0.15,0.2);
                \end{tikzpicture}
            \right)
            (v_i^+ \otimes v_j^+)
        \)
        for $i \ne j$ is immediate.
    }
    For the first relation in \cref{ocurlbub}, first note that, since $V^+$ is simple, we have $\dim_\kk \End_\Ualg(V^+) = 1$.  Thus, there exists $c \in \kk$ such that
    \begin{equation} \label{shawarma}
        \bR_\OS
        \left(
            \begin{tikzpicture}[centerzero,xscale=-1]
                \draw (0,-0.4) to[out=up,in=180] (0.25,0.15) to[out=0,in=up] (0.4,0);
                \draw[wipe] (0.25,-0.15) to[out=180,in=down] (0,0.4);
                \draw[->] (0.4,0) to[out=down,in=0] (0.25,-0.15) to[out=180,in=down] (0,0.4);
            \end{tikzpicture}
        \right)
        = c 1_{V^+}
        = c\,\bR_\OS
        \left(
            \begin{tikzpicture}[centerzero]
                \draw[->] (0,-0.4) -- (0,0.4);
            \end{tikzpicture}
        \right).
    \end{equation}
    Using \cref{Vbraid,coev,ev-,button}, we compute that the action of the left-hand side of \cref{shawarma} on $v_{1-n}^+$ is given by
    \begin{align*}
        v_{1-n}^+
        &\xmapsto{\coev_+ \otimes 1} \sum_{i \in \Vset} (-1)^{i} q^{-|i|} v_i^- \otimes v_i^+ \otimes v_{1-n}^+
        \\
        &\xmapsto{1 \otimes T_{++}} \sum_{i > 1-n} (-1)^{i+p(i)p(1-n)} q^{-|i|} v_i^- \otimes v_{1-n}^+ \otimes v_i^+ \\
        &\qquad \qquad \qquad
        + (-1)^{1-n+p(1-n)} q^{-|1-n|+1-2p(1-n)} v_{1-n}^- \otimes v_{1-n}^+ \otimes v_{1-n}^+
        \\
        &\xmapsto{\ev^- \otimes 1} q^{m-|1-n|+1-2p(1-n)} v_j^+
        = q^{m-n} v_j^+.
    \end{align*}
    Thus, $c=q^{m-n}$, as desired.
    \details{
        If $n>0$, the exponent of $q$ is
        \[
            m-(n-1)+1-2 = m-n.
        \]
        If $n=0$, the exponent of $q$ is
        \[
            m-1+1 = m.
        \]
    }
    The proof that $\bR_\OS$ respects the second equality in \cref{ocurlbub} is similar.
    \details{
        We have
        \begin{align*}
            v_m^+
            &\xmapsto{1 \otimes \coev_-} \sum_{j \in \Vset} (-1)^{j+p(j)} q^{|j-1|-m-n} v_m^+ \otimes v_j^+ \otimes v_j^- \\
            &\xmapsto{T_{++} \otimes 1} \sum_{j<m} (-1)^{j+p(j)+p(m)p(j)} q^{|j-1|-m-n} v_j^+ \otimes v_m^+ \otimes v_j^- \\
            &\qquad \qquad \qquad
            + (-1)^m q^{|m-1|-m-n+1-2p(m)} v_m^+ \otimes v_m^+ \otimes v_m^- \\
            &\xmapsto{1 \otimes \ev^+} q^{|m-1|-n+1-2p(m)} v_m^+
            = q^{m-n} v_m^+.
        \end{align*}
        Above, we used the fact that, if $m>0$, the exponent of $q$ is
        \[
            m-1-n+1 = m-n,
        \]
        and, if $m=0$, the exponent of $q$ is
        \[
            1-n+1-2=-n.
        \]
    }

    To show that $\bR_\OS$ respects the last relation in \cref{ocurlbub}, we compute that
    \begin{multline*}
        \bR_\OS \left( \rightbub \right) \colon 1
        \xmapsto{\coev_-} \sum_{i \in \Vset} (-1)^{i+p(i)} q^{|i-1|-m-n} v_i^+ \otimes v_i^-
        \\
        \xmapsto{\ev^+} \sum_{i \in \Vset} (-1)^{p(i)} q^{|i|+|i-1|-m-n}
        = \frac{q^{m-n}-q^{n-m}}{q-q^{-1}},
    \end{multline*}
    as desired.

    The verification that $\bR_\OS$ respects \cref{adjunction} is a standard direct verification.
\end{proof}

The following result will be used in the proof of \cref{disincarnate}.

\begin{lem} \label{rodent}
    We have
    \begin{gather*}
        \bR_\OS(\negrightcross) (v_i^+ \otimes v_j^-)
        = (-1)^{p(i)p(j)} q^{-p(i,j)} v_j^- \otimes v_i^+ - \delta_{i=j} (q-q^{-1}) \sum_{k>i} (-q)^{|i|-|k|} v_k^- \otimes v_k^+,
%        \begin{cases}
%            (-1)^{p(i)p(j)} v_j^- \otimes v_i^+ & \text{if } i \ne j, \\
%            (-1)^{p(i)} q^{2p(i)-1} v_i^- \otimes v_i^+ -(q-q^{-1}) \sum_{k>i} (-q)^{|i|-|k|} v_k^- \otimes v_k^+ & \text{if } i=j,
%        \end{cases}
        \\
        \bR_\OS(\posdowncross) (v_i^- \otimes v_j^-)
        = (-1)^{p(i)p(j)} q^{p(i,j)} v_j^- \otimes v_i^- + \delta_{i>j} (q-q^{-1}) v_i^- \otimes v_j^-.
%        \begin{cases}
%            (-1)^{p(i)p(j)} v_j^- \otimes v_i^- + (q-q^{-1}) v_i^- \otimes v_j^- & \text{if } i>j, \\
%            (-1)^{p(i)} q^{1-2p(i)} v_i^- \otimes v_i^- & \text{if } i=j, \\
%            (-1)^{p(i)p(j)} v_j^- \otimes v_i^- & \text{if } i<j.
%        \end{cases}
    \end{gather*}
\end{lem}

\begin{proof}
    Using \cref{windmill,button}, we compute
    \begin{align*}
        \bR_\OS(\negrightcross)
        &\colon v_i^+ \otimes v_j^-
        \xmapsto{\coev_+ \otimes 1 \otimes 1} \sum_{k \in \Vset} (-q)^{-|k|} v_k^- \otimes v_k^+ \otimes v_i^+ \otimes v_j^-
        \\
        &\xmapsto{1 \otimes T_{++}^{-1} \otimes 1} \sum_{k \in \Vset} (-q)^{-|k|} v_k^- \otimes \big( (-1)^{p(k)p(i)} q^{-p(k,i)} v_i^+ \otimes v_k^+ - \delta_{k>i} (q-q^{-1}) v_k^+ \otimes v_i^+ \big) \otimes v_j^-
        \\
        &\xmapsto{1 \otimes 1 \otimes \ev^+} (-1)^{p(i)p(j)} q^{-p(i,j)} v_j^- \otimes v_i^+ - \delta_{i=j} (q-q^{-1}) \sum_{k>i} (-q)^{|i|-|k|} v_k^- \otimes v_k^+.
    \end{align*}
    The second equation is proved similarly.
    \details{
        We compute
        \begin{align*}
            \bR_\OS(\posrightcross)
            &\colon v_i^+ \otimes v_j^-
            \xmapsto{\coev_+ \otimes 1 \otimes 1} \sum_{k \in \Vset} (-q)^{-|k|} v_k^- \otimes v_k^+ \otimes v_i^+ \otimes v_j^-
            \\
            &\xmapsto{1 \otimes T_{++} \otimes 1} \sum_{k \in \Vset} (-q)^{-|k|} v_k^- \otimes \big( (-1)^{p(k)p(i)} q^{p(k,i)} v_i^+ \otimes v_k^+ + \delta_{k<i} (q-q^{-1}) v_k^+ \otimes v_i^+ \big) \otimes v_j^-
            \\
            &\xmapsto{1 \otimes 1 \otimes \ev^+} (-1)^{p(i)p(j)} q^{p(i,j)} v_j^- \otimes v_i^+ + \delta_{i=j} (q-q^{-1}) \sum_{k<i} (-q)^{|i|-|k|} v_k^- \otimes v_k^+.
        \end{align*}
        Therefore, we have
        \begin{align*}
            \bR_\OS(\posdowncross)
            &\colon v_i^- \otimes v_j^-
            \xmapsto{\coev_+ \otimes 1 \otimes 1} \sum_{k \in \Vset} (-q)^{-|k|} v_k^- \otimes v_k^+ \otimes v_i^- \otimes v_j^-
            \\
            &\mapsto \sum_{k \in \Vset} (-q)^{-|k|} v_k^- \otimes \left( (-1)^{p(k)p(i)} q^{p(k,i)} v_i^- \otimes v_k^+ + \delta_{k=i} (q-q^{-1}) \sum_{l<k} (-q)^{|k|-|l|} v_l^- \otimes v_l^+ \right) \otimes v_j^-
            \\
            &\xmapsto{1 \otimes 1 \otimes \ev^+} (-1)^{p(i)p(j)} q^{p(i,j)} v_j^- \otimes v_i^- + \delta_{i>j} (q-q^{-1})  v_i^- \otimes v_j^-.
        \end{align*}
    }
\end{proof}

As explained in the first paragraph of the proof of \cref{oincarnate}, the $V^+ \otimes V^-$ and $V^- \otimes V^-$ components of the braiding on $\Ualg\tmod$ are
\[
    T_{+-} := \bR_\OS(\negrightcross)
    \qquad \text{and} \qquad
    T_{--} := \bR_\OS(\posdowncross).
\]

%-----------------------------------------------
\subsection{The disoriented incarnation functor}
%-----------------------------------------------

The composition of functors
\begin{equation} \label{xenon}
    \Uis\tmod \times \OS(q,q^{m-2n})
    \xrightarrow{\id \times \bR_\OS} \Uis\tmod \times \Us\tmod
    \xrightarrow{\otimes} \Uis\tmod
\end{equation}
endows $\Uis\tmod$ with the structure of a right $\OS(q,q^{m-2n})$-module category.

The following lemma will allow us to simplify an argument in the proof of \cref{disincarnate}.

\begin{lem} \label{generate}
    As a $\Ui$-module (hence, also as a $\Uis$-module), $\Res(V^- \otimes V^-)$ is generated by
    \begin{equation} \label{genset}
        \{v_i^-\otimes v_j^- : \phi(i)\neq j\} \cup \{v_i^-\otimes v_i^-: i>0\}.
    \end{equation}
\end{lem}

\begin{proof}
    Let $W$ be the $\Ui$-submodule generated by the given vectors.  By the definition \cref{pump} of $\phi$, it suffices to show that
    \begin{equation} \label{bear}
        v_{-2k}^- \otimes v_{-2k-1}^- \in W
        \quad \text{and} \quad
        v_{-2k-1}^- \otimes v_{-2k}^- \in W
    \end{equation}
    for $0 \le k \le n-1$.  We prove this by induction on $k$.

    By \cref{tubeless,iodine1}, we have
    \[
        B_0 = F_0 + \va_0 (E_{-1}E_0-qE_0E_{-1}) K_0^{-1}.
    \]
    Thus
    \[
        B_0(v_{-1}^-\otimes v_1^-)
        \overset{\cref{comult}}{\underset{\cref{panda-}}{=}} - v_{-1}^-\otimes v_0^- - \va_0 v_1^-\otimes v_1^-.
    \]
    Since $v_{-1}^-\otimes v_1^-,v_1^-\otimes v_1^-\in W$, we conclude that $v_{-1}^-\otimes v_0^-\in W$. On the other hand, we have
    \[
        E_{-1}(v_{-1}^-\otimes v_{-1}^-)
        \overset{\cref{comult}}{=} (E_{-1}\otimes 1+K_{-1}\otimes E_{-1})(v_{-1}^-\otimes v_{-1}^-)=v_0^-\otimes v_{-1}^-+qv_{-1}^-\otimes v_0^-.
    \]
    Since $v_{-1}^-\otimes v_{-1}^-\in W$, we see that $v_0^-\otimes v_{-1}^-\in W$.  Thus, \cref{bear} holds for $k=0$.

    Next suppose that $1 \le k \le n-1$ and that \cref{bear} holds for $k-1$.  Let $\Ualg = \bigoplus_{\alpha \in X} \Ualg_\alpha$ be the usual grading on $\Ualg$ by the root lattice and, for $M \in \Ualg\tmod$ and a weight $\mu$, let $M_\mu$ denote the $\mu$-weight space of $M$.  Note that
    \[
        B_{-2k} \in F_{-2k} + \Ualg_{\epsilon_{-2k-1}-\epsilon_{-2k+2}}.
    \]
    Since
    \[
        F_{-2k}(v_{-2k+1}^-\otimes v_{-2k-1}^-)
        \overset{\cref{comult}}{=} (F_{-2k}\otimes K_{-2k}^{-1} + 1 \otimes F_{-2k})(v_{-2k+1}^- \otimes v_{-2k-1}^-)
        \overset{\cref{panda-}}{=} v_{-2k+1}^-\otimes v_{-2k}^-,
    \]
    we have
    \[
        B_{-2k} (v_{-2k+1}^-\otimes v_{-2k-1}^-)
        \in v_{-2k+1}^-\otimes v_{-2k}^- + (V^- \otimes V^-)_{-\epsilon_{-2k+1}-\epsilon_{-2k+2}}.
    \]
    Because $(V^- \otimes V^-)_{-\epsilon_{-2k+1}-\epsilon_{-2k+2}}$ is spanned by $v_{-2k+1}^- \otimes v_{-2k+2}^-$ and $v_{-2k+2}^- \otimes v_{-2k+1}^-$, both of which lie in $W$ by the inductive hypothesis, we conclude that $v_{-2k+1}^-\otimes v_{-2k-1}^-\in W$.

    Finally, we calculate that
    \begin{multline*}
        E_{-2k+1} (v_{-2k+1}^-\otimes v_{-2k+1}^-)
        \overset{\cref{comult}}{=} (E_{-2k+1}\otimes 1 + K_{-2k+1} \otimes E_{-2k+1})(v_{-2k+1}^-\otimes v_{-2k+1}^-)
        \\
        \overset{\cref{panda-}}{=} v_{-2k}^-\otimes v_{-2k+1}^- + q v_{-2k+1}^- \otimes v_{-2k}^-.
    \end{multline*}
    Since $v_{-2k+1}^- \otimes v_{-2k+1}^- \in W$ and $v_{-2k+1}^- \otimes v_{-2k}^- \in W$, we conclude that $v_{-2k}^- \otimes v_{-2k+1}^- \in W$ as well. Thus, \cref{bear} holds, completing the proof of the induction step.
\end{proof}

\begin{theo} \label{disincarnate}
    There is a unique strict morphism of $\OS(q,q^{m-2n})$-modules $\bR_\DS \colon \DS(q,q^{m-2n}) \to \Uis\tmod$ such that
    \begin{equation} \label{helium}
        \Bobj_0 \mapsto \Res(\triv),\qquad
        \togupdown \mapsto \varphi,\qquad
        \togdownup \mapsto \varphi^{-1}.
    \end{equation}
\end{theo}

\begin{proof}
    The property of being a strict morphism of $\OS(q,q^{m-2n})$-modules, together with the definition \cref{xenon}, means that the diagram
    \[
        \begin{tikzcd}
            \DS(q,q^{m-2n}) \times \OS(q,q^{m-2n}) \arrow[r,"\otimes"] \arrow[d,"\bR_\DS \times \bR_\OS"'] & \DS(q,q^{m-2n}) \arrow[d,"\bR_\DS"]
            \\
            \Uis\tmod \times \Us\tmod \arrow[r,"\otimes"] & \Uis\tmod
        \end{tikzcd}
    \]
    commutes, where the horizontal arrows are the right actions.  It follows that $\bR_\DS$ is uniquely determined by \cref{helium}, since $B_0$, $\togupdown$, and $\togdownup$ generate $\DS(q,q^{m-2n})$ as an $\OS(q,q^{m-2n})$-module.  In particular, $\bR_\DS(\Bobj_r) = \Res \left( (V^+)^{\otimes r} \right)$.  It remains to show that $\bR_\DS$ respects the relations \cref{DStoggles,DScurls}.  The first two relations in \cref{DStoggles} are immediate.

    Next we verify the first relation in \cref{DScurls}.  The image under $\bR_\DS$ of the left-hand side is the map
    \begin{align*}
        v_i^+ \otimes v_j^+
        &\xmapsto{T_{++}} (-1)^{p(i)p(j)} q^{p(i,j)} v_j^+ \otimes v_i^+ + \delta_{i<j} (q-q^{-1}) v_i^+ \otimes v_j^+
        \\
        &\xmapsto{\varphi^{-1} \otimes 1} (-1)^{p(i)p(j)} q^{p(i,j)} \tva_j^{-1} v_{\phi(j)}^- \otimes v_i^+ + \delta_{i<j} (q-q^{-1}) \tva_i^{-1} v_{\phi(i)}^- \otimes v_j^+
        \\
        &\xmapsto{\ev^-} \delta_{i,\phi(j)} \tva_i^{-1} \Big( (-1)^i q^{p(i,j)} q^{m+2n-|i-1|} + \delta_{i<j} (q-q^{-1}) (-1)^{j+p(i)} q^{m+2n-|j-1|} \Big)
        \\
        &= \delta_{i,\phi(j)} (-1)^i q^{m+2n-|j-1|+1} \tva_i^{-1}
    \end{align*}
    and
    \(
        \bR_\DS
        \left(
            \begin{tikzpicture}[anchorbase]
                \draw[->] (0.2,-0.3) -- (0.2,0.1) arc(0:180:0.2) -- (-0.2,0);
                \draw[<-] (-0.2,0) -- (-0.2,-0.3);
                \opendot{-0.2,0};
            \end{tikzpicture}
        \right)
    \)
    is the map
    \[
        v_i^+ \otimes v_j^+
        \xmapsto{\varphi^{-1} \otimes 1} \tva_i^{-1} v_{\phi(i)}^- \otimes v_j^+
        \xmapsto{\ev^-} \delta_{i,\phi(j)} (-1)^{j+p(i)} q^{m+2n-|j-1|} \tva_i^{-1}
        = \delta_{i,\phi(j)} (-1)^i q^{m+2n-|j-1|} \tva_i^{-1},
    \]
    verifying the first relation in \cref{DScurls}.

    Now consider the second relation in \cref{DScurls}.  The image under $\bR_\DS$ of the left-hand side is the map
    \begin{align*}
        1 &\xmapsto{\coev_+} \sum_{i \in \Vset} (-q)^{-|i|} v_i^- \otimes v_i^+
        \\
        &\xmapsto{\varphi \otimes 1} \sum_{i \in \Vset} (-q)^{-|i|} \tva_i v_{\phi(i)}^+ \otimes v_i^+
        \\
        &\xmapsto{T_{++}} \sum_{i \in \Vset} (-q)^{-|i|} \tva_i \left( (-1)^{p(i)} q^{p(\phi(i),i)} v_i^+ \otimes v_{\phi(i)}^+ + \delta_{\phi(i)<i} (q-q^{-1}) v_{\phi(i)}^+ \otimes v_i^+ \right)
        \\
        &= \sum_{i>0} (-q)^{1-|i|} \tva_i v_i^+ \otimes v_i^+
        - \sum_{i \in \{0,-2,\dotsc,2-2n\}} (-q)^i \tva_i \left( v_i^+ \otimes v_{i-1}^+ - (q-q^{-1}) v_{i-1}^+ \otimes v_i^+ \right)
        \\ &\qquad \qquad
        - \sum_{i \in \{-1,-3,\dotsc,1-2n\}} (-q)^i \tva_i v_i^+ \otimes v_{i+1}^+
        \\
        &= \sum_{i \in \Vset} (-q)^{1-|i|} \tva_i v_{\phi(i)}^+ \otimes v_i^+
    \end{align*}
    and
    \(
        \bR_\DS
        \left(
            \begin{tikzpicture}[anchorbase]
                \draw[<-] (0.2,0.3) -- (0.2,-0.1) arc(360:180:0.2) -- (-0.2,0);
                \draw[->] (-0.2,0) -- (-0.2,0.3);
                \opendot{-0.2,0};
            \end{tikzpicture}
        \right)
    \)
    is the map
    \[
        1 \xmapsto{\coev_+} \sum_{i \in \Vset} (-q)^{-|i|} v_i^- \otimes v_i^+
        \xmapsto{\varphi \otimes 1} \sum_{i \in \Vset} (-q)^{-|i|} \tva_i v_{\phi(i)}^+ \otimes v_i^+,
    \]
    verifying the second relation in \cref{DScurls}.

    It remains to verify the last relation in \cref{DStoggles}.  By \cref{generate}, it suffices to show that the images under $\bR_\DS$ of the two sides of the relation agree on the set \cref{genset}.  Note that $p(\phi(i)) = p(i)$ for all $i \in \Vset$.  Using \cref{rodent}, we compute, for $i>0$,
    \begin{multline*}
        \bR_\DS
        \left(
            \begin{tikzpicture}[centerzero]
                \draw[<-] (0.2,-0.6) -- (0.2,-0.4) \braidup (-0.2,0);
                \draw[wipe] (-0.2,-0.6) -- (-0.2,-0.4) \braidup (0.2,0);
                \draw[<-] (-0.2,-0.6) -- (-0.2,-0.4) \braidup (0.2,0);
                \draw[->] (0.2,0) \braidup (-0.2,0.4) -- (-0.2,0.6);
                \draw[wipe] (-0.2,0) \braidup (0.2,0.4) -- (0.2,0.6);
                \draw[->] (-0.2,0) \braidup (0.2,0.4) -- (0.2,0.6);
                \opendot{-0.2,0.4};
                \opendot{-0.2,0};
            \end{tikzpicture}
        \right)
        \colon v_i^- \otimes v_i^-
        \xmapsto{T_{--}} q v_i^- \otimes v_i^-
        \xmapsto{\varphi \otimes 1} q \tva_i v_i^+ \otimes v_i^-
        \xmapsto{T_{+-}} \tva_i v_i^- \otimes v_i^+ + (1-q^2) \tva_i \sum_{k>i} (-q)^{i-k} v_k^- \otimes v_k^+
        \\
        \xmapsto{\varphi \otimes 1} \tva_i^2 v_i^- \otimes v_i^+ + (1-q^2) \tva_i \sum_{k>i} (-q)^{i-k} \tva_k v_k^- \otimes v_k^+,
    \end{multline*}
    and
    \begin{multline*}
        \bR_\DS
            \left(
                \begin{tikzpicture}[centerzero]
                \draw[<-] (0.2,-0.6) -- (0.2,-0.4) \braidup (-0.2,0);
                \draw[wipe] (-0.2,-0.6) -- (-0.2,-0.4) \braidup (0.2,0);
                \draw[<-] (-0.2,-0.6) -- (-0.2,-0.4) \braidup (0.2,0);
                \draw[->] (0.2,0) \braidup (-0.2,0.4) -- (-0.2,0.6);
                \draw[wipe] (-0.2,0) \braidup (0.2,0.4) -- (0.2,0.6);
                \draw[->] (-0.2,0) \braidup (0.2,0.4) -- (0.2,0.6);
                \opendot{-0.2,-0.4};
                \opendot{-0.2,0};
            \end{tikzpicture}
        \right)
        \colon v_i^- \otimes v_i^-
        \xmapsto{\varphi \otimes 1} \tva_i v_i^+ \otimes v_i^-
        \xmapsto{T_{+-}} q^{-1} \tva_i v_i^- \otimes v_i^+ - (q-q^{-1}) \tva_i \sum_{k>i} (-q)^{i-k} v_k^- \otimes v_k^+
        \\
        \xmapsto{\varphi \otimes 1} q^{-1} \tva_i^2 v_i^+ \otimes v_i^+ - (q-q^{-1}) \tva_i \sum_{k>i} (-q)^{i-k} \tva_k v_k^+ \otimes v_k^+
        \xmapsto{T_{++}} \tva_i^2 v_i^+ \otimes v_i^+ + (1-q^2) \tva_i \sum_{k>i} (-q)^{i-k} \tva_k v_k^+ \otimes v_k^+,
    \end{multline*}
    which agree.  If $\phi(i) \ne j$, then
    \begin{align*}
        \bR_\DS
        \left(
            \begin{tikzpicture}[centerzero]
                \draw[<-] (0.2,-0.6) -- (0.2,-0.4) \braidup (-0.2,0);
                \draw[wipe] (-0.2,-0.6) -- (-0.2,-0.4) \braidup (0.2,0);
                \draw[<-] (-0.2,-0.6) -- (-0.2,-0.4) \braidup (0.2,0);
                \draw[->] (0.2,0) \braidup (-0.2,0.4) -- (-0.2,0.6);
                \draw[wipe] (-0.2,0) \braidup (0.2,0.4) -- (0.2,0.6);
                \draw[->] (-0.2,0) \braidup (0.2,0.4) -- (0.2,0.6);
                \opendot{-0.2,0.4};
                \opendot{-0.2,0};
            \end{tikzpicture}
        \right)
        \colon v_i^- \otimes v_j^-
        &\xmapsto{T_{--}} (-1)^{p(i)p(j)} q^{p(i,j)} v_j^- \otimes v_i^- + \delta_{i>j} (q-q^{-1}) v_i^- \otimes v_j^-
        \\
        &\xmapsto{\varphi \otimes 1} (-1)^{p(i)p(j)} q^{p(i,j)} \tva_j v_{\phi(j)}^+ \otimes v_i^- + \delta_{i>j} (q-q^{-1}) \tva_i v_{\phi(i)}^+ \otimes v_j^-
        \\
        &\xmapsto{T_{+-}} q^{p(i,j)} \tva_j v_i^- \otimes v_{\phi(j)}^+ + \delta_{i>j} (-1)^{p(i)p(j)} (q-q^{-1}) \tva_i v_j^- \otimes v_{\phi(i)}^+
        \\
        &\xmapsto{\varphi \otimes 1} q^{p(i,j)} \tva_i \tva_j v_{\phi(i)}^+ \otimes v_{\phi(j)}^+ + \delta_{i>j} (-1)^{p(i)p(j)} (q-q^{-1}) \tva_i \tva_j v_{\phi(j)}^+ \otimes v_{\phi(i)}^+
    \end{align*}
    and
    \begin{align*}
        \bR_\DS
            \left(
                \begin{tikzpicture}[centerzero]
                \draw[<-] (0.2,-0.6) -- (0.2,-0.4) \braidup (-0.2,0);
                \draw[wipe] (-0.2,-0.6) -- (-0.2,-0.4) \braidup (0.2,0);
                \draw[<-] (-0.2,-0.6) -- (-0.2,-0.4) \braidup (0.2,0);
                \draw[->] (0.2,0) \braidup (-0.2,0.4) -- (-0.2,0.6);
                \draw[wipe] (-0.2,0) \braidup (0.2,0.4) -- (0.2,0.6);
                \draw[->] (-0.2,0) \braidup (0.2,0.4) -- (0.2,0.6);
                \opendot{-0.2,-0.4};
                \opendot{-0.2,0};
            \end{tikzpicture}
        \right)
        \colon v_i^- \otimes v_j^-
        &\xmapsto{\varphi \otimes 1} \tva_i v_{\varphi(i)}^+ \otimes v_j^-
        \\
        &\xmapsto{T_{+-}} (-1)^{p(i)p(j)} \tva_i v_j^- \otimes v_{\phi(i)}^+
        \\
        &\xmapsto{\varphi \otimes 1} (-1)^{p(i)p(j)} \tva_i \tva_j v_{\phi(j)}^+ \otimes v_{\phi(i)}^+
        \\
        &\xmapsto{T_{++}} q^{p(i,j)} \tva_i \tva_j v_{\phi(i)}^+ \otimes v_{\phi(j)}^+ + \delta_{\phi(i)>\phi(j)} (-1)^{p(i)p(j)} (q-q^{-1}) \tva_i \tva_j v_{\phi(j)}^+ \otimes v_{\phi(i)}^+,
    \end{align*}
    which also agree since the condition $\phi(i) \ne j$ implies that $i>j \iff \phi(i)>\phi(j)$.
    \details{
        We can also compute a general expression for the maps considered above.  Using \cref{rodent}, we compute that the image under $\bR_\DS$ of the right-hand side of the last relation in \cref{DStoggles} is the map
        \begin{align*}
            v_i^- &\otimes v_j^-
            \xmapsto{T_{--}} (-1)^{p(i)p(j)} q^{p(i,j)} v_j^- \otimes v_i^-
            + \delta_{i>j} (q-q^{-1}) v_i^- \otimes v_j^-
            \\
            &\xmapsto{\varphi \otimes 1} (-1)^{p(i)p(j)} q^{p(i,j)} \tva_j v_{\phi(j)}^+ \otimes v_i^-
            + \delta_{i>j} (q-q^{-1}) \tva_i v_{\phi(i)}^+ \otimes v_j^-
            \\
            &\xmapsto{T_{+-}} q^{p(i,j)-p(i,\phi(j))} \tva_j v_i^- \otimes v_{\phi(j)}^+
            - \delta_{i,\phi(j)} (-1)^{p(i)p(j)} q^{p(i,j)} (q-q^{-1}) \tva_j \sum_{k > i} (-q)^{|i|-|k|} v_k^- \otimes v_k^+ \\
            & \quad + \delta_{i>j} (-1)^{p(i)p(j)} q^{-p(\phi(i),j)} (q-q^{-1}) \tva_i v_j^- \otimes v_{\phi(i)}^+
            - \delta_{i>j} \delta_{\phi(i),j} (q-q^{-1})^2 \tva_i \sum_{k > \phi(i)} (-q)^{|\phi(i)|-|k|} v_k^- \otimes v_k^+
            \\
            &\xmapsto{\varphi \otimes 1} q^{p(i,j)-p(i,\phi(j))} \tva_i \tva_j v_{\phi(i)}^+ \otimes v_{\phi(j)}^+
            - \delta_{i,\phi(j)} (-1)^{p(i)p(j)} q^{p(i,j)} (q-q^{-1}) \tva_j \sum_{k > i} (-q)^{|i|-|k|} \tva_k v_{\phi(k)}^+ \otimes v_k^+ \\
            & \quad + \delta_{i>j} (-1)^{p(i)p(j)} q^{-p(\phi(i),j)} (q-q^{-1}) \tva_i \tva_j v_{\phi(j)}^+ \otimes v_{\phi(i)}^+
            - \delta_{i>j} \delta_{\phi(i),j} (q-q^{-1})^2 \tva_i \sum_{k > \phi(i)} (-q)^{|\phi(i)|-|k|} \tva_k v_{\phi(k)}^+ \otimes v_k^+.
        \end{align*}
        On the other hand, the image under $\bR_\DS$ of the left-hand side of the last relation in \cref{DStoggles} is the map
        \begin{align*}
            v_i^- \otimes v_j^-
            &\xmapsto{\varphi \otimes 1}
            \tva_i v_{\phi(i)}^+ \otimes v_j^-
            \\
            &\xmapsto{T_{+-}} (-1)^{p(i)p(j)} q^{-p(\phi(i),j)} \tva_i v_j^- \otimes v_{\phi(i)}^+
            - \delta_{\phi(i),j} (q-q^{-1}) \tva_i \sum_{k > \phi(i)} (-q)^{|\phi(i)|-|k|} v_k^- \otimes v_k^+
            \\
            &\xmapsto{\varphi \otimes 1} (-1)^{p(i)p(j)} q^{-p(\phi(i),j)} \tva_i \tva_j v_{\phi(j)}^+ \otimes v_{\phi(i)}^+
            - \delta_{\phi(i),j} (q-q^{-1}) \tva_i \sum_{k > \phi(i)} (-q)^{|\phi(i)|-|k|} \tva_k v_{\phi(k)}^+ \otimes v_k^+
            \\
            &\xmapsto{T_{++}} q^{p(i,j)-p(\phi(i),j)} \tva_i \tva_j v_{\phi(i)}^+ \otimes v_{\phi(j)}^+
            + \delta_{\phi(j)<\phi(i)} (-1)^{p(i)p(j)} q^{-p(\phi(i),j)} (q-q^{-1}) \tva_i \tva_j v_{\phi(j)}^+ \otimes v_{\phi(i)}^+
            \\ &\quad - \delta_{\phi(i),j} (q-q^{-1}) \tva_i \sum_{k > \phi(i)} (-1)^{p(k)} q^{\delta_{k>0}} (-q)^{|\phi(i)|-|k|} \tva_k v_k^+ \otimes v_{\phi(k)}^+
            \\ &\quad - \delta_{\phi(i),j} (q-q^{-1})^2 \tva_i \sum_{k > \phi(i)} \delta_{\phi(k)<k} (-q)^{|\phi(i)|-|k|} \tva_k v_{\phi(k)}^+ \otimes v_k^+.
        \end{align*}
        It is possible to show directly that these two maps are equal.  However, the argument based on \cref{generate} is simpler.
    }
\end{proof}

\begin{rem}\label{reflection}
    If either $m = 0$ or $n = 0$, then the preservation of the last relation in \cref{DStoggles} under $\bR_\DS$ follows from the reflection equation \cite[(9.16)]{BK19} between the universal $R$-matrix of $\Ualg$ and the universal $K$-matrix of $\Ui$.  Indeed, let $\tau_0$ be the algebra automorphism of $\Ualg$ corresponding to the nontrivial involution of its Dynkin diagram.  By \cite[Cor.~7.7]{BK19}, the universal $K$-matrix $\mathcal{K}$ defines an isomorphism of $\Ui$-modules
    \[
        \mathcal{K}_M \colon M \to M^{\tau_0}
    \]
    for any finite-dimensional $\Ualg$-module $M$, where $M^{\tau_0}$ has the same underlying vector space as $M$ but with the action given by $(x, v) \mapsto \tau_0(x) v$ for $x \in \Ualg$ and $v \in M^{\tau_0}$.  (Recall that the diagram automorphism $\tau$, which is part of the data of the super Satake diagram \cref{satake}, is the identity in our case.)  It follows that the universal $K$-matrix $\mathcal{K}$ defines a $\Ui$-isomorphism
    \[
        \mathcal{K}_{V^-} \colon V^- \to (V^-)^{\tau_0}.
    \]
    However, it is straightforward to check that $(V^-)^{\tau_0} \cong V^+$ as $\Ualg$-modules.  Since $V^- \cong V^+$ as $\Ui$-modules and both are irreducible, Schur's Lemma implies that one can identify $\mathcal{K}_{V^-}$ with the map $\varphi$, up to a scalar. Then the last relation in~\cref{DStoggles} follows from the reflection equation established in \cite[(9.16)]{BK19}.  The proof of \cref{disincarnate} could be made significantly shorter if the machinery of the $K$-matrix existed in the super setting; unfortunately, this is not yet the case.
\end{rem}

%--------------------------------------------
\subsection{The iquantum incarnation functor}
%--------------------------------------------

Define
\[
    V_r = \Res \left( (V^+)^{\otimes r} \right), \qquad r \in \N.
\]

\begin{prop}
    The composite morphism of $\OS(q,q^{m-2n})$-modules
    \begin{equation}\label{iqBincar}
        \bR_\iqB := \bR_\DS \bG \colon \iqB(q,q^{m-2n}) \to \Uis\tmod
    \end{equation}
    is given on objects by $\Bobj_r \mapsto V_r$, $r \in \N$, and on morphisms by
    \begin{gather*}
        \begin{tikzpicture}[centerzero]
            \draw (-0.15,0.2) -- (-0.15,0.1) arc(180:360:0.15) -- (0.15,0.2);
            \draw[multi] (0.45,-0.2) \botlabel{r} -- (0.45,0.2);
        \end{tikzpicture}
        \mapsto \big( (\varphi \otimes 1_{V_1}) \circ \coev_+ \big) \circ \otimes 1_{V_r},
        \qquad
        \begin{tikzpicture}[centerzero]
            \draw (-0.15,-0.2) -- (-0.15,-0.1) arc(180:0:0.15) -- (0.15,-0.2);
            \draw[multi] (0.45,-0.2) \botlabel{r} -- (0.45,0.2);
        \end{tikzpicture}
        \mapsto \big( \ev^- \circ (\varphi^{-1} \otimes 1_{V_1}) \big) \otimes 1_{V_r},
        \\
        \thickstrand{r} \poscross\ \thickstrand{s}
        \mapsto 1_{V_r} \otimes T_{++} \otimes 1_{V_s},
        \qquad
        \thickstrand{r} \negcross\ \thickstrand{s}
        \mapsto 1_{V_r} \otimes T_{++}^{-1} \otimes 1_{V_s}.
    \end{gather*}
\end{prop}

\begin{proof}
    This follows immediately from the definitions of $\bR_\DS$ and $\bG$.
\end{proof}

Note that $\bR_\iqB$ is not a \emph{strict} morphism of $\OS(q,q^{m-2n})$-modules, since $\bG$ is not; see \cref{laser}.  The map
\[
    \bR_\iqB
    \left(\,
        \begin{tikzpicture}[centerzero]
            \draw (-0.15,0.3) -- (-0.15,0.2) arc(180:360:0.15) -- (0.15,0.3);
            \draw (-0.15,-0.3) -- (-0.15,-0.2) arc(180:0:0.15) -- (0.15,-0.3);
        \end{tikzpicture}
    \, \right)
    \colon V_2 \to V_2
\]
recovers the operator $\Xi$ defined in \cite[\S7.3]{SW24}.

%=================================================
\section{Basis theorems and endomorphism algebras}
%=================================================

In this section, we describe bases for the morphism spaces of $\DS(q,t)$ and $\iqB(q,t)$.  Throughout this section $\kk$ is an arbitrary integral domain and $q,t$ are elements of $\kk^\times$ such that $t-t^{-1}$ is divisible by $q-q^{-1}$.
%-------------------------
\subsection{Basis theorem}
%-------------------------

For $\lambda,\mu \in \word$, a \emph{$(\lambda,\mu)$-matching} is a partition of the set
\[
    E(\lambda,\mu) := \{-1,-2,\dotsc,-\ell(\lambda),1,2,\dotsc,\ell(\mu)\}
\]
into subsets of size two.  In a string diagram representing a morphism in $\Hom_{\DS(q,t)}(\lambda,\mu)$, each non-closed string has two endpoints.  We view these endpoints as elements of $E(\lambda,\mu)$ by numbering the $\upobj$, $\downobj$ in $\lambda$ by $-1,-2,\dotsc,-\ell(\lambda)$ from left to right and the $\upobj$, $\downobj$ in $\mu$ by $1,2,\dotsc,\ell(\mu)$ from left to right.  A \emph{reduced $(\lambda,\mu)$-diagram} for a given $(\lambda,\mu)$-matching is a string diagram representing a morphism in $\Hom_\DS(\lambda,\mu)$ such that:
\begin{itemize}
    \item the endpoints of each string correspond under the given matching (i.e., they form one of the two-element subsets of the partition);
    \item there are no closed strings (i.e., strings with no endpoints);
    \item no string has more than one critical point (i.e., more than one point at which the tangent is horizontal);
    \item there are no self-intersections of strings and no two strings cross each other more than once;
    \item all toggles on strings connecting the top and bottom of the diagram are near their bottom endpoint.
    \item all toggles on strings with both endpoints at the top of the diagram or both endpoints at the bottom of the diagram are near the left endpoint of the string;
\end{itemize}
For example, for $\lambda = \downobj \upobj \downobj \downobj \downobj \upobj \upobj$ and $\mu = \upobj \upobj \upobj \upobj \downobj$,
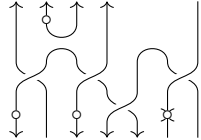
\[
    \begin{tikzpicture}
        \draw[->] (0.4,-0.2) -- (0.4,0.4) \braidup (0,0.8) -- (0,1.6);
        \draw[wipe] (0,0.4) \braidup (0.4,0.8);
        \draw[<->] (0,-0.2) -- (0,0.4) \braidup (0.4,0.8) to[out=up,in=up,looseness=1.5] (0.8,0.8) \braiddown (1.2,0.4) \braiddown (1.6,0) -- (1.6,-0.2);
        \draw[wipe] (0.8,0.4) \braidup (1.2,0.8);
        \draw[<->] (0.8,-0.2)  -- (0.8,0.4) \braidup (1.2,0.8) -- (1.2,1.6);
        \draw[wipe] (1.2,0) \braidup (1.6,0.4);
        \draw[<-] (1.2,-0.2) -- (1.2,0) \braidup (1.6,0.4) -- (1.6,0.8) to[out=up,in=up,looseness=1.5] (2,0.8) \braiddown (2.4,0.4) -- (2.4,-0.2);
        \draw[->] (2,-0.2) -- (2,0.1);
        \draw[wipe] (2,0.4) \braidup (2.4,0.8);
        \draw[<-] (2,0.1) -- (2,0.4) \braidup (2.4,0.8) -- (2.4,1.6);
        \draw[<->] (0.4,1.6) -- (0.4,1.3) to[out=down,in=down,looseness=1.5] (0.8,1.3) -- (0.8,1.6);
        \opendot{0,0.1};
        \opendot{2,0.1};
        \opendot{0.8,0.1};
        \opendot{0.4,1.35};
    \end{tikzpicture}
\]
is a reduced $(\lambda,\mu)$-diagram for the $(\lambda,\mu)$-matching
\[
    \big\{ \{-1,-5\}, \{-2,1\}, \{-3,4\}, \{-4,-7\}, \{-6,5\}, \{2,3\} \big\}.
\]
Fix a set $M_\DS(\lambda,\mu)$ consisting of a choice of reduced $(\lambda,\mu)$-diagram for each of the $(\lambda,\mu)$-matchings.

\begin{theo} \label{basis}
    For objects $\lambda,\mu \in \DS(q,t)$, the morphism space $\Hom_{\DS(q,t)}(\lambda,\mu)$ is a free $\kk$-module with basis $M_\DS(\lambda,\mu)$.
\end{theo}

\begin{proof}
    Set $\DS = \DS(q,t)$.  It suffices to consider the ground ring
    \[
        \kk = \Z[q^{\pm 1}, t^{\pm 1}, (t-t^{-1})/(q-q^{-1})],
    \]
    since the more general case then follows by base extension.  The defining relations and the additional relations deduced in \cref{subsec:OS,subsec:DS} give Reidemeister-type relations modulo terms with fewer crossings, plus a skein relation.  The relation \cref{togcross}, together with the skein relation \cref{oskein}, implies that we can slide toggles through crossings, modulo diagrams with fewer crossings.  Furthermore, we can slide toggles over cups and caps, modulo diagrams with fewer crossings, by moving cups and caps to the left of the diagram, then using \cref{togcross,togcupcap}.  These relations allow us to transform diagrams for morphisms in $\DS$ in way similar to the way oriented tangles are simplified in skein categories, modulo diagrams with fewer crossings.  Therefore, there is a straightening algorithm to rewrite any diagram representing a morphism $\lambda \to \mu$ as a linear combination of the ones in $M_\DS(\lambda,\mu)$.

    It remains to prove linear independence.  For $\mu \in \word$, let $\mu'$ denote the word obtained from $\mu$ by rotating $180\degree$.  For example, $(\upobj \upobj \downobj \upobj \downobj)' = (\upobj \downobj \upobj \downobj \downobj)$.  For $\lambda, \mu \in \word$, we have a $\kk$-linear isomorphism
    \begin{equation}\label{kiko}
        \Hom_\DS(\lambda,\mu) \to \Hom_\DS(\lambda \otimes \mu', \one),
        \qquad
        \begin{tikzpicture}[centerzero]
            \draw[multi] (0,-0.5) \botlabel{\lambda} -- (0,0.5) \toplabel{\mu};
            \coupon{0,0}{f};
        \end{tikzpicture}
        \mapsto
        \begin{tikzpicture}[centerzero]
            \draw[multi] (0,-0.5) \botlabel{\lambda} -- (0,0.2) to[out=up,in=up,looseness=2] (0.5,0.2) -- (0.5,-0.5) \botlabel{\mu'};
            \coupon{0,0}{f};
        \end{tikzpicture}
        ,
    \end{equation}
    with inverse
    \[
        \Hom_\DS(\lambda \otimes \mu', \one) \to \Hom_\DS(\lambda,\mu),
        \qquad
        \begin{tikzpicture}[centerzero]
            \draw[multi] (-0.2,-0.5) \botlabel{\lambda} -- (-0.2,0);
            \draw[multi] (0.2,-0.5) \botlabel{\mu'} -- (0.2,0);
            \draw[fill=white,rounded corners] (-0.4,-0.2) rectangle (0.4,0.2);
            \draw (0,0) node {\strandlabel{f}};
        \end{tikzpicture}
        \mapsto
        \begin{tikzpicture}[centerzero]
            \draw[multi] (-0.2,-0.5) \botlabel{\lambda} -- (-0.2,0);
            \draw[multi] (0.2,-0.2) to[out=down,in=down,looseness=2] (0.6,-0.2) -- (0.6,0.5) \toplabel{\mu};
            \draw[fill=white,rounded corners] (-0.4,-0.2) rectangle (0.4,0.2);
            \draw (0,0) node {\strandlabel{f}};
        \end{tikzpicture}
        .
    \]
    Via the straightening algorithm mentioned above (e.g., using \cref{adjunction} to remove instances of multiple critical points and \cref{oskein} to flip crossings if necessary), there is a bijection from $M_\DS(\lambda,\mu)$ to ${M_\DS(\lambda \otimes \mu', \one)}$ such that \cref{kiko} acts by this bijection modulo diagrams with fewer crossings.  Thus, it suffices to prove the linear independence of $M_\DS(\lambda,\one)$.  Since $\Hom_\DS(\lambda,\one) = 0$ when $\ell(\lambda) \notin 2\N$, we assume that $\ell(\lambda) \in 2\N$.

    We will use the disoriented incarnation functor $\bR_\DS$ in the case that $n=0$, which we assume throughout this proof.  We also make the choice \cref{oilers} for the $\va_i$.  This implies that $\tva_i = 1$ for all $i \in \Vset$; see \cref{kowloon}.
    
    For $\mu \in \word$ and $1 \le a \le \ell(\mu)$, let
    \begin{equation} \label{unicorn}
        \pi_a(\mu)
        =
        \begin{cases}
            + &\text{if the $a$\,th term in $\mu$ is $\upobj$}, \\
            - &\text{if the $a$\,th term in $\mu$ is $\downobj$}.
        \end{cases}
    \end{equation}
    Then, for $\bi = (i_1,\dotsc,i_{\ell(\mu)}) \in \Vset^{\ell(\mu)}$, let
    \[
        v_{\bi,\mu} = v_{i_1}^{\pi_1(\mu)} \otimes v_{i_2}^{\pi_2(\mu)} \otimes \dotsb \otimes v_{i_{\ell(\mu)}}^{\pi_{\ell(\mu)}(\mu)}
        \in \bR_\DS(\mu).
    \]
    We define
    \[
        W_\mu = \sum_{\bi \in \Vset^{\ell(\mu)}} \Z[q^{\pm 1}] v_{\bi,\mu} \subseteq \bR_\DS(\mu).
    \]

    For a morphism $f \colon \mu \to \nu$ in $\DS(q,t)$, define
    \[
        \bR_\DS^W(f) = \bR_\DS(f)|_{W_\mu}.
    \]
    Let $J$ be the ideal of $\Z[q^{\pm 1}]$ generated by $q-1$.  It follows from \cref{Vbraid,Vbraidinv} that, for $\mu,\nu \in \word$,
    \[
        \bR_\DS
        \left(
            \thickstrand{\mu} \posupcross\ \thickstrand{\nu}
        \right)
        (W_{\mu \upobj \upobj \nu}) \subseteq W_{\mu \upobj \upobj \nu}
    \]
    and
    \[
        \bR_\DS^W
        \left(
            \thickstrand{\mu} \posupcross\ \thickstrand{\nu}
        \right)
        = 1_{W_\mu} \otimes \flip \otimes 1_{W_\nu} \mod JW_{\mu \upobj \upobj \nu},
    \]
    where $\flip \colon u \otimes w \to w \otimes u$ is the tensor flip.  Similarly, the other crossings also induce the flip map.  Thus, using also \cref{ev-,ev+,varphidef}, we see that, for $f \in M_\DS(\lambda,\one)$ and $\bi = (i_1,\dotsc,i_{\ell(\lambda)}) \in \Vset^{\ell(\lambda)}$,
    \[
        \bR_\DS(f) (v_{\bi,\lambda})
        =
        \begin{cases}
            \pm 1 \mod J & \text{if } i_a = i_b \text{ whenever $a$ and $b$ are joined by a string in $f$}, \\
            0 \mod J & \text{otherwise}.
        \end{cases}
    \]
    Now suppose that $m \ge \ell(\lambda)$ and that
    \begin{equation} \label{gravel}
        \sum_{f \in M_\DS(\lambda,\one)} \gamma_f \bR_\DS(f) = 0
        \qquad \text{for some } \gamma_f \in \C(q).
    \end{equation}
    Then, for each $f \in M_\DS(\lambda,\one)$, evaluating \cref{gravel} at $v_{\bi,\lambda}$ for some $\bi$ whose entries are equal if and only if are matched under the matching corresponding to $f$ shows that $\gamma_f = 0$.  Thus, the $\bR_\DS(f)$, $f \in M_\DS(\lambda,\one)$, are linearly independent if $m \ge \ell(\lambda)$.

    Now suppose that $\lambda \in \word$ and that
    \[
        \sum_{f \in M_\DS(\lambda,\one)} c_f(q,t) f = 0
        \qquad \text{for some } c_f(q,t) \in \kk.
    \]
    Multiplying by an appropriate power of $(q-q^{-1})$, we may assume that $c_f(q,t) \in \Z[q^{\pm 1}, t^{\pm 1}]$ for all $f \in M_\DS(\lambda,\one)$.  Applying $\bR_\DS$, we have
    \[
        \sum_{f \in M_\DS(\lambda,\one)} c_f(q,q^m) \bR_\DS(f) = 0
        \qquad \text{for all } m \ge \ell(\lambda).
    \]
    Since the $\bR_\DS(f)$, $f \in M_\DS(\lambda,\one)$, are linearly independent, we have $c_f(q,q^m)=0$ for all $f \in M_\DS(\lambda,\one)$ and $m \ge \ell(\lambda)$.  As there are infinitely many choices of $m$, it follows that $c_f(q,t)=0$ for all $f$.  Thus, the set $M_\DS(\lambda,\one)$ is linearly independent, as desired.
\end{proof}

For $r,s \in \N$, define
\[
    M_\iqB(r,s) = \bF(M_\DS(\upobj^r,\upobj^s)).
\]

\begin{cor} \label{iqBbasis}
    For $r,s \in \N$, the morphism space $\Hom_{\iqB(q,t)}(\Bobj_r,\Bobj_s)$ is a free $\kk$-module with basis $M_\iqB(r,s)$.
\end{cor}

\begin{proof}
    This follows immediately from \cref{equivthm}.
\end{proof}

Let $A \subseteq \C(q) = \kk$ be the $\C$-subalgebra of rational functions regular at $q=1$.  This is a local ring with maximal ideal $\fm = A(q-1)$.  In what follows, we identify $\C$ with $A/\fm$.  For $d \in \C$, let $\Brauer(d)$ denote the Brauer category, as described in \cite[Th.~2.6]{LZ15}.

\begin{cor} \label{Guilin}
    There is an isomorphism of $\C$-linear categories from $\iqB_A(q,q^{m-2n}) \otimes_A \C$ to the Brauer category $\Brauer(m-2n)$ given by
    \[
        \begin{tikzpicture}[centerzero]
            \draw (-0.15,0.2) -- (-0.15,0.1) arc(180:360:0.15) -- (0.15,0.2);
            \draw[multi] (0.45,-0.2) \botlabel{r} -- (0.45,0.2);
        \end{tikzpicture}
        \mapsto
        \begin{tikzpicture}[centerzero]
            \draw (-0.15,0.2) -- (-0.15,0.1) arc(180:360:0.15) -- (0.15,0.2);
            \draw[multi] (0.45,-0.2) \botlabel{r} -- (0.45,0.2);
        \end{tikzpicture}
        \ ,\quad
        \begin{tikzpicture}[centerzero]
            \draw (-0.15,-0.2) -- (-0.15,-0.1) arc(180:0:0.15) -- (0.15,-0.2);
            \draw[multi] (0.45,-0.2) \botlabel{r} -- (0.45,0.2);
        \end{tikzpicture}
        \mapsto
        \begin{tikzpicture}[centerzero]
            \draw (-0.15,-0.2) -- (-0.15,-0.1) arc(180:0:0.15) -- (0.15,-0.2);
            \draw[multi] (0.45,-0.2) \botlabel{r} -- (0.45,0.2);
        \end{tikzpicture}
        \ ,\quad
        \thickstrand{r} \poscross\ \thickstrand{s},\
        \thickstrand{r} \negcross\ \thickstrand{s}
        \mapsto
        \thickstrand{r}
        \begin{tikzpicture}[centerzero]
            \draw (0.2,-0.2) -- (-0.2,0.2);
            \draw (-0.2,-0.2) -- (0.2,0.2);
        \end{tikzpicture}
        \ \thickstrand{s},
        \qquad r,s \in \N.
    \]
\end{cor}

\begin{proof}
    It is straightforward to verify that the given functor is well defined (that is, the relations in \cref{iqBdef} are satisfied).  It follows from \cite[Th.~2.6]{LZ15} that, for each $r,s \in \N$ the image of $M_\iqB(r,s)$ under the above functor is a basis of $\Brauer(m-2n)$.  Here we use the fact that
    \[
        \bF \left( \thickstrand{\lambda} \leftcup \thickstrand{\mu} \right)
        = q^{-1} t
        \begin{tikzpicture}[centerzero,cscale]
            \draw[multi] (-0.4,0.2) \braidup (-0.8,0.6) \toplabel{\ell(\lambda)};
            \draw[wipe] (-0.4,0.6) \braiddown (-0.8,0.2);
            \draw (-0.4,0.6) \braiddown (-0.8,0.2)  -- (-0.8,-0.2) to[out=down,in=down,looseness=2] (-0.4,-0.2) \braidup (0,0.2) -- (0,0.6);
            \draw[wipe] (0,-0.2) \braidup (-0.4,0.2);
            \draw[multi] (0,-0.6) -- (0,-0.2) \braidup (-0.4,0.2);
            \draw[multi] (0.4,-0.6) -- (0.4,0.6) \toplabel{\ell(\mu)};
        \end{tikzpicture}
        \text{ is equal to }
        \begin{tikzpicture}[centerzero,cscale]
            \draw[multi] (-0.8,-0.6) \braidup (-0.8,0.6) \toplabel{\ell(\lambda)};
            \draw (-0.4,0.6) -- (-0.4,0) to[out=down,in=down,looseness=2] (0,0) -- (0,0.6);
            \draw[multi] (0.4,-0.6) -- (0.4,0.6) \toplabel{\ell(\mu)};
        \end{tikzpicture}
        \text{ at $q=1$},
    \]
    and similarly for the other diagrams in $\iqB(q,t)$ appearing in \cref{water3,water4}.  It follows that the functor in the statement of the lemma is an isomorphism.
\end{proof}

%---------------------------------
\subsection{Endomorphism algebras}
%---------------------------------

For $r \in \N$, the \emph{iquantum Brauer algebra} $\mathrm{B}_r(q,t)$ is the associative $\kk$-algebra with generators $\sigma_1,\sigma_2,\dotsc,\sigma_{r-1}$ and $e$, subject to the relations
\begin{gather*}
    \sigma_i \sigma_{i+1} \sigma_i = \sigma_{i+1} \sigma_i \sigma_{i+1},\qquad 1 \le i \le r-2,
    \\
    \sigma_i \sigma_j = \sigma_j \sigma_i,\qquad 1 \le i,j \le r-1,\ |i-j| >1,
    \\
    \sigma_i^2 = (q-q^{-1}) \sigma_i + 1,\qquad 1 \le i \le r-1,
    \\
    e^2 = \frac{t-t^{-1}}{q-q^{-1}} e,
    \\
    e \sigma_1 = \sigma_1 e = q e,\quad
    e \sigma_2 e = t e,\quad
    e \sigma_i = \sigma_i e,\qquad i > 2,
    \\
    \sigma_2 \sigma_3 \sigma_1^{-1} \sigma_2^{-1} e_{(2)}
    = e_{(2)}
    = e_{(2)} \sigma_2 \sigma_3 \sigma_1^{-1} \sigma_2^{-1},
    \quad \text{where } e_{(2)} = e \sigma_2 \sigma_3 \sigma_1^{-1} \sigma_2^{-1} e.
\end{gather*}

\begin{rem} \label{terminology}
    The algebra $\mathrm{B}_r(q,t)$ is isomorphic to the algebra $Br_r(t^2,q^2)$ of \cite[\S3.1]{Wen12} via the map
    \[
        \mathrm{B}_r(q,t) \to Br_r(t^2,q^2),\qquad
        e \mapsto qt^{-1} e,\quad
        \sigma_i \mapsto q^{-1} g_i,\quad
        1 \le i \le r-1,
    \]
    as can be checked by direct verification of the defining relations.
    \details{
        The relation \cite[(1.7)]{Wen12} is equivalent to $g_i^2 = (q^2-1)g_i + q^2$.  Replacing $g_i$ by $q \sigma_i$, this becomes $\sigma_i^2 = (q-q^{-1}) \sigma_i + 1$.  The relation \cite[(E1)$'$]{Wen12} is $e^2 = \frac{t^2-1}{q^2-1}$.  After replacing $e$ by $q^{-1}te$, this becomes $e^2 = \frac{t-t^{-1}}{q-q^{-1}} e$.  Similarly, the relations \cite[(E2)$'$]{Wen12} become
        \[
            e \sigma_1 = q e = \sigma_1 e,\quad
            e \sigma_2 e = t e,\quad
            e \sigma_2^{-1} e = t^{-1} e,\quad
            e \sigma_i = \sigma_i e,\qquad i>2,
        \]
        and the relation \cite[(E3)]{Wen12} becomes the relation
        \[
            \sigma_2 \sigma_3 \sigma_1^{-1} \sigma_2^{-1} e_{(2)}
            = e_{(2)}
            = e_{(2)} \sigma_2 \sigma_3 \sigma_1^{-1} \sigma_2^{-1},
            \quad \text{where } e_{(2)} = e \sigma_2 \sigma_3 \sigma_1^{-1} \sigma_2^{-1} e.
        \]
    }
    Similar algebras appeared previously in \cite[Def.~2.1]{Mol03}.  The algebras $\mathrm{B}_r(q,t)$ have been called \emph{$q$-Brauer algebras} in the literature, presumably because some authors view them as deformations of Brauer algebras and the notation used for the deformation parameter is often $q$.  We prefer the term \emph{iquantum Brauer algebras} to avoid possible confusion with Birman--Murakami--Wenzl (BMW) algebras, which are also deformations of Brauer algebras, and also to avoid terminology involving notation.
\end{rem}

\begin{prop}\label{river}
    For all $r \in \N$ and $\lambda \in \word$ with $\ell(\lambda)=r$, we have isomorphisms of algebras
    \[
        \End_{\DS(q,t)}(\lambda)
        \cong \End_{\iqB(q,t)}(\Bobj_r)
        \cong \mathrm{B}_r(q,t).
    \]
\end{prop}

\begin{proof}
    By  \cref{equivthm}, it suffices to prove the result for $\End_{\iqB(q,t)}(\lambda)$.  It is straightforward to verify that the map $\mathrm{B}_r(q,t) \to \End_{\iqB(q,t)}(\Bobj_r)$ given by
    \[
        e \mapsto
        \begin{tikzpicture}[centerzero]
            \draw[multi] (0.6,-0.3) \botlabel{r-2} -- (0.6,0.3);
            \draw (-0.15,0.3) -- (-0.15,0.2) arc(180:360:0.15) -- (0.15,0.3);
            \draw (-0.15,-0.3) -- (-0.15,-0.2) arc(180:0:0.15) -- (0.15,-0.3);
        \end{tikzpicture}
        ,\qquad
        \sigma_i \mapsto
        \begin{tikzpicture}[centerzero]
            \draw[multi] (-0.8,-0.3) \botlabel{i-1} -- (-0.8,0.3);
            \draw[multi] (0.8,-0.3) \botlabel{r-i-1} -- (0.8,0.3);
            \draw (0.15,-0.3) -- (-0.15,0.3);
            \draw[wipe] (-0.15,-0.3) -- (0.15,0.3);
            \draw (-0.15,-0.3) -- (0.15,0.3);
        \end{tikzpicture},
        \qquad 1 \le i \le r-1,
    \]
    satisfies the defining relations of $\mathrm{B}_r(q,t)$.  It then follows from the basis theorems \cite[Th.~3.8]{Wen12} and \cref{iqBbasis} that this map is in an isomorphism.
\end{proof}

%=================================================================
\section{Fullness of the incarnation functors\label{sec:fullness}}
%=================================================================

Our goal in this final section is to show that the functors $\bR_\DS$ and $\bR_\iqB$ are full.  In light of \cref{revolution}, it suffices to prove this for one particular choice of the parameters $\va_i$, $i \in I_\circ$.  Thus, throughout this section we make the choice \cref{oilers}.  This implies that $\tva_i = 1$ for all $i \in \Vset$; see \cref{kowloon}.

Recall that $A \subseteq \C(q)$ is the $\C$-subalgebra of rational functions regular at $q=1$, that $\fm = A(q-1)$ is its unique maximal ideal, and that we identify $\C$ with $A/\fm$.  Let $\UiA$ be the $A$-subalgebra of $\Ui$ generated by the elements \cref{tubeless}, let $\UisA$ be the $A$-subalgebra of $\Uis$ generated by $\UiA$ and $\sigma$, and let $V_A$ be the $A$-submodule of $V^+$ generated by $v_i^+$, $i \in \Vset$.  Since all the coefficients appearing in \cref{panda+,sloth} lie in $A$, we see that $V_A$ is a $\UisA$-module.  We also have the $A$-bilinear map
\[
    \beta_A = \ev^- \circ (\varphi^{-1} \times 1)|_{V_A \times V_A} \colon V_A \times V_A \to A.
\]
The elements $v_i := v_i^+ \otimes 1$, $i \in \Vset$, form a $\C$-basis of $V_\C := V_A \otimes_A \C$.  We have the induced $\C$-bilinear map
\[
    \beta_\C = \beta_A \otimes_A \C \colon V_\C \times V_\C \to \C,\qquad
    \beta_\C(v_i, v_j) = \delta_{i,\phi(j)} (-1)^{i+p(i)}.
\]

Instead of working with supergroups, we will work over the equivalent theory of Harish-Chandra pairs.  We refer the reader to \cite{Gav20} and the references cited therein for a proof of this equivalence.  A brief summary, well suited to the current paper, can also be found in \cite[\S3.6, \S7.5]{SS24}.  Let
\begin{align*}
    G_\textup{red}(\beta_\C) &= \{x \in \Aut(V_\C)_0 : \beta_\C(xv,xw) = \beta_\C(v,w) \text{ for all } v,w \in V\},
    \\
    \fg(\beta_\C) &= \{ x \in \End(V_\C) : \beta_\C(xv,w) = -(-1)^{\bar{x} \bar{v}} \beta_\C(v,xw) \text{ for all } v,w \in V\}.
\end{align*}
The pair $G(\beta_\C) := (G_\textup{red}(\beta_\C), \fg(\beta_\C))$ is a Harish-Chandra pair.  It follows from \cref{pump} that $\beta_\C$ is supersymmetric and nondegenerate.  Thus, we may identify $\fg(\beta_\C)$ with $\osp(m|2n)$ and $G(\beta_\C)$ with the orthosymplectic supergroup $\OSp(m|2n)$.

The action of $\sigma$ on $V^+$, defined in \cref{panda+}, induces an automorphism $\sigma$ of $V_\C$.  When $m>0$, this automorphism has determinant $-1$, and so lies in the connected component of $G_\textup{red}$ not containing the identity.  Thus, as explained in \cite[\S3.6]{SS24}, for any $r,s \in \N$, we have
\begin{multline} \label{coffee}
    \Hom_{\OSp(m|2n)}(V_\C^{\otimes r}, V_\C^{\otimes s})
    = \Hom_{\sigma,\osp(m|2n)}(V_\C^{\otimes r}, V_\C^{\otimes s})
    \\
    := \{f \in \Hom_{\osp(m|2n)}(V_\C^{\otimes r}, V_\C^{\otimes s}) : f(\sigma v)= \sigma f(v) \ \forall\ v \in V_\C^{\otimes r} \}.
\end{multline}

\begin{theo} \label{fullthm}
    The functors $\bR_\DS$ and $\bR_\iqB$ are full.
\end{theo}

\begin{proof}
    To simplify notation, let $\iqB_A = \iqB_A(q,q^{m-2n})$ and $\iqB_{\C(q)} = \iqB_{\C(q)}(q,q^{m-2n})$.  By \cref{iqBincar,ginger}, it suffices to prove that $\bR_\iqB$ is full.  Since all the coefficients appearing in \cref{Vbraid,Vbraidinv,ev-,coev,varphidef} lie in $A$, we have an incarnation functor
    \[
        \bRA_\iqB \colon \iqB_A \to \UisA\tmod.
    \]

    Consider the $\C$-algebra homomorphism
    \begin{equation}
        \rho_\C \colon \UisA \otimes_A \C \to \End_A(V_A) \otimes_A \C
        \cong \End_\C(V_\C)
    \end{equation}
    induced by the $\UisA$-module structure on $V_A$.  It follows from \cref{sloth,panda+} that
    \begin{align} \label{hack1}
        \rho_\C(B_i \otimes 1) &= E_{i+1,i} + \delta_{i>0} E_{i,i+1} + \delta_{i<0} E_{i-1,i+2} + \delta_{i,0} E_{-1,1},& i \in I_\circ,
        \\ \label{hack2}
        \rho_\C(E_{2k-1} \otimes 1) &= E_{2k-1,2k},\quad
        \rho_\C(F_{2k-1} \otimes 1) = E_{2k,2k-1},&
        1-n \le k \le 0,
    \end{align}
    where $E_{i,j}$ is the linear map given by $E_{i,j} v_k^+ = \delta_{j,k} v_i^+$.  A direct computation shows that the elements \cref{hack1,hack2} generate $\fg(\beta_\C)$ as a Lie superalgebra.
    \details{
        Writing vectors in $V_\C$ as column matrices whose entries are the coordinates of the vectors in the basis $v_i$, $i \in \Vset$, we have
        \[
            \beta_\C(u,v) = u^\mathrm{st} J v,
            \qquad \text{where} \quad
            J = \sum_{i \in \Vset} (-1)^i E_{i,\phi(j)}. 
        \]
        Thus, $X \in \fg(\beta_\C)$ if and only if $X^\mathrm{st} J = - JX$.  Equivalently,
        \[
            X =(X_{i,j})_{i,j \in \Vset} \in \fg(\beta_\C) \iff X_{i,j} = - (-1)^{p(i)p(j)+p(i)+i+j} X_{\phi(j),\phi(i)}.
        \]
        It follows that $\fg(\beta_\C)$ is spanned by the elements
        \begin{align*} \tag{$\star 1$}
            &E_{i,j} - (-1)^{i+j} E_{j,i}, & 1 \le i < j \le m,
            \\ \tag{$\star 2$}
            &E_{2k,2l} - E_{2l-1,2k-1}, & 1-n \le k,l \le 0,
            \\ \tag{$\star 3$}
            &E_{2k,2l-1} + E_{2l,2k-1}, & 1-n \le k \le l \le 0,
            \\ \tag{$\star 4$}
            &E_{2k-1,2l} + E_{2l-1,2k}, & 1-n \le k \le l \le 0,
            \\ \tag{$\star 5$}
            &E_{i,2k} - (-1)^i E_{2k-1,i}, & i > 0,\ 1-n \le k \le 0,
            \\ \tag{$\star 6$}
            &E_{i,2k-1} + (-1)^i E_{2k,i}, & i > 0,\ 1-n \le k \le 0.
        \end{align*}
        Let $\ft$ be the Lie superalgebra generated by \cref{hack1,hack2}.  Since all the elements \cref{hack1,hack2} lie in $\fg(\beta_\C)$, we have $\ft \subseteq \fg(\beta_\C)$.  We prove the reverse inclusion by showing that the elements ($\star 1$)--($\star 6$) lie in $\ft$.  In what follows, $1 \le i,j \le m$ and $1-n \le k,l \le 0$.
       
        From \cref{hack1}, we have $E_{i,i+1} + E_{i+1,i} \in \ft$.  Since
        \[
            [E_{i,j} - (-1)^{i+j} E_{j,i}, E_{j,j+1} + E_{j+1,j}]
            = E_{i,j+1} - (-1)^{i+j+1} E_{j+1,i},
        \]
        it follows by induction that all the elements ($\star 1$) lie in $\ft$.
        
        By \cref{hack2}, the $k=l$ cases of ($\star 3$) and ($\star 4$) lie in $\ft$.  Taking $i=2k$ in \cref{hack1}, we see that the $l=k+1$ case of ($\star 4$) lies in $\ft$.  We also have
        \begin{gather*} \tag{$\ast 1$}
            [E_{2k,2k-1}, E_{2k-1,2l} + E_{2l-1,2k}]
            = 2^{\delta_{k,l}} (E_{2k,2l} - E_{2l-1,2k-1}),
            \\ \tag{$\ast 2$}
            [E_{2l,2l-1}, E_{2k-1,2l} + E_{2l-1,2k}]
            = 2^{\delta_{k,l}} (E_{2l,2k} - E_{2k-1,2l-1}),
            \\ \tag{$\ast 3$}
            [E_{2k,2l} - E_{2l-1,2k-1}, E_{2l,2l-1}]
            = E_{2k,2l-1} + E_{2l,2k-1},
            \\ \tag{$\ast 4$}
            [E_{2k-1,2k}, E_{2k,2l} - E_{2l-1,2k-1}]
            = E_{2k-1,2l} + E_{2l-1,2k},
            \\ \tag{$\ast 5$}
            [E_{2k,2l-1} + E_{2l,2k-1}, E_{2l-1,2l+2} + E_{2l+1,2l}]
            = E_{2k,2l+2} - E_{2l+1,2k-1}.
        \end{gather*}
        Note that, for $1-n \le k',l' \le 0$:
        \begin{itemize}
            \item By ($\ast 1$) and ($\ast 2$), if $\ft$ contains ($\star 4$) for $(k,l)=(k',l')$, then it contains ($\star 2$) for $(k,l) = (k',l')$ and $(k,l) = (l',k')$.
            \item By ($\ast 3$), if $\ft$ contains ($\star 2$) for $(k,l)=(k',l')$, then it contains ($\star 3$) for $(k,l)=(k',l')$.
            \item By ($\ast 4$), if $\ft$ contains ($\star 2$) for $(k,l)=(k',l')$, then it contains ($\star 4$) for $(k,l)=(k',l')$.
            \item By ($\ast 5$), if $\ft$ contains ($\star 3$) for $(k,l)=(k',l')$, then it contains ($\star 2$) for $(k,l)=(k',l'+1)$.
        \end{itemize}
        It follows by induction that $\ft$ contains ($\star 2$), ($\star 3$), and ($\star 4$) for all $k$ and $l$ in the given ranges.
        
        From \cref{hack1}, we have $E_{1,0} + E_{-1,1} \in \ft$.  This is the $i=1$, $k=0$ case of ($\star 5$).  Since
        \[
            [E_{i,1} + (-1)^i E_{1,i}, E_{1,0} + E_{-1,1}]
            = E_{i,0} - (-1)^i E_{-1,i},
        \]
        all the elements ($\star 5$) with $k=0$ lie in $\ft$.  Then, noting that
        \[
            [E_{i,2k} - (-1)^i E_{2k-1,i}, E_{2k,2k+2} - E_{2k+1,2k-1}]
            = E_{i,2k+2} - (-1)^i E_{2k+1,i},
        \]
        it follows by induction that the elements ($\star 5$) lie in $\ft$ for all $i$ and $k$ in the given ranges.  Finally, since
        \[
            [E_{i,2k} - (-1)^i E_{2k-1,i}, E_{2k,2k-1}]
            = E_{i,2k-1} + (-1)^i E_{2k,i},
        \]
        we see that the elements ($\star 6$) lie in $\ft$ for all $i$ and $k$ in the given ranges.
    }
    Furthermore, for $r \in \N$, it follows from \cref{comult} that the action of $\UiA[\C] := \UiA \otimes_A \C$ on $V_\C^{\otimes r}$ is determined by the usual comultiplication for Lie superalgebras.  Thus, we have an equivalence of categories
    \[
        \UiA[\C]\tmod \simeq \fg(\beta_\C)\tmod.
    \]
    Together with \cref{coffee}, this implies that we have an equivalence of categories
    \[
        \UisA[\C]\tmod \simeq \OSp(m|2n)\tmod.
    \]

    The above discussion shows that we have a commutative diagram of functors
    \[
        \begin{tikzcd}
            \iqB_A \otimes_A \C \arrow[d,"\cong"'] \arrow[rr,"\bRA_\iqB \otimes 1"] & & \UisA[\C]\tmod \arrow[d,"\simeq"] \\
            \Brauer(m-2n) \arrow[rr,"\bR_\Brauer"] & & \OSp(m|2n)\tmod
        \end{tikzcd}
    \]
    where the left vertical map is the isomorphism of \cref{Guilin} and $\bR_\Brauer$ is the monoidal incarnation functor of \cite[Th.~5.4]{LZ17} determined by
    \[
        \Bcap \mapsto \beta_\C,\qquad \symcross \mapsto \flip.
    \]
    (Here we used the fact that $T_{++} \otimes_A 1 = \flip$, which follows easily from \cref{Vbraid}.)
    By \cite[Thm.~5.6]{LZ17}, the functor $\bR_\Brauer$ is full; see also \cite{ES16}.  Thus, for all $r,s \in \N$, the map
    \begin{equation} \label{wild}
        \bRA_\iqB \otimes 1 \colon \Hom_{\iqB_A}(\Bobj_r,\Bobj_s) \otimes_A \C \to \Hom_{\UisA}(V_A^{\otimes r},V_A^{\otimes s}) \otimes_A \C
    \end{equation}
    is surjective.

    Being a localization of a noetherian ring, the ring $A$ is also Noetherian.  Thus, for all $r,s \in \N$, the space $\Hom_{\UisA}(V_A^{\otimes r},V_A^{\otimes s})$ is finitely generated as an $A$-module, since it is a submodule of the free, finitely-generated $A$-module $\Hom_A(V_A^{\otimes r},V_A^{\otimes s})$.  The space $\Hom_{\iqB_A}(\Bobj_r,\Bobj_s)$ is also finitely generated as an $A$-module by \cref{basis}.  Since \cref{wild} is surjective, it follows from Nakayama's Lemma that
    \[
        \bRA_\iqB \colon \Hom_{\iqB_A}(\Bobj_r,\Bobj_s) \to \Hom_{\UisA}(V_A^{\otimes r},V_A^{\otimes s})
    \]
    is surjective.  Then, extending the ground ring from $A$ to $\C(q)$, we conclude that
    \[
        \bR_\iqB = \bRA_\iqB \otimes 1
        \colon \Hom_{\iqB_{\C(q)}}(\Bobj_r,\Bobj_s) \cong \Hom_{\iqB_A}(\Bobj_r,\Bobj_s) \otimes_A \C(q) \to \Hom_{\Uis}(V^{\otimes r},V^{\otimes s})
    \]
    is also surjective, as desired.
\end{proof} 

%===========
% References
%===========

\bibliographystyle{alphaurl}
\bibliography{DOSQB}

\end{document}